\documentclass[a4paper]{amsart}

\usepackage{amsfonts}
\usepackage{amssymb}
\usepackage{amsmath}
\usepackage{hyperref}
\usepackage{mathrsfs}
\usepackage{centernot}
\usepackage{mathdots}
\usepackage{stmaryrd}
\usepackage[all]{xy}

\usepackage{mathtools}

\usepackage{color}

\newtheorem{thm}{Theorem}[section]
\newtheorem{lem}[thm]{Lemma}
\newtheorem{prp}[thm]{Proposition}
\newtheorem{cor}[thm]{Corollary}

\newtheorem{cnj}[thm]{Conjecture}

\newtheoremstyle{roman} 
    {8.0pt plus 2.0pt minus 4.0pt}                    
    {8.0pt plus 2.0pt minus 4.0pt}                    
    {\normalfont}                
    {}                           
    {\bfseries}                  
    {.}                          
    {5pt plus 1pt minus 1pt}     
    {}  

\theoremstyle{roman}

\newtheorem{example}[thm]{Example}
\newtheorem{remark}[thm]{Remark}

\theoremstyle{plain}

\newcommand{\Step}[1]{\noindent {\it Step #1.}} 

\newcommand{\rem}[1]{}

\newcommand{\N}{\mathbb{N}}

\newcommand{\Z}{\mathbb{Z}}

\newcommand{\HH}{{\mathrm{H}}}

\newcommand{\frakP}{{\mathfrak{P}}}
\newcommand{\frakQ}{{\mathfrak{Q}}}

\newcommand{\fraka}{{\mathfrak{a}}}
\newcommand{\frakb}{{\mathfrak{b}}}
\newcommand{\frakc}{{\mathfrak{c}}}

\newcommand{\frakm}{{\mathfrak{m}}}

\newcommand{\frakp}{{\mathfrak{p}}}
\newcommand{\frakq}{{\mathfrak{q}}}

\newcommand{\frakt}{{\mathfrak{t}}}

\newcommand{\calD}{{\mathcal{D}}}

\newcommand{\calM}{{\mathcal{M}}}

\newcommand{\calP}{{\mathcal{P}}}

\newcommand{\calT}{{\mathcal{T}}}

\newcommand{\catC}{{\mathscr{C}}}

\newcommand{\what}[1]{\widehat{#1}}

\newcommand{\veps}{\varepsilon}
\newcommand{\vphi}{\varphi}

\newcommand{\onto}{\rightarrow\mathrel{\mkern-14mu}\rightarrow} 
\newcommand{\embeds}{\hookrightarrow}

\newcommand{\xonto}[2][]{%
  \xrightarrow[#1]{#2}\mathrel{\mkern-14mu}\rightarrow
}
\newcommand{\xembeds}[2][]{\xhookrightarrow[#1]{#2}}

\newcommand{\suchthat}{\,:\,}
\newcommand{\where}{\,|\,}

\newcommand{\quo}[1]{\overline{#1}}

\DeclareMathOperator{\ann}{ann}

\DeclareMathOperator{\Ann}{Ann} %
\DeclareMathOperator{\Ass}{Ass} %
\DeclareMathOperator{\Br}{Br} %
\DeclareMathOperator{\Cent}{Cent} %
\DeclareMathOperator{\coker}{coker} %
\DeclareMathOperator{\depth}{depth} %
\DeclareMathOperator{\End}{End} %
\DeclareMathOperator{\Ext}{Ext} %
\DeclareMathOperator{\hgt}{hgt} %
\DeclareMathOperator{\Hom}{Hom} %
\DeclareMathOperator{\id}{id} %
\DeclareMathOperator{\im}{im} %
\DeclareMathOperator{\ind}{ind} %
\DeclareMathOperator{\Int}{Int} %

\DeclareMathOperator{\Max}{Max}
\newcommand{\op}{\mathrm{op}} %
\DeclareMathOperator{\rank}{rank}
\DeclareMathOperator{\Spec}{Spec} %
\DeclareMathOperator{\supp}{supp} %
\DeclareMathOperator{\Sym}{Sym} %
\DeclareMathOperator{\Tr}{Tr} %
\newcommand{\nMat}[2]{\mathrm{M}_{#2}(#1)}

\newcommand{\uU}{{\mathbf{U}}}

\newcommand{\et}{\mathrm{\acute{e}t}}

\newcommand{\units}[1]{{#1^\times}}

\usepackage[foot]{amsaddr}  

\usepackage{enumitem} 

\usepackage{tikz}
\usetikzlibrary{arrows,shapes,trees}

\numberwithin{equation}{section}                            

\newcommand{\rMod}[1]{\calM({#1})}  
\newcommand{\rmod}[1]{\calM_{\mathrm{f}}(#1)}  
\newcommand{\rproj}[1]{{\calP({#1})}}              

\newcommand{\Herm}[2][]{{\mathcal{H}}^{#1}(#2)} 
\newcommand{\tHerm}[2][]{{\tilde{\mathcal{H}}}^{#1}(#2)} 

\newcommand{\dd}{{\mathrm{d}}} 

\newcommand{\aGW}[1]{{\mathcal{GW}}^{#1}_+} 
\newcommand{\GW}[1]{{\mathcal{GW}}^{#1}} 

\renewcommand{\Sym}{{\mathcal{S}}}    
\renewcommand{\Cent}{{\mathrm{Z}}}    

\title{On the Grothendieck--Serre Conjecture for Classical Groups}

\date{\today}

\author{Eva Bayer-Fluckiger$^1$}
\email{eva.bayer@epfl.ch}
\author{Uriya A.\ First$^2$}
\email{uriya.first@gmail.com}
\author{Raman Parimala$^3$}
\email{parimala.raman@emory.edu}

\address{$^1$Department of Mathematics, \'Ecole Polytechnique F\'ed\'erale de Lausanne}
\address{$^2$Department of Mathematics, University of Haifa}
\address{$^3$Department of Mathematics, Emory University}



\begin{document}

\begin{abstract}
We prove some new cases of the Grothendieck--Serre conjecture for classical groups. This
is based on a new construction of the Gersten--Witt complex for Witt groups of Azumaya
algebras with involution on regular semilocal rings, with explicit second residue maps; the
complex is shown to be exact when the ring is of dimension $\le 2$ (or $\le 4$, with
additional hypotheses on the algebra with involution). Note that we do not assume 
that the ring contains a field.
\end{abstract}

\maketitle

\section*{Introduction}

Let $R$ be a regular local ring of Krull dimension $d$, and let $G$ be a reductive group scheme over $R$; let $F$ be the field of fractions of $R$.
The Grothendieck--Serre conjecture states the following:

\begin{cnj}[Grothendieck--Serre] 
The restriction map 
$$\HH^1_{\et}(R,G) \to \HH^1_{\et}(F,G)$$ has trivial kernel.
\end{cnj}

\noindent In other words, it is conjectured that a $G$-torsor over $R$ is trivial if it is trivial over $F$. 
 
The first forms of this
conjecture --- in special cases --- appear in the works of Serre 
\cite[Rem.\ p.~31]{Serre_1958_espaces_fibres_algebriques} 
and Grothendieck 
\cite[Rem.~3, pp.~26-27]{Grothendieck_1958_homological_torsion}; it was first stated in the above
form by Grothendieck 
\cite[Rem.~1.11.a]{Grothendieck_68_groupe_de_Brauer_II}. 
It is now proved when $R$ {\it contains a field} by Fedorov and Panin 
\cite{Fedorov_2015_Grothendieck_Serre_conj},
and Panin \cite{Panin_2017__Grothendieck_Serre_R_contains_finite_fld_preprint};
see for instance \cite[\S5]{Panin_2018_Grothendieck_Serre_conj_survey} 
for a detailed description of the history of this conjecture.

\medskip

The aim of the present paper is to prove some new cases of this conjecture, when $R$ {\it does not necessarily
contain a field}. We assume that $2 \in R^{\times}$, and that $G$ is of classical type. More precisely:

\begin{thm}[{Theorems~\ref{TH:GS-and-purity}}]
	\label{TH:main-intro}
	Let $R$ be a regular semilocal domain with fraction field $F$,
	let $(A,\sigma)$ be an Azumaya $R$-algebra with involution, 
	and let $\uU(A,\sigma)$ be the corresponding unitary group scheme. 
	Suppose
	that   one of the following hold:
	\begin{enumerate}[label=(\arabic*)]
		\item $\dim R=2$;
		\item $\dim R\leq 4$ and $\ind A$ is odd. 
	\end{enumerate}
	Then the restriction map $\HH^1_{\et}(R,\uU(A,\sigma))\to\HH^1_{\et}(F,\uU(A,\sigma))$ 
	has trivial kernel.
\end{thm}

The same holds for the neutral connected component of $\uU (A,\sigma)\to \Spec R$.
We also prove a purity result for the Witt group of $(A,\sigma)$  under the same assumptions.

\medskip

The proof uses the {\it Gersten--Witt complex} of Witt groups of hermitian and skew-hermitian forms over $(A,\sigma)$.
In fact, we establish its exactness under the assumptions (1), (2) above (Theorems~\ref{TH:exactness-odd-index} 
and \ref{TH:GW-exact-dim-2}).

\medskip

Recall that in the special case where $A = R$, a conjecture of Pardon \cite[Conjecture~A]{Pardon_1982_Gersten_conjecture} predicts the existence of an exact cochain complex
\[
0\to W (R)\xrightarrow{\dd_{-1}} W (F)\xrightarrow{\dd_0} \bigoplus_{\frakp\in R^{(1)}}W(k(\frakp))
\to \dots \to  \bigoplus_{\frakp\in R^{(d)}}W(k(\frakp))\to 0 
\]
called the {\it Gersten--Witt complex} of $R$; here, $R^{(e)}$
is the set of height-$e$ prime ideals, 
$k(\frakp)$ is the fraction field of $R/\frakp$
and $W(R)$ is the Witt group of $R$.
The map $\dd_{-1}$ is  base-change from $R$ to $F$, and the $\frakp$-component
of $\dd_{0}$ is equivalent to 
the \emph{second residue} map
$\partial:W(F)\to W(k(\frakp))$ (see \cite[p.~209]{Scharlau_1985_quadratic_and_hermitian_forms},
for instance).

Constructions of  a Gersten--Witt complex, applying to any regular domain $R$,  
were given by Pardon  \cite{Pardon_1982_relation_between_Witt_groups_zero_cycles} and Balmer--Walter \cite{Balmer_2002_Gersten_Witt_complex}. 
When $R$ is local,
the exactness has been  established  
when $\dim R\leq 4$ by Balmer--Walter \cite{Balmer_2002_Gersten_Witt_complex} 
and when 
$R$ contains a field by Balmer--Gille--Panin--Walter \cite{Balmer_2002_Gersten_conj_equicharacteristic}
(see also \cite{Balmer_2001_Witt_cohomology}). 
The more
general case of an Azumaya algebra with involution was considered by Gille
\cite{Gille_2007_hermitian_GW_complex_I}, \cite{Gille_2009_hermitian_GW_complex_II}, who also proved its exactness
when $R$ contains a field 
\cite[Theorem~7.7]{Gille_2013_coherent_herm_Witt_grps}. 
See  Balmer--Preeti
\cite[p.~3]{Balmer_2005_shifted_Witt_groups_semilocal}  and 
Jacobson \cite[Theorem~3.8]{Jacobson_2018_cohomological_invariants}
for further positive results on the exactness.

\medskip

We introduce a new  construction of the Gersten-Witt complex 
of   $\veps$-hermitian forms over an Azumaya algebra with involution $(A,\sigma)$ 
(Section~\ref{sec:second-res}). 
In more detail,
the complex that we construct is denoted
$\aGW{A,\sigma,\veps}$
and takes the form 
\[
0\to W_\veps (A)\to  {W}_\veps (A_F)\to\bigoplus_{\frakp\in R^{(1)}}
\tilde{W}_\veps(A (\frakp))
\to \dots \to  \bigoplus_{\frakp\in R^{(d)}}\tilde{W}_\veps(A (\frakp))\to 0,
\]
where 
$\tilde{W}_\veps(A (\frakp))$ denotes the Witt group of $\veps$-hermitian
forms over $ A(\frakp):=A\otimes k(\frakp) $ which are valued in $A\otimes \tilde{k}(\frakp)$,
and 
\[\tilde{k}(\frakp)
=\Ext_{R_\frakp}^{\dim R_\frakp}(k(\frakp),R_\frakp)=
\Hom_{R_\frakp}({\textstyle\bigwedge_{R_\frakp}^{\dim R_\frakp}}(\frakp_\frakp/\frakp_\frakp^2),k(\frakp)) .
\] 
(Since $\tilde{k}(\frakp)$ is a $1$-dimensional $k(\frakp)$-vector space,
the group $\tilde{W}_\veps(A (\frakp)) $ is isomorphic to the ordinary
Witt group of $\veps$-hermitian forms over $A (\frakp)$, but the isomorphism is not canonical.)
The $(-1)$-differential of $\aGW{A,\sigma,\veps}$ is base-change from $R$ to $F$, and
the $(\frakp,\frakq)$-component of each of the other differentials is given as 
a generalized second residue map, see \ref{subsec:second-res}.
This description   is
sufficiently explicit to allow
hands-on computation  of the differentials (Section~\ref{sec:complete-intersection}),
which facilitates many of our arguments.

\medskip

The broad  idea of the proof of Theorem~\ref{TH:main-intro} 
and the exactness of the Gersten--Witt complex
is to use the $8$-periodic exact sequence of Witt groups of 
\cite{First_2022_octagon} (recalled in Section~\ref{sec:octagon}) in order
to induct on the degree of $A$. The basis of the induction is the case $(A,\sigma) = (R,\id)$ established by
Balmer, Preeti and Walter \cite{Balmer_2002_Gersten_Witt_complex},
\cite{Balmer_2005_shifted_Witt_groups_semilocal}.

Several supplementary results
are required to make this approach work.
First, we show that the Gersten--Witt complex
is functorial in $(A,\sigma)$ (after
\emph{d\'evissage}) relative to  \emph{base change} 
and \emph{involution traces} (Theorems~\ref{TH:base-change-is-compatible-with-GW} and
\ref{TH:pi-transfer-is-compatible}),
and hence compatible with the $8$-periodic 
exact sequence  of Witt groups of
\cite{First_2022_octagon} (Theorem~\ref{TH:octagon-is-compatible}).
Second,
we 
prove that the Gersten--Witt complex is exact at the last, i.e.\ $d$-th,
term for all $(A,\sigma,\veps)$ and semilocal $R$, and also at the $0$-term if $\dim R\leq 2$
(Section~\ref{sec:surjectivity}).
Third, we establish a version of the \emph{weak Springer  Theorem on odd-degree
base change}
for Azumaya algebras over semilocal rings (Corollary~\ref{CR:Springer-odd-etale}).

\medskip

The paper is organized as follows:
Section~\ref{sec:preliminaries}
is preliminary and recalls Azumaya algebras with involution,
hermitian categories and some homological algebra facts.
In Section~\ref{sec:second-res},
we give our   construction of the Gersten--Witt complex, and in
Section~\ref{sec:complete-intersection} we explain how the differentials
can be computed ``by hand''.
The construction is used in
Sections~\ref{sec:functoriality} and~\ref{sec:surjectivity}
to establish various 
functoriality
properties
for the Gersten--Witt complex and the surjectivity
of its last differential when $R$ is semilocal.
Section~\ref{sec:octagon} recalls
the $8$-periodic exact sequence of Witt groups of \cite{First_2022_octagon}
and shows that it is compatible with the Gersten--Witt complex.
Applying the $8$-periodic exact sequence to our situation
is the subject
matter of   Sections~\ref{sec:Springer} and~\ref{sec:odd-extension}.
Finally, in Section~\ref{sec:exactness},
we put the machinery of the previous sections together to prove
new cases of the Grothendieck--Serre  conjecture
and the exactness of the Gersten--Witt complex.

\subsection*{Acknowledgements}
We are grateful to Stefan Gille and Paul
Balmer for several useful 
correspondences. 
We also thank the anonymous referees for their comments and suggestions.

\subsection*{Notation}

	A ring means a commutative (unital) ring.
	Algebras are unital and associative, but not necessarily commutative.
	\emph{We assume   that $2$ is invertible in all rings and algebras.}
	
	Unless otherwise indicated, $R$ denotes a   ring.
	By default, unadorned tensors, $\Hom$-groups  and $\Ext$-groups are   taken over $R$.
	We will also suppress the base ring when tensoring or taking $\Hom$-groups
	over a localization of $R$. 
	An $R$-ring means a commutative $R$-algebra. An invertible $R$-module
	means a rank-$1$ projective $R$-module.
	
	Given $\frakp\in \Spec R$, we let $k(\frakp)$ denote
	the fraction field of $R/\frakp$.		
	The set of primes  in $\Spec R$  of codimension (i.e.\ height) $e$ is denoted
	$R^{(e)}$. 
	The  support of an $R$-module $M$ is denoted $\supp_R M$. If all   primes
	in $\supp_R M$ are of codimension at least $e$,
	then $M$ is said to be supported, or $R$-supported, in codimension $ e$.
	An element $r\in R$ is said to act faithfully on $M$ if the map $x\mapsto xr:M\to M$
	is injective; the set of all such elements is a   multiplicative subset of $R$.
	
	Let $S$ be an $R$-ring.
	Given  $R$-modules $M$, $N$ and $f\in\Hom(M,N)$,
	we write $M_S:=M\otimes S$ and $f_S:=f\otimes \id_S\in\Hom_S(M_S,N_S)$.
	When $S=k(\frakp)$ for $\frakp\in \Spec R$,
	we write $M(\frakp)=M_{k(\frakp)}$ and $f(\frakp)=f_{k(\frakp)}$.
	In addition, we
	let $m(\frakp)$
	denote the image  of $m\in M$ in $M(\frakp)$.
	
	For $M, N, f$ as before, we write $M^\vee =\Hom(M,R)$ and $f^\vee=\Hom(f,R)$.

	Given an $R$-algebra $A$, its units and its center 
	are denoted $\units{A}$ and $\Cent(A)$, 
	respectively. Given $a\in \units{A}$, we write $\Int(a)$
	for the inner automorphism $x\mapsto axa^{-1}:A\to A$.
	The category of finite (i.e.\ finitely
	generated) projective right $A$-modules is denoted $\rproj{A}$.
	The category of all, resp.\ finite, right $A$-modules is denoted $\rMod{A}$,
	resp.\ $\rmod{A}$. Unless otherwise indicated, $\Hom_A$ and $\Ext^*_A$
	indicate $\Hom$-groups and $\Ext$-groups taken relative to the category
	of right $A$-modules.
	If $R$ is noetherian, $C$ is a multiplicative subset of $R$ and $M,N\in\rmod{A}$, 
	then we shall freely identify $\Ext^*_A(M,N)C^{-1}$ with
	$\Ext^*_{AC^{-1}}(MC^{-1},NC^{-1})$, see \cite[Theorem~2.39]{Reiner_2003_maximal_orders_reprint}.
	
	In situations when an abelian group $M$ can be regarded as a module over multiple $R$-algebras,
	we shall sometimes write $M_A$ to denote ``$M$, viewed as a right $A$-module''.

	An $R$-algebra with involution means an $R$-algebra with an $R$-linear involution.
	Given an $R$-involution $\sigma:A\to A$  and $\veps\in \mu_2(R):=\{r\in R\suchthat
	r^2=1\}$, we let
	$\Sym_{\veps}(A,\sigma)=\{a\in A\suchthat a=\veps a^{\sigma}\}$.
	
\section{Preliminaries}	
\label{sec:preliminaries}

\subsection{Azumaya Algebras with Involution}
\label{subsec:Azumaya}

Given  an Azumaya $R$-algebra $A$,
the degree and the index  of $A$
are denoted $\deg A$ and $\ind A$, respectively.
When $R$ is connected, they can be regarded as integers.
The Brauer class of $A$ is denoted $[A]$, and the Brauer group of $R$
is denoted $\Br R$.

A finite \'etale $R$-algebra $S$ of rank $2$ is called
\emph{quadratic \'etale}. In this case, $S$ admits a unique $R$-involution
$\theta$ such that $R=S^{\{\theta\}}:=\{s\in S\suchthat s^\theta=s\}$; it is given
by $x^\theta=\Tr_{S/R}(x)-x$ and called the \emph{standard
$R$-involution} of $S$.

In the sequel, we shall often consider $R$-algebras $A$ such that $A$ is Azumaya over $\Cent(A)$
and $\Cent(A)$ is finite \'etale over $R$.
These turn out to be precisely
the $R$-algebras which are separable and projective over $R$
\cite[Proposition~1.1]{First_2022_octagon}. We therefore
call such algebras \emph{separable projective}.
A module over a separable projective $R$-algebra
is projective over the algebra if and only if it is projective over
$R$ \cite[Proposition~2.14]{Saltman_1999_lectures_on_div_alg}.

\medskip

Let $(A,\sigma)$ be an $R$-algebra with ($R$-linear) involution.
Following \cite[\S1.4]{First_2022_octagon} and \cite[\S4]{Ojaguren_2001_rationally_trivial_hermitian_spaces},
we say that $(A,\sigma)$ is Azumaya over $R$
if $A$ is a separable projective $R$-algebra
and $r\mapsto 1_A\cdot r$
defines an isomorphism
$R\to \Cent(A)^{\{\sigma\}}=\{s\in\Cent(A)\suchthat s^\sigma=s\}$.
In this case, provided $R$ is connected, $\Cent(A)$ is either $R$ or a quadratic \'etale 
$R$-algebra.

If $(A,\sigma)$ is a separable projective $R$-algebra with involution,
then $(A,\sigma)$ is Azumaya over $\Cent(A)^{\{\sigma\}}$
and $\Cent(A)^{\{\sigma\}}$ is finite \'etale over $R$  
\cite[Example~1.20]{First_2022_octagon}.
Thus, if $S$ is a connected finite \'etale
$R$-algebra and $\sigma:S\to S$ is a nontrivial $R$-involution,
then $S$ is quadratic \'etale over $S^{\{\sigma\}}$,
the standard $S^{\{\sigma\}}$-involution of $S$ is $\sigma$
and  $S^{\{\sigma\}}$
is finite \'etale over $R$.

\medskip

Suppose that $R$ is connected and let $(A,\sigma)$ be an Azumaya
$R$-algebra with involution. If $\Cent(A)=R$,
then the type of $\sigma$ can be either \emph{orthogonal} or \emph{symplectic},
and if $\Cent(A)\neq R$,  then $\sigma$ is said to be of \emph{unitary} type;
see \cite[\S1.4]{First_2022_octagon}.  
Symplectic
involutions can exist only when $\deg A$ is even.
Given $\veps\in \mu_2(R)=\{\pm 1\}$, we define the type $(\sigma,\veps)$ to be
the type of $\sigma$ if $\veps=1$ and the opposite   type of $\sigma$
if $\veps=-1$; here, orthogonal and symplectic types are opposite to each
other and the unitary type is self-opposite.

These definitions extend  to the non-connected case by considering
the types as functions on $\Spec R$ valued in $\{\text{orthogonal},\text{symplectic},\text{unitary}\}$;
see \cite[\S1.4]{First_2022_octagon}.

\subsection{Hermitian Forms}
\label{subsec:herm-forms}

We proceed with recalling some facts about hermitian spaces and Witt groups.
See \cite{Scharlau_1985_quadratic_and_hermitian_forms}, \cite{Knus_1991_quadratic_hermitian_forms}
or \cite{Balmer_2005_Witt_groups}
for an extensive discussion.

Let $(A,\sigma)$ be an $R$-algebra with involution and let $Z$ be an $(A,A)$-bimodule
admitting a map $\theta:Z\to Z$ satisfying $x^{\theta \theta}=x$, $(axb)^\theta=b^\sigma x^\theta a^\sigma$
and $rx=xr$ for all $x\in Z$, $a,b\in A$ and $r\in R$,
e.g.\ $(Z,\theta)=(A,\sigma)$. Let $\veps\in \mu_2(R)$.
An \emph{$(Z,\theta)$-valued $\veps$-hermitian space} over $(A,\sigma)$ is a
pair $(V,f)$ consisting of a right $A$-module $V$ and a biadditive map $f:V\times V\to Z$
satisfying $f(xa,yb)=a^\sigma f(x,y)b$ and $f(x,y)=\veps f(y,x)^\theta$
for all $x,y\in V$, $a,b\in A$. 
For example, when $(Z,\theta)=(A,\sigma)$ and $V\in \rproj{A}$, this is just
an $\veps$-hermitian space over $(A,\sigma)$ in the usual sense.
Isometries of and orthogonal sums of $(Z,\theta)$-valued  $\veps$-hermitian spaces
over $(A,\sigma)$ are defined in the usual way.

Let $\Hom^\sigma_A(V,Z)$   denote $\Hom_A(V,Z)$ endowed with the \emph{right}
$A$-module structure given by $(\phi a)x=a^\sigma (\phi x)$ ($\phi\in \Hom_A(V,Z)$, $x\in V$, $a\in A$). 
Then $\Hom^\sigma_A(-,Z):\rMod{A}\to \rMod{A}$ is a contravariant functor,
and there
is a natural transformation $\omega_V : V\to \Hom^\sigma_A(\Hom^\sigma_A(V,Z),Z)$
given by $(\omega x)\phi=(\phi x)^\theta$.
Every $(Z,\theta)$-valued $\veps$-hermitian space $(V,f)$ gives
rise to  a right $A$-module map $x\mapsto f(x,-):V\to \Hom^\sigma_A(V,Z)$.
We say   that $(V,f)$ is \emph{unimodular} if this map is bijective.

\medskip

Let $\catC$ be a   full additive subcategory of $\rmod{A}$, e.g.\ $\rproj{A}$.
Then $\catC$ inherits a exact category structure from $\rMod{A}$. Suppose that  
\begin{enumerate}[label=(H\arabic*)]
\item \label{item:herm-exact-cond} $\Hom^\sigma_A(-,Z)$ restricts to an exact functor $\catC \to \catC$;
\item \label{item:herm-reflexive-cond} $\omega_V$ is an isomorphism for all $V\in \catC$.
\end{enumerate}
Then $(\catC,\Hom^\sigma_A(-,Z),\omega)$ is an exact hermitian category
in the sense of \cite[\S1.1]{Balmer_2005_Witt_groups}, and so we may
consider its Witt group of $\veps$-hermitian forms, denoted
$W_\veps(\catC)$. 
In more detail, let $\Herm[\veps]{\catC}$ 
denote the category whose objects are unimodular   $(Z,\theta)$-valued $\veps$-hermitian spaces with underlying
module in $\catC$ and whose morphisms are isometries.
Then   $W_\veps(\catC)$ is the quotient of the   Grothendieck group of 
$\Herm[\veps]{\catC}$ (relative to   orthogonal sum) by    the subgroup
generated by   metablic $\veps$-hermitian spaces. Here, $(V,f)\in\Herm[\veps]{\catC}$ is \emph{metabolic} 
(relative to $\catC$)
if there is submodule $M\subseteq V$ such that both $M$ and $V/M$
are in $\catC$ and $M=M^\perp:=\{x\in V\suchthat f(M,x)=0\}$. 
The   class represented by $(V,f)\in\Herm[\veps]{\catC}$ in $W_\veps(\catC)$ is 
the 
\emph{Witt class} of $(V,f)$; it is denoted  by $[V,f]$ or $[f]$.

For the most part, we shall be concerned with the following two  examples where $\catC=\rproj{A}$.
A few examples with $\catC\neq \rproj{A}$ will be encountered in later sections.

\begin{example}\label{EX:herm-ring-with-inv}
	When $(Z,\theta)=(A,\sigma)$ and $\catC=\rproj{A}$, the category
	$\Herm[\veps]{\catC}$ is just the category of unimodular
	$\veps$-hermitian spaces over $(A,\sigma)$ in the usual sense, e.g.\ \cite[\S2]{First_2022_octagon}.
	In this case, we write $\Herm[\veps]{A,\sigma}$ for $\Herm[\veps]{\catC}$ and 
	$W_\veps(A,\sigma)$ for $W_\veps(\catC)$. 
\end{example}

\begin{example}\label{EX:herm-line-bundle}
	Let $M$ be an invertible $R$-module and take
	$Z= A\otimes M$, $\theta= \sigma\otimes \id_M$ and 
	$\catC=\rproj{A}$. 	
	Conditions \ref{item:herm-exact-cond} and \ref{item:herm-reflexive-cond} 
	hold in this situation  because they hold
	after localizing at every $\frakp\in \Spec R$.
	The corresponding category and Witt group of $\veps$-hermitian forms
	are denoted $\Herm[\veps]{A,\sigma;M}$ and $W_\veps(A,\sigma;M)$, respectively.
\end{example}

Continuing Example~\ref{EX:herm-line-bundle}, let $S$ be an $R$-ring.
For every $(V,f)\in \Herm[\veps]{A,\sigma;M}$, let $f_S:V_S\times V_S\to Z_S $
be the biadditive map determined by $f_S(x\otimes s,y\otimes t)=f(x,y)\otimes st$
($x,y\in V$, $s,t\in S$). It is easy to check that $(V_S,f_S)\in\Herm[\veps]{A_S,\sigma_S;M_S}$.
The assignment $(V,f)\mapsto (V_S,f_S)$ extends to a  functor
$\Herm[\veps]{A,\sigma;M}\to \Herm[\veps]{A_S,\sigma_S;M_S}$ by mapping every
isometry $\phi$ to $\phi_S$. This in turn induces a group homomorphism
$W_\veps(A,\sigma;M)\to W_\veps(A_S,\sigma_S;M_S)$.
When $S=R_\frakp$ for $\frakp\in\Spec R$, we shall write $f_S$ as $f_\frakp$.

\subsection{Some Homological Algebra}
\label{subsec:Koszul-complex}

We record several facts about Cohen--Macaulay modules and the Koszul complex that will be needed in the sequel.
See \cite{Bruns_1993_cohen_macaulay_rings} for 
an extensive discussion.
All rings in this subsection are noetherian.

\begin{prp}\label{PR:CM-properties-I}
	Let $R$ be a Cohen-Macaulay local ring   with dualizing module
	$E$,
	let $M\in\rmod{R}$,
	and let $C$ be the set of elements in $R$ acting faithfully on $M$.
	\begin{enumerate}[label=(\roman*)]
	\item
	If $M$ is supported in codimension $e$, 
	then $\Ext^i_R(M,E)=0$ for all $i<e$, and $\Ext^e_R(M,E)$ is $C$-torsion-free.
	\item If $M$ is Cohen-Macaulay of dimension $e$,
	then $\Ext^i_R(M,E)=0$ for all $i\neq e$ and $\Ext^e_R(M,E)$
	is Cohen-Macualay of dimension $e$.
	\end{enumerate}
\end{prp}

\begin{proof}
	Part (ii) is \cite[Theorem~3.3.10(c)]{Bruns_1993_cohen_macaulay_rings},
	and
	the first assertion of (i) follows from
	\cite[Proposition~1.2.10, Theorem~2.1.2(b)]{Bruns_1993_cohen_macaulay_rings} 
	(applied with $M=E$ and $I=\Ann_R M$).
	Finally, given   $c\in C$, the short exact sequence $M\xembeds{c} M\onto M/cM$ induces
	a short exact sequence $\Ext^e_R(M/cM,E)\to \Ext^e_R(M,E)\xrightarrow{c} \Ext^e_R(M,E)$.
	Since $M/cM$ is supported in codimension $e+1$, we have
	$\Ext^e_R(M/cM,E)=0$, so $c$ acts faithfully on $\Ext^e_R(M,E)$. 
%
%
%
\end{proof}

\begin{lem}\label{LM:zero-divisors}
	Let $R$ be a Cohen--Macaulay ring. Then the set of zero divisors
	in $R$ is $\bigcup_{\frakp\in R^{(0)}} \frakp$. Moreover,
	every ideal of $R$ not contained in a minimal prime contains a regular element.
\end{lem}

\begin{proof}
	See
	\cite[Theorems~6.1(ii), 17.3(i)]{Matsumura_1989_commutative_ring_theory} for the first statement.	
	The second statement follows from the first and the Prime Avoidance Lemma
	\cite[Exercise~16.8]{Matsumura_1989_commutative_ring_theory}.
\end{proof}

\begin{prp}\label{PR:CM-properties-one-dim}
	Let $R$ be a  $1$-dimensional Cohen--Macaulay  local ring with maximal ideal
	$\frakm$,   let $C$ denote the set of non-zero divisors in $R$, and let
	$ M $ be a nonzero finite $R$-module. Then the following are equivalent:
	\begin{enumerate}[label={\rm(\alph*)}]
		\item  $M$ is $C$-torsion-free;
		\item $M$ is Cohen-Macaulay of dimension $1$;
		\item $\frakm\notin \Ass_R M$.
	\end{enumerate}	 
\end{prp}

\begin{proof}
	(a)$\implies$(b): By Lemma~\ref{LM:zero-divisors}, there
	exists $r\in  
	\frakm \cap C$. 
	This   $r$ acts faithfully on $M$, so $\depth_R M\geq 1$. Since $\dim R=1$,
	this means that $M$ is Cohen-Macaulay of dimension $1$.

	(b)$\implies$(c) is 
	\cite[\href{https://stacks.math.columbia.edu/tag/0BUS}{Tag 0BUS}]{DeJong_2018_stacks_project}
	or \cite[Theorem~2.12(a)]{Bruns_1993_cohen_macaulay_rings}.

	(c)$\implies$(a): Since $\dim R=1$ and $R$ is local,
	(c) implies $\Ass_R M\subseteq R^{(0)}$,
	so by \cite[\href{https://stacks.math.columbia.edu/tag/00LD}{Tag 00LD}]{DeJong_2018_stacks_project},
	elements of $R-(\cup_{\frakp\in R^{(0)}}\frakp)$ act faithfully on $M$.
	Since $C\subseteq R-(\cup_{\frakp\in R^{(0)}} \frakp)$, (a) follows.	
\end{proof}

\begin{prp}\label{PR:CM-properties-divisor}
	Let $R$ be a     Cohen--Macaulay local ring 
	with dualizing module $E$, let $I\subseteq J$ be    ideals of   $R$
	such that $R/I$ is   
	of dimension $e$,
	and 
	let $C$ denote the set of elements of $R$ acting faithfully
	on $R/I$. Assume $C\cap J\neq \emptyset$, so that
	$JC^{-1}=RC^{-1}$ and $\Ext^e_R(J/I,E)C^{-1}
	=\Ext^e_R(R/I,E)C^{-1}$.
	Then the natural maps from
	$\Ext^e_R(R/I,E)$ and $\Ext^e_R(J/I,E)$
	to $\Ext^e_{R }(R /I ,E)C^{-1}$
	are injective and, upon regarding the former as submodules of the latter,
	we have
	\[
	\Ext^e_R(J/I,E)=
	\{x\in \Ext^e_{R }(R /I ,E)C^{-1}\suchthat xJ\subseteq \Ext^e_R(R/I,E)\}.
	\]
\end{prp}

\begin{proof}	
	Let $\{r_1,\dots,r_n\}$ be a generating set for $J$. 
	Define $f:(R/I)^n\to J /I$   by $f(x_1,\dots,x_n)=r_1x_1+\dots+r_nx_n$
	and let $K=\ker f$. The short exact sequence $K\embeds (R/I)^n\onto J /I$
	and its localization at $C$ give rise to a commutative diagram with exact rows:
	\[
	\xymatrix{
	\Ext_{R }^{e}(R /I ,E)C^{-1} \ar[r]^{f' } & 
	(\Ext_{R }^{e}(R /I ,E)C^{-1})^n \ar[r] & 
	\Ext_{R}^{e}(K,E)C^{-1} \\
	\Ext_R^{e}(J/I,E)\ar[u]\ar[r] &
	\Ext_{R }^{e}(R/I,E )^n \ar[u]\ar[r]  &
	\Ext_{R }^{e}(K,E )\ar[u]
	}
	\]
	The vertical arrows are given by localization at  $C$,
	and are injective by Proposition~\ref{PR:CM-properties-I}(i).
	The same proposition also tells us that $\Ext_{R}^{e-1}(K,E)=0$ if $e>0$, so
	$f'$ is also injective. 
	Since $f'$
	is given by $f'(x)=(xr_1,\dots,xr_n)$, the proposition is a consequence
	of a simple diagram chase.
\end{proof}

\begin{prp}\label{PR:CM-Azumaya-properties}
	Let $R$ be a $d$-dimensional Cohen--Macaulay local ring with dualizing module
	$E$, let
	$A$ be a separable projective $R$-algebra, 
	and let $M\in\rmod{A}$.
	Suppose that $M$ is a Cohen-Macaulay $R$-module of dimension $e$.
	Then:
	\begin{enumerate}[label=(\roman*)]
	\item $\Ext^i_A(M,  A\otimes E)=0$ for all $i\neq d-e$ and $\Ext^{d-e}_A(M,  A\otimes E)$ is a 
	Cohen-Macaulay $R$-module of dimension $e$.
	\item If $e=d$, then the natural map $M\to \Hom'_A(\Hom_A(M,  A\otimes E),  A\otimes E)$
	is an isomorphism. Here, $\Hom'_A$ denotes the set of \emph{left} $A$-module homomorphisms.
	\end{enumerate}	 
\end{prp}

\begin{proof}
	We prove (i) and (ii) together.
	Let $R'$ denote a strict henselization of $R$.
	Then $R'$ is a Cohen--Macaulay local ring
	which is faithfully flat over $R$ and has the same dimension as $R$
	\cite[Tags %
	\href{https://stacks.math.columbia.edu/tag/06LK}{06LK}, 
	\href{https://stacks.math.columbia.edu/tag/07QM}{07QM}, 
	\href{https://stacks.math.columbia.edu/tag/06LM}{06LM}
	]{DeJong_2018_stacks_project}.
	By virtue
	of \cite[Theorem~3.3.14(a)]{Bruns_1993_cohen_macaulay_rings} and
	\cite[Tag \href{https://stacks.math.columbia.edu/tag/04GP}{04GP}]{DeJong_2018_stacks_project},
	$E_{R'}$ is a canonical module of $R'$,
	and by \cite[Tag \href{https://stacks.math.columbia.edu/tag/0338}{0338}]{DeJong_2018_stacks_project},
	$M_{R'}$ is a Cohen--Macualay $R'$-module of dimension $e$.
	Moreover, by  \cite[Theorem~2.39]{Reiner_2003_maximal_orders_reprint},
	there is a canonical isomorphism $\Ext^*_A(M,E)\otimes R'\cong 
	\Ext^*_{A\otimes R'}(M_{R'},E_{R'})$.
	We may therefore tensor $M$, $A$ and $E$ with $R'$
	and   assume $R=R'$. 
	In this case,
	$A=\prod_{i=1}^t \nMat{R}{n_i}$ for some $t\geq 0$ and $n_1,\dots,n_t \in\N$.
	Working at each factor separately and using Morita equivalence,
	we may assume $A=R$. This case is contained in
	\cite[Theorem~3.3.10]{Bruns_1993_cohen_macaulay_rings}. 
\end{proof}

\begin{cor}\label{CR:CM-Azumaya-properties}
	Let $R$ be a Cohen-Macaulay ring dimension $e$
	with dualizing module $E $ and let $(A,\sigma)$ be a separable projective $R$-algebra
	with involution.
	Set $Z=  A \otimes E $, $\theta= \sigma \otimes \id_{E }$ and let $\catC$
	denote the full subcategory of $\rmod{A}$ consisting of   
	$A$-modules which are Cohen-Macaulay $R$-modules of dimension $e$.
	Then $\catC$ and $(Z,\theta)$ satisfy conditions \ref{item:herm-exact-cond} and \ref{item:herm-reflexive-cond} 
	of \ref{subsec:herm-forms}.
\end{cor}

\begin{proof}
	It is enough to check this   
	after localizing at every
	$\frakp$, so we may assume that $R$ is local.
	Observe that for every $V\in\rMod{A}$, there is a natural
	isomorphism $\psi \mapsto \theta\circ \psi:\Hom_A^\sigma(\Hom_A^\sigma(V,Z),Z)\to
	\Hom_A'(\Hom_A(V,Z),Z)$ and under this isomorphism the map $\omega_V$
	is just the natural map $V\to \Hom_A'(\Hom_A(V,Z),Z)$.
	With this at hand, \ref{item:herm-exact-cond} and \ref{item:herm-reflexive-cond}  follow from parts (i) and (ii) of 
	Proposition~\ref{PR:CM-Azumaya-properties}, respectively.
\end{proof}

We finish  with recalling the construction of  the Koszul complex and several of its properties.

Let $E\in \rmod{R}$ and let $s\in E^\vee = \Hom_R(E, R)$.
The \emph{Koszul complex} of $(E,s)$ is the chain complex $K=K(E,s)=(K_i, d_i)_{i\in \Z}$,
where
$K_i=\bigwedge^i_R E$ for $i\geq 0$, $K_i=0$ for $i <0$, and $d_i$ is given by
\[
d_i(e_1\wedge\dots\wedge e_i)=\sum_{1\leq j\leq i} (-1)^{j+1} s(e_j)\cdot e_1\wedge\dots\wedge\hat{e}_j\wedge \dots \wedge e_i
\]
for all $e_1,\dots,e_i\in E$
($\hat{e}_j$ means omitting $e_j$).

\begin{prp}\label{PR:Koszul-properties} 
	Let $r_1,\dots,r_n$ be a regular sequence in $R$,
	let $\fraka=\sum_i r_i R$, and let
	$K=K(E,s)$ where $E$ is a free $R$-module with basis $x_1,\dots,x_n$
	and $s$ is determined by $s(x_i)=r_i$ for all $i$.
	Then:
	\begin{enumerate}[label=(\roman*)]
		\item $\HH^0(K)=R/\fraka$ and $\HH^i(K)=0$ for all $i\neq 0$.
		\item $\fraka/\fraka^2$ is a free $R/\fraka$-module
		with basis $\{r_1+\fraka^2,\dots,r_n+\fraka^2\}$. 
		\item 
		There is a canonical   isomorphism
		$\Ext^n_R(R/\fraka,R)\cong \Hom({\textstyle\bigwedge^n_R}(\fraka/\fraka^2),R/\fraka)$.
	\end{enumerate}
\end{prp}

\begin{proof}
(i) is 
\cite[Corollary~1.6.14]{Bruns_1993_cohen_macaulay_rings}
and (ii) is 
\cite[Theorem~1.1.8]{Bruns_1993_cohen_macaulay_rings}.

(iii) This is well-known, but we recall the proof for later reference. 
By (i),   $\Ext^n_R(R/\fraka,R)$ is 
the cokernel
of $d_n^\vee: (\bigwedge^{n-1} E)^\vee\to (\bigwedge^n E)^\vee$.
Let $\theta\in  (\bigwedge^n E )^\vee$ denote
the element mapping $x_1\wedge\dots\wedge x_n$ to $1$.
One readily checks that the image of $d_n^\vee$ is $\fraka \theta$.
Thus, $\coker(d_n^\vee)\cong  (\bigwedge^n E)^\vee/\fraka (\bigwedge^n E)^\vee=
\Hom(\bigwedge^n E,R)\otimes (R/\fraka)\cong  \Hom(\bigwedge^n (E/\fraka E),R/\fraka)$.
By (ii), we may identify $E/\fraka E$ with $\fraka/\fraka^2$ via $s:E\to R$,
so $\Ext^n_R(R/\fraka,R)\cong \Hom(\bigwedge^n (\fraka/\fraka^2),R/\fraka)$.
\end{proof}

\section[The Gersten--Witt Complex]{The Gersten--Witt Complex via Second Residue Maps}
\label{sec:second-res}

Henceforth, $R$ denotes a $d$-dimensional regular ring, $(A,\sigma)$
is a separable projective $R$-algebra with involution, e.g.\ an Azumaya $R$-algebra with involution,
and $\veps\in \mu_2(R)$.

In this section, we give a new construction of the Gersten--Witt complex
of $\veps$-hermitian forms over $(A,\sigma)$, denoted  $\GW{A/R,\sigma,\veps}$ or $\GW{A,\sigma,\veps}$, in which
the differentials are defined using   generalized
second residue maps. 

\medskip

Given $e\geq 0$ and $\frakp\in R^{(e)}$, write
\[\tilde{k}(\frakp):= \Ext_{R_\frakp}^{\dim R_\frakp}(k(\frakp),R_{\frakp})
=\Hom(\textstyle\bigwedge^{\dim R_\frakp}_{R_\frakp} (\frakp_\frakp/\frakp_\frakp^2),k(\frakp))
\]
(see~Proposition~\ref{PR:Koszul-properties}(iii) for second equality), and, using
the notation of Example~\ref{EX:herm-line-bundle}, set
\[
\tilde{W}_\veps(A(\frakp))=W_\veps(A(\frakp),\sigma(\frakp);\tilde{k}(\frakp)).
\]
Since $\tilde{k}(\frakp)\cong k(\frakp)$  as
$R$-modules, 
we actually have $\tilde{W}_\veps(A(\frakp))\cong W_\veps(A(\frakp),\sigma(\frakp))$,
but this isomorphism is not canonical unless $e=0$. (In contrast, for $e=0$,  we have
$\tilde{k}(\frakp)=
\Hom_{R_\frakp}(R_\frakp,k(\frakp))=\End_{k(\frakp)}(k(\frakp))=k(\frakp)$,
and so $\tilde{W}_\veps(A(\frakp))= W_\veps(A(\frakp),\sigma(\frakp))$.)

The cochain complex $\GW{A,\sigma,\veps}$ will  take the form
\[
0 
\to 
\bigoplus_{\frakp\in R^{(0)}}\tilde{W}_\veps(A(\frakp) )
\xrightarrow{\dd_0} \bigoplus_{\frakp\in R^{(1)}}\tilde{W}_\veps(A(\frakp) ) 
\xrightarrow{\dd_1} \dots
\xrightarrow{\dd_{d-1}} \bigoplus_{\frakp\in R^{(d)}}\tilde{W}_\veps(A(\frakp) )
\to 0.
\]
We will sometimes augment $\GW{A,\sigma,\veps}$ with the additional map
\[\dd_{-1}:W_\veps(A,\sigma)\to \bigoplus_{\frakp\in R^{(0)}}\tilde{W}_\veps(A(\frakp) ),\]
given by localizing at $\frakp$ at the $\frakp$-component.
The resulting complex is the   \emph{augmented Gersten--Witt complex}
of $(A,\sigma,\veps)$, denoted  $\aGW{A/R,\sigma,\veps}$  or $\aGW{A,\sigma,\veps}$.

The true challenge in defining a Gersten--Witt complex is defining its differentials,
and this shall be our concern in the remainder of this section.
We shall first introduce a generalization of the classical second residue map
for Witt groups
(see \cite[p.~209]{Scharlau_1985_quadratic_and_hermitian_forms}),
and then use it to define the differentials of $\aGW{A ,\sigma,\veps}$.
Section~\ref{sec:complete-intersection} will   provide a    method for explicit computation of the differentials.

\begin{remark}
	In an unpublished work, Schmid \cite{Schmid_1997_PhD}
	defined a Gersten--Witt complex when $A=R$ and $R$ is
	a localization of a finite-type ring over a field.
	The differentials in Schmid's Gersten--Witt complex involve classical second residue maps, but also
	normalizations and traces, which are not present in our construction.
\end{remark}

\subsection{The Residue Map}
\label{subsec:second-res}

Suppose that $R$ is regular local of dimension $e+1$, let $\frakq$ denote the maximal ideal of $R$,
and let $I$ be an ideal of $R$ such that $S:=R/I$ is a $1$-dimensional Cohen-Macaulay ring, e.g., 
any prime ideal $\frakp\in R^{(e)}$.
Let $C$ denote the set of elements of $R$ acting faithfully on $R/I$ and write
\begin{align*}
k(I)&= SC^{-1} & \tilde{k}(I)&= \Ext^e_{RC^{-1}}(SC^{-1},RC^{-1})\\
A(I)&=  A \otimes k(I) & \tilde{A}(I)&=A\otimes \tilde{k}(I)\\
\sigma(I) &= \sigma\otimes  \id_{k(I)}  & \tilde{\sigma}(I) &=  \sigma \otimes \id_{\tilde{k}(I)}.
\end{align*}
This agrees with our previous notation when $I=\frakp$ for some $\frakp\in R^{(e)}$.
Note also that $k(I)$ is a $0$-dimensional Cohen-Macaulay ring with dualizing modules
$\tilde{k}(I)$  \cite[Theorems~3.3.7(b), 3.3.5]{Bruns_1993_cohen_macaulay_rings}.

We shall be concerned with $(\tilde{A}(I),\tilde{\sigma}(I))$-valued
$\veps$-hermitian forms over $(A(I),\sigma(I))$  in the sense of \ref{subsec:herm-forms}.
By Corollary~\ref{CR:CM-Azumaya-properties},
conditions \ref{item:herm-exact-cond} and \ref{item:herm-reflexive-cond} 
of \ref{subsec:herm-forms} hold for the category $\rmod{A(I)}$, so we may
define the associated category of unimodular $\veps$-hermitian forms, denoted $\tHerm{A(I)}$, and the associated
Witt group, denoted $\tilde{W}_\veps(A(I))$.
Note that when $I=\frakp\in R^{(e)}$, these are just  $ \Herm[\veps]{A(\frakp),\sigma(\frakp);\tilde{k}(\frakp)}$
and $\tilde{W}_\veps(A(\frakp))$, respectively.
We shall define a group homomorphism,
\[
\partial_{I,\frakq}=\partial_{I,\frakq}^A:\tilde{W}_\veps(A(I))\to \tilde{W}_\veps(A(\frakq)) ,
\]
called the \emph{$(I,\frakq)$-residue map}.

\medskip

Let $\frakm=\frakq/I$  and write
\begin{align*}
\tilde{S}&= \Ext^e_R(S,R) & \tilde{\frakm}^{-1} &=\Ext^e_R(\frakm,R) \\
\tilde{A}_S&= A \otimes \tilde{S} & A \tilde{\frakm}^{-1} &=  A \otimes \tilde{\frakm}^{-1}.
\end{align*}
Then $\tilde{S}$ is a dualizing module for $S$ \cite[Theorem~3.3.7(b)]{Bruns_1993_cohen_macaulay_rings}.
By Proposition~\ref{PR:CM-properties-divisor} and the fact $A$ that is a finite projective $R$-module, we may
regard   $\tilde{A}_S$ and $\tilde{\frakm}^{-1}A$ 
as submodules  of  $\tilde{A}(I)$, and then
\[
 A \tilde{\frakm}^{-1} =\{x\in \tilde{A}(I)\suchthat x \frakm\subseteq \tilde{A}_S\} .
\]
By Proposition~\ref{PR:CM-properties-I}, the short exact sequence $  \frakm\embeds S \onto k(\frakq) $
induces a short exact sequence 
$\tilde{S}\embeds \tilde{\frakm}^{-1}\onto \tilde{k}(\frakq)$, and
by tensoring with $A$, we get 
\begin{align}\label{EQ:dual-m-S-k-A}
  \tilde{A}_S\embeds \tilde{A} \frakm^{-1} \xonto{\calT} \tilde{A}(\frakq) .
\end{align}
We denote the right map by $\calT_{I,\frakq,A}$,
dropping some or all of the subscripts when they are clear from the context.

\medskip

Let $V\in\rmod{A(I)}$.
Following \cite[p.~129]{Reiner_2003_maximal_orders_reprint},
an \emph{$A $-lattice} in $V$
is an  $A $-submodule
$U\subseteq V$ which is finitely generated
over $S$ and satisfies ${U\cdot k(I)}=V$.

Let $(V,f)\in\tHerm[\veps]{A(I)}$.
Given  an $A $-lattice $U$ in $V$,
we write
\[
U^f:=\{x\in V\suchthat  {f}(U,x)\subseteq \tilde{A}_S\}.
\]
The following lemma was shown by the first and second
authors when $S$ is a discrete valaution ring;
see
\cite[Theorem~4.1]{Bayer_2017_rational_isomorphism_forms} and its proof.

\begin{lem}\label{LM:U-prime-properties}
	In the previous notation:
	\begin{enumerate}[label=(\roman*)]
	\item For every $A$-lattice $U$ in $V$, the set $U^f$ is an $A $-lattice in $V$
	and $U=U^{ff}$.	
	\item There exists an $A $-lattice $U$ in $V$ such that 
	$ U^f \frakm \subseteq U\subseteq U^f$.
	\end{enumerate}
\end{lem}

\begin{proof}
	(i) Recall from \ref{subsec:herm-forms} that $\Hom^\sigma_A(U,\tilde{A}_S)$
	denotes $\Hom_A(U,\tilde{A}_S)$ endowed with the \emph{right} $A$-module
	structure given by $(a\phi) x=a^\sigma \cdot \phi x$
	($a\in A$, $\phi\in\Hom_A(U,\tilde{A}_S)$, $x\in U$).
	
	Since $f$ is unimodular, the map
	$x\mapsto f(x,-):U^f\to \Hom_A^\sigma (U,\tilde{A}_S)$
	is an isomorphism of right $A_S$-modules.
	As $U$ is finitely generated over $A$ and $\tilde{A}_S$
	is noetherian, this means that $ U^f $ is finitely generated over $A_S$.
	Let $v\in V$. Then $f(U,v)$ is a left $A_S$-submodule of $\tilde{A}(I)$.
	Since $\tilde{A}_S$ is an $A$-lattice in $\tilde{A}(I)$, 
	there exists $c\in C$ such that $\hat{f}(U,v)c\subseteq \tilde{A}_S$,
	so $vc\in U^f$. As this holds for all $v\in V$,
	we have shown that  $U^f\cdot k(I)=V$.

	The unimodularity of $f$
	also implies that $x\mapsto f(x,-):U^{ff}\to \Hom^\sigma_A( {U^f},\tilde{A}_S)$
	is an isomorphism. Since $ U^f \cong \Hom^\sigma_A(U,\tilde{A}_S)$,
	we get an isomorphism $U^{ff}\cong \Hom^\sigma_A(\Hom^\sigma_A(U,\tilde{A}_S),\tilde{A}_S)$.
	It is routine to check
	that the composition $U\to U^{ff}\to \Hom^\sigma (\Hom_A^\sigma (U,\tilde{A}_S),\tilde{A}_S)$
	is   the map $\veps\cdot \omega_U$ of 	  \ref{subsec:herm-forms}, defined
	using $(Z,\theta)=(\tilde{A}_S,\sigma\otimes \id_{\tilde{S}} )$.
	By Corollary~\ref{CR:CM-Azumaya-properties} 
	and Proposition~\ref{PR:CM-properties-one-dim}, $\omega_U$ is an isomorphism, so $U=U^{ff}$.

	(ii) Choose an
	$A $-lattice $L$ in $V$.
	If $L\nsubseteq L^f$, replace $L$
	with $L\cap L^f$.
	Since $LC^{-1}=V=L^fC^{-1}$, there exists $c\in C$
	such that $L^fc\subseteq L$. This means that 
	$\supp_S(L^f/L)=\{\frakm\}$, so
	there exists $t\in \N\cup\{0\}$
	with $(L^f/L)\frakm^t\neq 0$ and $(L^f/L)\frakm^{t+1}=0$
	\cite[\href{https://stacks.math.columbia.edu/tag/00L6}{Tag 00L6}]{DeJong_2018_stacks_project}.
	If $t=0$, take $U=L$.
	If not, then let $M=L+L^f\frakm^t$.
	We have $f(M,M)\subseteq
	\tilde{A}_S+f(L^f\frakm^t,L^f\frakm^t) =
	\tilde{A}_S+f(L^f\frakm^{2t},L^f)\subseteq
	\tilde{A}_S+f(L,L^f)=\tilde{A}_S$,
	so
	$M\subseteq M^f$
	and $L\subsetneq M\subseteq M^f\subseteq L^f$.
	
	Replacing $L$ with $M$,
	we can repeat this process until we have $L^f\frakm\subseteq L$.
	The number of   iterations must be finite
	because $L^f$
	is noetherian.
\end{proof}

We can now   define 
$\partial_{I,\frakq}: \tilde{W}_\veps(A(I))\to
\tilde{W}_\veps(A(\frakq))$.
Given $(V,f)\in \tHerm[\veps]{A(I)}$,
apply Lemma~\ref{LM:U-prime-properties}(ii) to choose an $A $-lattice $U$ in $V$
such that $ U^f \frakm\subseteq U \subseteq U^f$.
Then $U^f/U\in \rproj{A(\frakq)}$,
and $ {f}(U^f,U^f)\subseteq  {A}\tilde{\frakm}^{-1}$ because 
$f(U^f,U^f)\frakm= f(U^f\frakm, U^f)\subseteq \tilde{A}_S$.
We may therefore define a biadditive map ${ \partial_Uf} :U^f/U\times U^f/U\to \tilde{A}(\frakq)$
by
\[
 {\partial_U f}(x+U,y+U)= \calT_{I,\frakq,A}(f(x,y)).
\]
We shall write $\partial f$ for $\partial_U f$ when $U$ is clear from the context.
Finally, set
\[
\partial_{I,\frakq}[V,f]=[U^f/U,\partial   f].
\]
This is well-defined by the following lemma.

\begin{lem}\label{LM:residue-well-defined} 
	In the previous notation, $ (U^f/U, {\partial   f})\in\tHerm[\veps]{A(\frakq)}$.
	Moreover, the Witt class $[U^f/U,\partial  f]\in\tilde{W}_\veps(A(\frakq))$ 
	depends only on the Witt class $[V,f]$ (and not on the choices of $V$, $f$, $U$).
\end{lem}

\begin{proof}
	Let $Z=\tilde{A}(I)/\tilde{A}_S= A\otimes (\tilde{k}(I)/S)$, $\theta=\sigma \otimes \id_{\tilde{k}(I)/S}$ 
	and let $\catC$ denote
	the category of right $A_S$-modules of finite length.
	We shall verify at the end that conditions \ref{item:herm-exact-cond} and \ref{item:herm-reflexive-cond} 
	of \ref{subsec:herm-forms}
	hold for $(Z,\theta)$ and $\catC$, so that we may consider
	the category $\Herm[\veps]{\catC}$ and the associated Witt group
	$W_\veps(\catC)$.
	
	We identify  $\tilde{A}(\frakq)$ with ${A}\tilde{\frakm}^{-1}/\tilde{A}_S\subseteq Z$
	via $\calT $. 
	Proposition~\ref{PR:CM-properties-divisor} implies that $A\tilde{\frakm}^{-1} /A_S$
	is the $S$-socle of $Z$, and thus
	$\Hom^\sigma_A(P,\tilde{A}(\frakq))=\Hom^\sigma_A(P,Z)$ for every
	$P\in \rproj{A(\frakq)}$. This allows us to regard
	every unimodular $(\tilde{A}(\frakq),\tilde{\sigma}(\frakq))$-valued
	$\veps$-hermitian space with underlying module in $\rproj{A(\frakq)}$
	as a unimodular $(Z,\theta)$-valued
	$\veps$-hermitian space.
	(More formally,   $(\rproj{A(\frakq)},\Hom^\sigma_A(-,\tilde{A}(\frakq)),\omega)$
	is as an exact hermitian subcategory of $(\catC,\Hom^\sigma_A(-,Z),\omega)$.)
	Since $\catC$ is abelian and its semisimple objects are the  objects of $\rproj{A(\frakq)}$,
	the 
	Quebbemann--Scharlau--Schulte Theorem
	\cite[Corollary~6.9, Theorem~6.10]{Quebbemann_1979_hermitian_categories}
	tells us that 
	the natural map $\tilde{W}_\veps(A(\frakq))\to W_\veps(\catC)$ is an isomorphism.

	For every $A $-lattice $U\subseteq V$ with $U\subseteq U^f$,
	define $ {\partial_U  f}:U^f/U\times U^f/U\to Z$ by   
	$ \partial_U  f (x+U,y+U)=f(x,y)+\tilde{A}_S$; this  agrees with our previous definition of $\partial_U f$
	when $U^f\frakm\subseteq U$. It is enough to show that $(U^f/U,\partial_U f)$
	lives in $\Herm[\veps]{\catC}$ and its Witt class in $W_\veps(\catC)$ 
	depends only on $[V,f]$.	

	That $\partial_U  f$ is $\veps$-hermitian is straightforward.
	The exactness of $\Hom^\sigma_A(-,Z)$ means that 
	$\operatorname{length}_A M=\operatorname{length}_A \Hom_A^\sigma(M,Z)$.
	Thus, in order
	to show that $\partial_U  f$ is unimodular,
	it is enough show  that $x\mapsto \partial f(x,-):U^f/U\to \Hom_A^\sigma(U^f/U,Z)$  is injective. 
	This is the same as saying that radical of $\partial f$ is $0$,
	which   follows readily from $U^{ff}=U$ (Lemma~\ref{LM:U-prime-properties}(i)).
	
	Next, assuming $V$ and $f$ are fixed,
	we check that $[\partial_U f]$ is independent of the $A$-lattice
	$U$. This an adaptation of a classical argument, \cite[p.~204]{Scharlau_1985_quadratic_and_hermitian_forms}
	(see also Lemma~6.1.4 and Theorem~5.3.4 in this source), to our more general situation.
	Let $U'$ be another $A$-lattice in $V$ with $U'\subseteq U'^f$.
	We need to show that $ [\partial_U f]= [\partial_{U'} f]$.
	Since $U'':=U\cap U'$ is an $A$-lattice  satisfying $U''\subseteq U''^f$ and $U''\subseteq U,U'$,
	it is enough to consider the case where $U'\subseteq U$.
	Observe that $U'\subseteq U\subseteq U^f\subseteq U'^f$,
	and let $g=\partial_U f\oplus (-\partial_{U'} f)$.
	We claim that $L := \{(x+U,y+U')\in U^f/U\times U^f/U'\suchthat x-y\in U\}$
	satisfies $L=L^\perp$ relative to $g$. Indeed, for all $x_1,x_2,y_1,y_2\in U^f$
	with $x_1-y_1,x_2-y_2\in U$, we have
	\begin{align*} g((x_1&+U,y_1+U'),(x_2+U,y_2+U'))
	=f(x_1,x_2)-f(y_1,y_2)+\tilde{A}_S  \\
	&= f(x_1,x_2)+f(x_1,y_2-x_2)+f(y_1-x_1,y_2)-f(y_1,y_2)+\tilde{A}_S=
	0+\tilde{A}_S ,
	\end{align*} so $g(L,L)=0$.
	On the other hand, if $x\in U$, $y\in U'$ satisfy
	$g(L,(x+U,y+U'))=0$, then $f(z,x)-f(w,y)\in \tilde{A}_S$ for all $z,w\in U^f$ with $z-w\in U$.
	Taking $z=w$ shows that $f(z,x-y)\in \tilde{A}_S$ for all $z\in U^f$, so $x-y\in U^{ff}=U$,
	whereas taking $z=0$ and $w\in U$  shows that $f(U,y)\in \tilde{A}_S$, so $y\in U^f$.
	Likewise, $x\in U^f$ and we conclude that $(x+U,y+U')\in L$, i.e., $L^{\perp}=L$.
	This means that $[\partial_U f\oplus (-\partial_{U'} f)]=0$ in $W_\veps(\catC)$,
	or rather, $[\partial_Uf ]=[\partial_{U'} f]$.
	
	We now show that $[\partial_U f]$ depends only on $[V,f]$ and  not   on $(V,f)$.
	Observe that if $(V',f')\in \tHerm[\veps]{A(\frakp)}$,
	and $U'$ is an $A$-lattice in $V'$ with $U'\subseteq U'^{f'}$,
	then $\partial_{U\oplus U'}(f\oplus f')\cong \partial_U f\oplus \partial_{U'}f'$. 
	Thus, in order to prove our claim, it is enough to show that $[\partial_U f]=0$
	whenever $(V,f)$ is metabolic. Since we assume that $2\in\units{R}$,
	\cite[Proposition~I.3.7.1]{Knus_1991_quadratic_hermitian_forms}
	tells us that there is $V_1\in\rproj{A(\frakp)}$ such that
	$V\cong V_1\times \Hom^\sigma_A(V_1,\tilde{A}(\frakp))$
	and under this isomorphism, $f$ is given by
	$f((x,\phi),(x',\phi'))=\phi x'+\veps (\phi' x)^{\tilde{\sigma}(\frakp)}$. Identify
	$V$ with $V_1\times \Hom^\sigma_A(V_1,\tilde{A}(\frakp))$. Let $U_1$
	be an $A$-lattice in $V_1$, and regard $U'_1:=\Hom(U_1,\tilde{A}_S)$
	as an $A$-lattice in $V'_1 := \Hom^\sigma_A(V_1,\tilde{A}(\frakp))$. 
	Then $f$ restricts to a $\tilde{A}_S$-valued  $A_S$-bilinear pairing on $U:=U_1\times U'_1$.
	Written in matrix form, the map $x\mapsto f(x,-):U\to \Hom^\sigma_A(U,\tilde{A}_S)$,
	denoted $\hat{f}$,
	is $[\begin{smallmatrix} 0 & \id_{U'_1} \\ \veps\omega_{U_1} & 0\end{smallmatrix}]$,
	where $\omega_{U_1}$ is the natural map $U_1\to \Hom_A^\sigma(\Hom^\sigma_A(U_1,\tilde{A}_S),\tilde{A}_S)$
	considered in \ref{subsec:herm-forms}.
	By Corollary~\ref{CR:CM-Azumaya-properties}
	and~\ref{PR:CM-properties-one-dim} (applied with $S$ in place of $R$ and $e=1$),
	$\omega_{U_1}$ is an isomorphism, and thus so is $\hat{f}$.
	This implies readily that $U^f=U$, so $[U^f/U,\partial  f]=0$ 
	in $W_\veps(\catC)$.

	It remains to 	establish \ref{item:herm-exact-cond} and \ref{item:herm-reflexive-cond} 
	for $\catC$ and $(Z,\theta)$.
	Condition \ref{item:herm-exact-cond} will follow if we show that $\Ext^1_{A_S}(M,Z)=0$
	for all $M\in\catC$. It is enough
	to check this when $M$ is semisimple, and hence
	for $M= {A}(\frakq)$.
	Since  $\Ext^1_{A_S}(A(\frakq),Z)\cong\Ext^1_S(k(\frakq),\tilde{k}(I)/\tilde{S})\otimes A $
	\cite[Theorem~2.39]{Reiner_2003_maximal_orders_reprint},
	we are reduced into showing that $\Ext^1_S(k(\frakq),\tilde{k}(I)/\tilde{S})=0$.
	To that end, it is enough to show that $\Ext^1_S(k(\frakq),\tilde{k}(I))=0$
	and $\Ext^2_S(k(\frakq),\tilde{S})=0$.
	The first claim holds because  $\Ext^1_S(k(\frakq),\tilde{k}(I))$ is a
	$C$-divisible  module annihilated by $\frakq$, and the second follows
	from Proposition~\ref{PR:CM-properties-I}(ii).
	Now that $\Hom_A^\sigma(-,Z)$ is exact on $\catC$, we can induct on the length of $M\in\catC$ in order
	to show that $\omega_M $ is an isomorphism.
	When $M$ is simple, this holds because $M\in\rproj{A(\frakq)}$
	and $\Hom_A(M,Z)=\Hom_A(M,\tilde{A}(\frakq))$, so we are done.
\end{proof}

	We shall see in Example~\ref{EX:second-residue-II} that when
	$S=R/I$ is a discrete valuation ring, $A=R$ and $\veps=1$, the map $\partial_{I,\frakq}$
	is   equivalent to the classical second residue map.

\subsection{Differentials}
\label{subsec:dd-description}

We   retain our original setting where $R$ is a  regular ring of dimension $d$.
For every $e\geq 0$, $\frakp\in R^{(e)}$ and $\frakq\in R^{(e+1)}$, define
$\partial_{\frakp,\frakq}=\partial_{\frakp,\frakq}^A:\tilde{W}_\veps(A(\frakp))\to\tilde{W}(A(\frakq))$
to be the $(\frakp_\frakq,\frakq_\frakq)$-residue map 
$\partial_{\frakp_{\frakq},\frakq_\frakq}$ of \ref{subsec:second-res} if $\frakp\subseteq\frakq$,
and $0$ otherwise. We define the $e$-th differential of $\GW{A,\sigma,\veps}$ to be
\[\dd_e:=\sum_{\frakp\in R^{(e)}}\sum_{\frakq\in R^{(e+1)}}\partial^A_{\frakp,\frakq}:\GW{A,\sigma,\veps}_e\to \GW{A,\sigma,\veps}_{e+1}.\]
By the following lemma,
the inner sum is finite   when evaluated on a Witt class.

\begin{lem}\label{LM:partial-well-defined}
	Let $\frakp\in R^{(e)}$
	and $(V,f)\in\tHerm{A(\frakp) }$.
	Then there exist only finitely many
	ideals $\frakq\in R^{(e+1)}$
	such that $\partial_{\frakp,\frakq}[V,f]\neq 0$.
\end{lem}

\begin{proof}
	We may assume $e<\dim R$.
	Write $S=R/\frakp$ (beware that $S$ is not local)
	and $\tilde{S}=\Ext^e_R(S,R)$.
	We may still consider $A $-lattices $U$ in $V$
	and define the $A $-lattice $U^f$ as in \ref{subsec:second-res}
	($U^{ff}=U$ is no longer guaranteed).

	Let $\frakq\in R^{(e+1)}$ be an ideal containing $\frakp$,
	and  identify  $\tilde{S}_\frakq$ with $\Ext^e_{R_\frakq}(S_\frakq,R_\frakq)$.
	We claim that   $(U^f)_\frakq=(U_\frakq)^{f }$, where the right hand
	side is defined as in \ref{subsec:second-res} relative to the ideals $I=\frakp_\frakq$
	and $\frakq_\frakq$.
	Indeed, recall from the proof of Lemma~\ref{LM:U-prime-properties} that $x\mapsto f(x,-)$
	defines an $S$-isomorphism $ {U^f}\to\Hom^\sigma_A(U, A \otimes \tilde{S})$,
	and likewise, $(U_\frakq)^{f }\cong \Hom^\sigma_A(U_\frakq, A \otimes \tilde{S}_\frakq)$.
	It is routine to check that under these isomorphisms,
	the natural isomorphism $\Hom_A(U,\tilde{S}\otimes A)_\frakq\to
	\Hom_A(U_\frakq,\tilde{S}_\frakq\otimes A)$ 
	corresponds
	to the natural map
	$(U^f)_\frakq\to (U_\frakq)^{f }$, hence our claim.

	Choose an $A $-lattice $U$ in $V$,
	and replace it with $U\cap U^f$  to assume $U\subseteq U^f$.
	By the previous paragraph, we have $(U_\frakq)^f=(U^f)_\frakq = U_\frakq$
	for every $\frakq\in \Spec R$ containing $\frakp$ with $\frakq\notin \supp_R U^f/U$.
	Thus, $\partial_{\frakp,\frakq} [V,f]=0$ for all $\frakq\notin \supp_R (U^f/U)\cap R^{(e+1)} $.
	Choose $s\in S-\{0\}$ such that $U^f s\subseteq U$. Then
	$U^f/U$ is a finite $S/sS$-module. Since $\dim S/sS\leq \dim S-1$,
	the set   $\supp_R (U^f/U)\cap R^{(e+1)}$ is finite, hence the lemma.
\end{proof}

	It remains to check that $\aGW{A,\sigma,\veps}$
	is a cochain complex, i.e., that
	$\dd_{e+1}\circ\dd_e=0$ for all $e\geq -1$.
	This is the content of the next theorem.
	
	\begin{thm}\label{TH:GW-is-a-complex}
		In the previous notation, $\dd_{e+1}\circ \dd_e=0$ for all $e\geq -1$.
	\end{thm}
	
	The proof is somewhat technical. 
	Given $M\in \rMod{R}$, we shall abbreviate $A\otimes M$ to $AM$ for brevity.
	We first prove the following   lemma.

\begin{lem}\label{LM:change-of-ideals}
	Assume $R$ is a regular local ring of dimension $e+1$ ($e\geq 0$) with
	maximal ideal $\frakq$.
	Let $I_2\subseteq I_1$ be two ideals of $R$ such that both $S_1:=R/I_1$ and $S_2:=R/I_2$
	are Cohen--Macaulay of dimension $1$. For $i=1,2$,
	define $C_i$,  $\frakm_i$, $\tilde{S}_i$, $\tilde{\frakm}^{-1}_i$   
	similarly to $C$, $\frakm$, $\tilde{S}$,  $\tilde{\frakm}^{-1}$ in \ref{subsec:second-res}.
	Then:
	\begin{enumerate}[label=(\roman*)]
		\item $C_2\subseteq C_1$, and
		for every finite $C_1$-torsion-free $S_1$-module $M$, the map $MC_2^{-1}\to MC_1^{-1}$ is an isomorphism.
	\end{enumerate}
	Let $\iota=\iota_{I_1,I_2}$ denote the composition
	\[
	\iota_{I_1,I_2}: \tilde{A}(I_1)=A\tilde{S}_1 C_1^{-1} =
	A\tilde{S}_1 C_2^{-1} \to A \tilde{S}_2 C_2^{-1} =\tilde{A}(I_2)
	\]
	induced by (i) and the quotient map $R/I_2\to R/I_1$.
	\begin{enumerate}[label=(\roman*), resume]	
		\item 
		The map $\iota: \tilde{A}(I_1)\to \tilde{A}(I_2)$ is injective,
		$\im(\iota)= \Ann_{\tilde{A}(I_2)} I_1$, 
		and we
		have   $A\tilde{S}_2 \cap \iota(\tilde{A}(I_1))=\iota(A\tilde{S}_1)$
		and $A\tilde{\frakm}_2^{-1}\cap \iota(\tilde{A}(I_1))=\iota(A\tilde{\frakm}_1^{-1})$
		as subsets of $\tilde{A}(I_2)$.
		\item
		For every $(V,f)\in \tHerm[\veps]{A(I_1) }$,
		we have $(V,\iota \circ f)\in \tHerm[\veps]{A(I_2) }$
		and 
		\[ \partial_{I_2,\frakq} [V,\iota \circ f]=\partial_{I_1,\frakq}[V, f] .\]
	\end{enumerate}
\end{lem}

\begin{proof}
	(i) That $C_2\subseteq C_1$ is a consequence of Lemma~\ref{LM:zero-divisors}.
	The same proposition also implies that $S_2C_2^{-1}$ is $0$-dimensional, so it is   artinian. 
	This means that $\operatorname{length}_{RC_2^{-1}} MC_2^{-1}<\infty$. 
	Since every element of $C_1$ acts faithfully on $M$, 
	it follows that $MC_2^{-1}$ is $C_1$-divisible,
	and  $MC_2^{-1}\to MC_2^{-1}C_1^{-1}=MC_1^{-1}$ is an isomorphism.
	
	(ii)
	Let $r_1,\dots,r_t\in R$ be generators
	of $I_1$. Define
	$f: S_2^t\to S_2$  
	by $f(x_1,\dots,x_t)=\sum_{i=1}^t r_i x_i $ and let $K=\ker f$.	
	Then we have an exact sequence
	\begin{align*}
	0\to K\to S_2^t\xrightarrow{f} S_2\to S_1\to 0
	\end{align*}
	which, thanks to Proposition~\ref{PR:CM-properties-one-dim}, consists of $1$-dimensional
	Cohen-Macaulay $R$-modules. Since $\Ext^e_R(-,R)$ induces a duality on the latter
	\cite[Theorem~3.3.10(c)]{Bruns_1993_cohen_macaulay_rings}, and since $A$ is flat over $R$,
	we have an exact sequence of $R$-modules
	\begin{align}\label{EQ:exact-S1-S2-sequence}
	0\to A\tilde{S}_1 \xrightarrow{\iota|_{A\tilde{S}_1}} A\tilde{S}_2 \xrightarrow{f'} (A\tilde{S}_2)^t
	\end{align}
	in which $f'$ is given by $f'(x)=(r_1 x,\dots,r_t x )$. 
	
	By localizing \eqref{EQ:exact-S1-S2-sequence}  at $C_2$, 
	we see that $\iota :\tilde{A}(I_1)\to \tilde{A}(I_2)$
	is injective and $\im(\iota)=\ann_{\tilde{A}(I_2)} I_1$.
	
	That  $A \tilde{S}_2 \cap \iota(\tilde{A}(I_1))=\iota(A \tilde{S}_1 )$
	follows by considering the map from \eqref{EQ:exact-S1-S2-sequence}  to its localization at $C_2$
	and simple diagram chasing. Note that all $R$-modules in \eqref{EQ:exact-S1-S2-sequence}
	are $C_2$-torsion-free (Proposition~\ref{PR:CM-properties-I}).
	
	Finally, 
	if $x\in \iota(A\tilde{\frakm_1}^{-1})$, then $x I_1=0$ and $x\frakm_2\subseteq \iota(A\tilde{S}_1)\subseteq A\tilde{S}_2$,
	so $x\in \tilde{A}\frakm_2^{-1}\cap	\Ann_{\tilde{A}(I_2)}I_1= \tilde{A}\frakm_2^{-1}\cap \iota(\tilde{A}(I_1))$.
	Conversely, if $x\in \tilde{A}\frakm_2^{-1}\cap \iota(\tilde{A}(I_1))$,
	then there is $y\in \tilde{A}(I_1)$ with $\iota(y)=x$. We have
	$x\frakm_2\subseteq  A\tilde{S}_2 \cap \iota(\tilde{A}(I_1))=\iota(\tilde{A}S_1)$, so   
	$y\frakm_1\in A\tilde{S}_1$.
	This means that $y\in A\tilde{\frakm}^{-1}_1$ and $x\in \iota(A\tilde{\frakm_1}^{-1})$.

	(iii) 	
	That $\iota\circ f$ is unimodular follows easily from   $\iota(\tilde{A}(I_1))=\Ann_{\tilde{A}(I_2)} I_1$
	and $VI_1=0$.

	By Lemma~\ref{LM:U-prime-properties},
	there exists   an $A$-lattice $U$ in $V$ with $U^f \frakq\subseteq U\subseteq U^f$.
	Thanks to (ii) and the fact that $U\cdot I_1=0$, we have
	$U^{\iota\circ f}=U^f$, so $\partial_{I_2,\frakq}[V,\iota \circ f]$
	is represented by $g:U^f/U\times U^f/U\to \tilde{k}(\frakq)$
	given by $g(x+U,y+U)=\calT_{I_2,\frakq}(\iota(f(x,y)))$.
	It is therefore enough to show that $\calT_{I_2,\frakq,A}\circ \iota|_{A\tilde{\frakm}_1^{-1}} =\calT_{I_1,\frakq,A}$.
	This follows by applying $\Ext_R^*(-,R)\otimes A $ to the to following morphism of short
	exact sequences:
	\[
	\begin{gathered}[b]
	\xymatrix{
	\frakm_2 \ar@{^{(}->}[r] \ar[d] &  S_2 \ar[r]\ar@{->>}[d] & k(\frakq)\ar@{=}[d] \\
	\frakm_1 \ar@{^{(}->}[r] & S_1 \ar@{->>}[r] & k(\frakq)
	} \\[-\dp\strutbox]
	\end{gathered}
	\qedhere
	\]
\end{proof}

\begin{proof}[Proof of Theorem~\ref{TH:GW-is-a-complex}]
	That $\dd_{0}\circ\dd_{-1}=0$ is   straightforward and is left to the reader. We will
	prove that $\dd_{e+1}\circ \dd_e=0$ for all $e\geq 0$. This amounts
	to showing that for all $\frakt\in R^{(e)}$, $\frakq\in R^{(e+2)}$ with $\frakt\subseteq \frakq$
	and $(V,f)\in\tHerm[\veps]{A(\frakt)}$, we have
	$\sum_{\frakp} \partial_{\frakp,\frakq} \partial_{\frakt,\frakp} [V,f]=0$,
	where the sum is taken over the set of primes lying strictly
	between $\frakt$ and $\frakq$.
	To that end, we may base-change from $R$ to $R_\frakq$ and assume that $R$ is regular local
	of dimension $e+2$ and $\frakq$ is its maximal ideal.
	
	The set of height-$i$ primes containing a given ideal $J$ will be denoted
	$R^{(i)}_{\supseteq J}$. Given $M\in \rMod{R}$, we shall abbreviate 
	\[M^{[i]}:=\Ext^i_R(M,R) ,\]
	e.g.\ $\tilde{k}(\frakt)=(R/\frakt)^{[e]}_\frakt$.

\medskip

	\Step{1}	
	Since $R$ is regular and $\hgt \frakt=e$, 
	there is a regular sequence $r_1,\dots,r_e\in \frakt$ 
	\cite[Theorem~17.4(i)]{Matsumura_1989_commutative_ring_theory}.
	Write $J=r_1R+\dots+r_e R$, and let $D$ denote the set of elements of $R$
	acting faithfully on $R/J$. 
	Then $R/J$ is a  complete intersection local ring of dimension $2=(e+2)-e$, hence Cohen-Macaulay 
	with dualizing module $(R/J)^{[e]}$.
	
	Proposition~\ref{PR:CM-properties-I}(i) and the flatness of $A$ over $R$
	imply that for every 
	ideal $J'\subseteq R$ with  $J\subseteq J'$ and $J'\cap D\neq \emptyset$, 
	the natural map 
	$A (J'/J)^{[e]}  \to A(J'/J)^{[e]}D^{-1}  =
	A(R/J)^{[e]}D^{-1} $ is injective.
	We shall therefore freely regard $A (J'/J)^{[e]} $ as a subset
	of $A(R/J)^{[e]}D^{-1} $. If $J''$ is another ideal satisfying  $J\subseteq J''$ 
	and $J''\cap D\neq\emptyset$,
	then  $A(J'/J)^{[e]}D^{-1}\subseteq A(J''/J)^{[e]}D^{-1}$ whenever $J''\subseteq J'$.
	Note that  $J'\cap D\neq \emptyset $ holds whenever $J'/J$ is not 
	contained in a minimal prime of $R/J$ (Lemma~\ref{LM:zero-divisors}).

	By Lemma~\ref{LM:change-of-ideals}(iii),
	for every $\frakp\in R^{(e+1)}$ containing $\frakt$, 
	we have
	$\partial_{J,\frakp}[\iota\circ f]=\partial_{\frakt,\frakp}$,
	where $\iota=\iota_{\frakt_\frakp,J_\frakp}: A(R/\frakt)^{[e]}_\frakt  
	\to A(R/J)^{[e]}D^{-1} $ (this is   independent of $\frakp$, in fact).
	Replacing $f$ with $\iota\circ f$, we are reduced into proving
	\[
	\sum_{\frakp} \partial_{\frakp,\frakq}\partial_{J,\frakp}[f]=0,
	\]
	where the sum is taken over $R^{(e+1)}_{\supseteq J}$
	(note that if $\frakt\nsubseteq \frakp$, then $V_\frakp=0$ and $\partial_{J,\frakp}[f]=0$).

\medskip

\Step{2}
	Let us call an $A$-submodule $L\subseteq V$
	an $A$-lattice (in $V$) if $L$ is finitely generated
	and $L D^{-1}=V$. In this case,   define
	\[L^{f}=\{x\in V\suchthat f(L,x) \subseteq A (R/J)^{[e]} \} .\]
	As in the proof of Lemma~\ref{LM:U-prime-properties}, $L^f$ is also an $A$-lattice in $V$.

	Choose an $A$-lattice $L$ in $V$.
	Replacing $L$  with $L\cap L^f$,
	we may assume that $L\subseteq L^{ f}$.
	Since $L^{ f}$ is a noetherian $R$-module, there is $L'\subseteq L^{f}$
	maximal with respect to the property that $L'\subseteq L'^{f}$.
	Replace $L$ with $L'$. Then every $A$-lattice $M$ containing $L$ and
	satisfying $M\subseteq M^{ f}$ must be equal to $L$.
	In particular, $L^{ff}=L$.
	
	Since $L D^{-1} = V=L^fD^{-1}$,  there is some $r\in D\cap \frakq$
	such that $L^{  f} r\subseteq L$. Fix such $r$, let 
	\[I=J+Rr=Rr_1+\dots+Rr_e+Rr\]
	and let $C$ denote the set of elements of $R$ acting faithfully on $R/I$.
	Since $r_1,\dots,r_e,r$ is a regular sequence,  $R/I$ is a $1$-dimensional
	complete intersection local ring with dualizing module $(R/I)^{[e+1]}$.

	As in the proof of Lemma~\ref{LM:partial-well-defined},
	the set 
	\[(LC^{-1})^f:=\{x\in V\suchthat f(LC^{-1},x)\subseteq A(R/J)^{[e]}C^{-1} \}\]
	coincides with $L^f C^{-1}$.
	Since $L^fC^{-1}r  \subseteq L C^{-1}$,
	we have $f(L^fC^{-1},L^fC^{-1}) r\subseteq A(R/J)^{[e]}C^{-1} $,
	or rather, $f(L^fC^{-1},L^fC^{-1}) \subseteq A(R/J)^{[e]}C^{-1}  \cdot r^{-1}$.
	By Proposition~\ref{PR:CM-properties-divisor}, this means that
	$f(L^fC^{-1},L^fC^{-1}) \subseteq
	A(I/J)^{[e]}C^{-1}  $. Writing  $U=L^fC^{-1}/LC^{-1}$, we may now define
	$g:U\times U\to \tilde{A}(I)=A (R/I)^{[e]}C^{-1} $ by
	\[
	g(x+LC^{-1},y+LC^{-1})=\calT_{J,I}(f(x,y)) ,
	\]
	where
	\[
	\calT_{J,I}:A(I/J)^{[e]}C^{-1} \to A(R/I)^{[e]}C^{-1} 
	\]
	is induced by   $I/J\embeds R/J \onto R/I$
	
\medskip
	
	\Step{3}	
	Let $I_0$ denote the radical of $I$.
	Since $L^{ f} I=L^fr\subseteq L$ and $R$ is noetherian,
	there is some $t\in \N\cup\{0\}$ such that $L^{f} I_0^{t}\subseteq L$.
	If $t>1$, then one readily checks that 
	$f(L+L^{ f} I_0^{t-1}, L+L^{f} I_0^{t-1})\subseteq A(R/J)^{[e]} $,
	or rather, $ L+L^fI_0^{t-1} \subseteq (L+L^fI_0^{t-1})^f$.
	By the maximality of $L$ (Step~2), this means that 
	$L^f I_0^{t-1}\subseteq L$.
	Continuing in this manner, we find that $L^fI_0\subseteq L$.

	As in the proof of Lemma~\ref{LM:partial-well-defined},
	for every $\frakp\in R^{(e+1)}_{\supseteq J}$,
	we have $(L_\frakp)^f=(L^f)_\frakp$.
	If $\frakp\supseteq I$, then 
	$(L^f)_\frakp \frakp_\frakp=(L^f)_\frakp (I_0)_\frakp \subseteq L_\frakp$, and thus  $\partial_{J,\frakp} [f]$ is represented by $((L^f)_\frakp/L_\frakp,\partial_{L_\frakp} f)$.
	Other other hand, if  $\frakp\nsupseteq I$, then $L^f_\frakp= L_\frakp$ and
	$\partial_{L_\frakp}f=0$.

	Since $U=L^fC^{-1}/LC^{-1}$ is a module over the artinian ring  $(R/I)C^{-1}$,
	the natural map 
	$U\to \prod_{\frakp\in R^{(e+1)}_{\supseteq I}} U_\frakp= \prod_{\frakp\in R^{(e+1)}_{\supseteq I}} (L^f)_\frakp/L_\frakp$
	is an isomorphism. We claim that under this isomorphism,
	\begin{align}
	\label{EQ:g-is-direct-sum}
	g=\bigoplus_{\frakp\in R^{(e+1)}_{\supseteq I}} (\iota_{\frakp,I}\circ \partial_{L_\frakp}f).
	\end{align}
	(In particular, $g:U\times U\to A (R/I)^{[e+1]}C^{-1} $ is unimodular.)
	Provided this holds,   Lemma~\ref{LM:change-of-ideals}(iii) tells us that
	\[
	\partial_{I,\frakq}[g]=
	\sum_{\frakp\in R^{(e+1)}_{\supseteq I}} \partial_{\frakp,\frakq}  \partial_{J,\frakp}[f]
	\]
	so we are reduced into proving that $\partial_{I,\frakq}[U,g]=0$.
	
	We now prove \eqref{EQ:g-is-direct-sum}.
	Let $\frakp\in R^{(e+1)}_{\supseteq I}$.
	Since $UI_0=0$, 
	the image of the (unique) $R$-module section of $U\to U_\frakp$ is the set
	of $\frakp$-torsion elements in $U$. Let $\quo{x},\quo{y}$ be elements
	in that set, and let $x,y\in L^fC^{-1}$ be lifts
	of $x,y$. Then $y\frakp\subseteq LC^{-1}$, and so $f(x,y)\frakp=f(x,y\frakp)\subseteq   A (R/J)^{[e]}C^{-1}$.
	By   Proposition~\ref{PR:CM-properties-divisor}, this means that
	$f(x,y)\in A (\frakp/J)^{[e]}C^{-1}  $,
	and thus $g(\quo{x},\quo{y})=\calT_{J,I}(f(x,y))$.
	On the other hand, $(\iota_{\frakp,I}\circ \partial_{L_\frakp }f)(\quo{x},\quo{y})=\iota_{\frakp, I} \calT_{J,\frakp}f(x,y)$, 
	so it is enough  enough to check that 
	$ \calT_{J,I}$ and $\iota_{\frakp, I} \circ \calT_{J,\frakp}$
	agree on $A(\frakp/J)^{[e]}C^{-1} $. This follows 
	from the commutative diagram on the right, which is induced by the commutative
	diagram on the left.	
	\[
	\xymatrix{
	I/J \ar@{^{(}->}[r] \ar[d] &  R/J \ar[r]\ar@{=}[d] & R/I\ar[d] \\
	\frakp/J \ar@{^{(}->}[r] & R/J \ar@{->>}[r] & R/\frakp
	}
	\qquad
	\xymatrix{
	A(I/J)^{[e]}C^{-1} \ar[r]^{\calT_{J,I}}  &  A(R/I)^{[e]}C^{-1}   \\
	A(\frakp/J)^{[e]}C^{-1}   \ar[r]^{\calT_{J,\frakp}} \ar@{^{(}->}[u] &
	A(R/\frakp)^{[e+1]}C^{-1}  \ar[u]^{\iota_{\frakp,I}}  
	}
	\]
	
\medskip

	\Step{4} 
	Write $M=L^f/L$.  We claim that 
	$M$ is either $0$ or a $1$-dimensional Cohen--Macaulay $R$-module.	
	By Proposition~\ref{PR:CM-properties-one-dim}, it is enough to show
	that $\frakq\notin \Ass_R  M $.
	Suppose that $\quo{x}\in M$ is annihilated by $\frakq$,
	and let $x\in L^f$ be a lift of $x$. Then $ x \frakq\in L$, so
	$f(x,x)\frakq=f(x,x\frakq)\subseteq A(R/J)^{[e]} $.
	By Proposition~\ref{PR:CM-properties-divisor}, this means
	that $f(x,x)\in A(\frakq/J)^{[e]} $.
	Since both $(R/\frakq)^{[e]}$ and $(R/\frakq)^{[e+1]}$ are $0$ (Proposition~\ref{PR:CM-properties-I}),
	the map $(R/J)^{[e]}\to (\frakq/J)^{[e]}$ is an isomophism, and
	$f(x,x)\in A(R/J)^{[e]}$. This means
	that    $f(L+xA,L+xA)\subseteq (R/J)^{[e]}$, so
	$L+xA=L$ by the maximality of $L$ (Step~2). We conclude that $\quo{x}=0$.

	Next, we claim that $L^f$ is $2$-dimensional Cohen--Macaulay $R$-module.
	Since the element $r$ of Step~2 acts faithfully on $L^f$,
	it is enough to show that $L^f/ L^fr$ is a $1$-dimensional Cohen-Macaulay $R$-module,
	which again amounts to showing $\frakq\notin \Ass_R (L^f/ L^fr)$.
	Let $\quo{x}\in L^f/ L^fr$ be an element annihilated by $\frakq$
	and let $x\in L^f$ be a lift of $\quo{x}$.
	Then $ x \frakq\in  L^f r$, hence $f(L,xr^{-1})\frakq=f(L,xr^{-1}\frakq)\subseteq A (R/J)^{[e]} $.
	As in the previous paragraph, this means that $f(L,xr^{-1})\subseteq A(R/J)^{[e]}$ and $xr^{-1}\in L^f$.
	Thus,  $x\in L^fr$ and $\quo{x}=0$.

\medskip	
	
	\Step{5} By Step~4 and Proposition~\ref{PR:CM-properties-one-dim}, 
	$M$ is $C$-torsion-free, so we may regard $M$ as an $A$-lattice in $U$.
	We  will prove that $M^g=M$,
	and thus   $\partial_{I,\frakq}[g]=0$. 
	
	We first prove that $M\subseteq M^g$, or equivalently, $g(M,M)\subseteq A(R/I)^{[e+1]}$.
	Consider the diagram on the right, which is induced by the commutative diagram on the left:
	\[
	\xymatrix{
	I/J \ar@{^{(}->}[r]\ar@{=}[d] & 
	\frakq/J \ar@{^{(}->}[d]  \ar@{->>}[r] &  
	\frakq/I \ar@{^{(}->}[d] \\
	I/J \ar@{^{(}->}[r]  & 
	R/J \ar@{->>}[r] \ar@{->>}[d]  &
	R/I \ar@{->>}[d] \\
	& 
	k(\frakq) \ar@{=}[r] &
	k(\frakq)
	}
	\qquad
	\xymatrix{
	A (I/J)^{[e]}   \ar[r]^{\calT_{I/J}} &
	A (\frakq/I)^{[e+1]}    \ar[r] \ar[d]^{\calT_{I,\frakq}} &
	A (\frakq/J)^{[e+1]}   \ar[d] \\
	&
	\tilde{A}(\frakq)   \ar@{=}[r] &
	\tilde{A}(\frakq)   
	}
	\]
	The top row of the right diagram is exact in the middle, so $\calT_{I,\frakq}\circ \calT_{I/J}|_{A(I/J)^{[e]}}=0$,
	or rather, $   \calT_{I,J} (A(I/J)^{[e]})\subseteq \ker \calT_{I,\frakq} = A(R/I)^{[e+1]} $ 
	(see~\eqref{EQ:dual-m-S-k-A}).
	Since $f(L^f,L^f)I=f(L^f I,L^f)\subseteq f(L,L^f)\subseteq A(R/J)^{[e]}$,
	we have $f(L^f,L^f)\subseteq A(I/J)^{[e]}$
	(Proposition~\ref{PR:CM-properties-divisor}),
	and  it follows that $g(M,M)\subseteq A(R/I)^{[e+1]} $.
	
\medskip

	\Step{6}	
	It remains to show that $y\mapsto g(y,-):M\to \Hom_A(M,A(R/I)^{[e]} )$,
	denoted $\vphi$,
	is an isomorphism. Consider the following  diagram of $R$-modules:
	\begin{align*}
	\xymatrix{
	L \ar[rr]^-{x\mapsto f(x,-)} \ar@{^{(}->}[d]  & & \Hom_{A/AJ}(L^f,A(R/J)^{[e]} )\ar[d] \\
	L^f \ar[rr]^-{x\mapsto f(x,-)} \ar@{->>}[d] & & \Hom_{A/AJ}(L,A(R/J)^{[e]} ) \ar[d]^u \\
	M \ar[r]^-{\vphi} & 
	\Hom_A(M,A(R/I)^{[e+1]} ) \ar[r]^v &
	\Ext^1_{A/AJ}(M,A(R/J)^{[e]} )
	}
	\end{align*}
	The right column is induced by $L\embeds L^f\onto M$, so it is exact in the middle, and the map
	$v$ is induced by $(R/J)^{[e]}\embeds (I/J)^{[e]}\onto (R/I)^{[e+1]}$ (we have
	$(R/I)^{[e]}=0$ and $(R/J)^{[e+1]}=0$ by Proposition~\ref{PR:CM-properties-I}(ii)).
	The top and middle horizontal maps are bijective because $f$ is unimodular and $L^{ff}=L$, 
	and the top rectangle clearly commutes.
	Thus, in order to show that $\vphi$ is an isomorphism, it is enough to show the following:
	\begin{enumerate}[label=(\roman*)]
		\item $u$ is onto,
		\item $v$ is bijective,  
		\item the bottom rectangle 
		commutes.
	\end{enumerate}

	By Proposition~\ref{PR:CM-Azumaya-properties}(i) and Step~4,
	we have  $\Ext^1_{A/AJ }(L^f,A(R/J)^{[e]} )=0$, so
	$u$ is onto. We have $I/J = r(R/J)$, so for similar reasons, $\Hom_{A/AJ}(M,A(I/J)^{[e]} )\cong
	\Hom_{A/AJ}(M,A(R/J)^{[e]})=0$ and 
	$v$ is injective.
	To show that $v$ is surjective, it is enough to check that 
	$\Ext^1_{A/AJ}(M,A(R/J)^{[e]} )\to \Ext^1_{A/AJ}(M,A(I/J)^{[e]} )$
	is the zero map. Observe that $I/J\embeds R/J\onto R/I$
	is isomorphic to $R/J\xembeds{r} R/J \onto R/I$. This means
	that the morphism $\Ext^1_{A/AJ}(M,A(R/J)^{[e]} )\to \Ext^1_{A/AJ}(M,A(I/J)^{[e]} )$
	is isomorphic to $r:\Ext^1_{A/AJ}(M,A(R/J)^{[e]} )\to \Ext^1_{A/AJ}(M,A(R/J)^{[e]} )$,
	which is the zero map because $Mr=0$. We conclude that (i) and (ii) hold.

	In order to prove (iii),
	we interpret elements of $\Ext^1_{A/AJ}(M,A(R/J)^{[e]} )$ as $A$-module extensions  
	$M$ by $A(R/J)^{[e]} $.
	Let $x\in L^f$, and let
	\begin{align*}
	X&=\{(u,m)\in (I/J)^{[e]}\times M\suchthat \calT_{I,J}(u)=g(x+L,m)\},\\
	Y&=(A(R/J)^{[e]} \times L^f)/\{(f(x,\ell),-\ell)\where \ell\in L\}.
	\end{align*}	
	One readily checks that the  extensions
	of corresponding to the two possible images of $x$ in $\Ext^1_A(M,A(R/J)^{[e]})$ are
	\begin{align*}
	A(R/J)^{[e]}  \xembeds{\alpha} X \xonto{\beta} M 
	\qquad
	\text{and}
	\qquad
	A(R/J)^{[e]}  \xembeds{\gamma} Y \xonto{\delta} M 
	\end{align*}
	where $\alpha(u)=(u,0)$, $\beta(u,m)=m$, 
	$\gamma(u)=\quo{(u,0)}$ and $\delta\quo{(u,\ell)}=\ell+L$,
	and $\quo{(u,\ell)}$ denotes the image of $(u,\ell)\in  A(R/J)^{[e]}  \times L^f$
	in $Y$.
	Define $\psi:Y\to X$ by
	\begin{align*}
	\psi\quo{(u,\ell)}=(u+f(x,\ell),\ell+L).
	\end{align*}
	It is easy to check that $\psi$ is a well-defined and defines
	a morphism between the extensions. Moreover, $\psi$ is an isomorphism by the Five Lemma.
	This completes the proof of (iii), so we
	have established the theorem.
\end{proof}	

\begin{prp}\label{PR:comparison-with-BW}
	The complex $\aGW{R,\id,1}$ is isomorphic to the Gersten--Witt complex of $R$
	defined by Balmer and Walter in \cite{Balmer_2002_Gersten_Witt_complex}.
\end{prp}

\begin{proof}	
	The proposition follows from the fact that
	$\partial_{\frakp,\frakq}:\tilde{W}(k(\frakp))\to \tilde{W}(k(\frakq))$
	($\frakp\in R^{(e)}, \frakq\in R^{(e+1)}$)
	is equivalent to the corresponding map   in \cite[Proposition~8.5(b)]{Balmer_2002_Gersten_Witt_complex},
	a fact that we now prove.
	By localizing at $\frakq$, we may   assume that $R$ is local and $\frakq$ is its maximal ideal.

	Let $(V,f)\in\tHerm[1]{k(\frakp)}$. 
	Choose a regular sequence $r_1,\dots,r_e\in \frakp$
	(use \cite[Theorem~17.4(i)]{Matsumura_1989_commutative_ring_theory}),
	let $I=r_1 R+\dots+r_e R$ and define $S$ and $\tilde{S}$
	as in \ref{subsec:second-res}. By Lemma~\ref{LM:change-of-ideals}, 
	$\partial_{\frakp,\frakq}[f]=\partial_{I,\frakq}[\iota_{I,\frakq}\circ f]$, 
	so we may replace $f$ with $\iota_{I,\frakq}\circ f$.
	Choose any $A$-lattice $U$ in $V$ with $U\subseteq U^f$,   set $Z=\tilde{k}(I)/\tilde{S}$,
	and define $\partial_U f: U^f/U\times U^f/U\to Z$ as in the proof of Lemma~\ref{LM:residue-well-defined}.
	Let $\vphi$ denote the isomorphism $x\mapsto \partial_U f(x,-):U^f/U\to \Hom(U^f/U,Z)$.
	One readily checks that upon identifying $U^f$ and $U$ with $\Hom(U,\tilde{S})$ and
	$\Hom(U^f,\tilde{S})$ via $y\mapsto f(y,-)$, the map $\vphi$ fits into the  commutative   diagram
	\begin{align}\label{EQ:comparison-diagram}
	\xymatrix{
	U  \ar[r] \ar[dd]^{\wr} &
	\Hom(U,\tilde{S}) \ar[r] \ar@{=}[dd] &
	U^f/U \ar[d]^{\vphi}_{\wr}  \\
	& &	
	\Hom(U^f/U,Z) 
	\ar[d]^v_{\wr} \\
	\Hom(\Hom(U,\tilde{S}),\tilde{S}) \ar[r] &
	\Hom(U,\tilde{S}) \ar[r]^-u &
	\Ext^1(U^f/U,\tilde{S})
	}
	\end{align}
	in which $v$ is   induced by $\tilde{S}\embeds \tilde{k}(I)\onto Z$
	and $u$ is induced by   the top row. The commutativity is shown
	as in Step~6 of the proof of Theorem~\ref{TH:GW-is-a-complex}. 
	The map $v$ is an isomorphism because $\Hom(U^f/U,\tilde{k}(I))$
	and $\Ext^1(U^f/U,\tilde{k}(I))$ are both annihilated by some power of $\frakq$
	and $C$-divisible, hence $0$ ($\frakq\cap C\neq \emptyset$ by Lemma~\ref{LM:zero-divisors}).
	
	For every $i$-dimensional Cohen--Macuaulay $S$-module $M$, there
	is an canonical isomorphism  $\Ext^{i}(M,\tilde{S})\cong \Ext^{i+e}(M,R)$.
	To see this, let  $P_e\to \dots\to P_0 \to S$
	and $Q_{e+i}\to \dots\to Q_0\to M$ be  
	$R$-projective resolutions with $P_0=R$
	(their lengths  
	are justified by \cite[Theorem~1.3.3]{Bruns_1993_cohen_macaulay_rings}).
	By Proposition~\ref{PR:CM-properties-I}(ii),
	$P_0^\vee\to \dots\to P_e^\vee\to \tilde{S}$ is an $R$-projective resolution of
	$\tilde{S}$. Consider the induced double complex:
\newcommand{\xyvdots}{\raisebox{.006\baselineskip}{\ensuremath{\vdots}}} 
	\[
	\xymatrix@C=1.3em@R=1.2em{
	 & 0 \ar[d] &   & 0 \ar[d] & 0 \ar[d] & \\
	\cdots \ar[r] & 
	\Hom(Q_i,P_0^\vee) \ar[r] \ar[d]  & 
	\cdots \ar[r] & 
	\Hom(Q_{e+i},P_0^\vee) \ar[r]  \ar[d] &
	\Ext^{e+i}(M,P_0^\vee) \ar[d]^-{w} \ar[r] &
	0\\
	& \xyvdots \ar[d] & & \xyvdots \ar[d] & \xyvdots \ar[d] & \\
	\cdots \ar[r] & 
	\Hom(Q_i,P_e^\vee) \ar[r]\ar[d] &  
	\cdots \ar[r] &
	\Hom(Q_{e+i},P_e^{\vee}) \ar[r]\ar[d] &
	\Ext^{e+i}(M,P_e^{\vee}) \ar[d] \ar[r] &
	0\\
	\cdots \ar[r]^-{h}  &
	\Hom(Q_i,\tilde{S}) \ar[r]^-{h'} \ar[d] &
	\cdots \ar[r]  &
	\Hom(Q_{e+i},\tilde{S}) \ar[r] \ar[d] &
	\Ext^{e+i}(M,\tilde{S}) \ar[r] \ar[d] &
	0 \\
	& 0 & & 0 & 0 &
	}
	\]
	All rows except   the bottom one and all columns except the right one are
	exact.
	Standard diagram chasing therefore implies that 
	$\Ext^{i}(M,\tilde{S})=(\ker h')/(\im h)$
	is isomorphic to $\ker w$. By taking $P_\bullet$ to be the Koszul
	complex associated to the regular sequence $r_1,\dots,r_e$ 
	(Proposition~\ref{PR:Koszul-properties}),
	we see that $w$ is in fact $0$, because $MI=0$.
	Thus, $\ker w=\Ext^{e+i}(M,P_0^\vee)\cong \Ext^{e+i}(M,R)$
	and $\Ext^{i}_S(M,\tilde{S})\cong \Ext^{e+i}(M,R)$.

	Now, after replacing $f$ with $(-1)^{\frac{(e-1)(e-2)}{2}}f$, the diagram
	\eqref{EQ:comparison-diagram} (without the middle term on the right column)
	is naturally isomorphic to   diagram (35) in \cite{Balmer_2002_Gersten_Witt_complex} when 
	$U$ is a free $R/\frakp$-module.
	This means that, up to a sign depending on $e$, the map $\partial_{\frakp,\frakq}$ is equivalent to the
	map in \cite[Proposition~8.5(b)]{Balmer_2002_Gersten_Witt_complex}.
\end{proof}

\begin{remark}\label{RM:split-primes-do-not-contribute}
	Suppose that $(A,\sigma)$ is an Azumaya $R$-algebra
	with \emph{unitary} involution.
	Then $S:=\Cent(A)$ is a quadratic \'etale $R$-algebra
	and $\sigma|_S$ is the standard $R$-involution on $S$.
	If $\frakp\in R^{(e)}$  splits in $S$,
	then $S(\frakp)\cong k(\frakp)\times k(\frakp)$
	and   $(A(\frakp),\sigma(\frakp))\cong (B\times B^\op,(x,y^\op)\mapsto (y,x^\op))$
	for some central simple $k(\frakp)$-algebra $B$. 
	As a result, $\tilde{W}_\veps(A(\frakp) )=0$ 
	\cite[Example~2.4]{First_2022_octagon} and   $\frakp$
	does not contribute to $\GW{A,\sigma,\veps}_e$. 
	In particular, if $S=R\times R$, then $\aGW{A,\sigma,\veps}$ is the zero complex.
\end{remark}

We finish this section by showing that the isomorphism class of $\aGW{A/R,\sigma,\veps}$
is independent of $R$. This will usually be used to replace
$R$ with $\Cent(A)^{\{\sigma\}}$ in order to assume that 
$(A,\sigma)$ is Azumaya over $R$.

\begin{thm}\label{TH:base-does-not-matter}
	Let $R'$ be a regular ring and suppose that
	$A$ is equipped with an $R'$-algebra structure
	such  
	that $(A,\sigma)$ is a separable projective
	$R'$-algebra with involution.
	Then $\aGW{A/R,\sigma,\veps}\cong \aGW{A/R',\sigma,\veps}$.
\end{thm}

\begin{proof}
	Since the structure morphism $R'\to A$ factors through
	$\Cent(A)^{\{\sigma\}}$, it is enough to prove
	the theorem when $R'=\Cent(A)^{\{\sigma\}}$. In particular,
	we may assume that $R'$ is a finite \'etale $R$-algebra.
	Under this additional assumption, we shall define an isomorphism 
	$\psi=\psi_{R,R'}:\aGW{A/R,\sigma,\veps}\xrightarrow{{\sim}} \aGW{A/R',\sigma,\veps}$ as follows.
	
	Let $\psi_{-1}$ be the identity map $W_\veps(A,\sigma)\to W_\veps(A,\sigma)$.
	
	To define $\psi_e$ for $e\geq 0$, let $\phi:\Spec R'\to \Spec R$ denote the morphism corresponding to the 
	structure map $R\to R'$.
	Since $R'$ is finite \'etale over $R$ the fibers of $\phi$
	are finite, $\phi(R'^{(e)})\subseteq R^{(e)}$ for all $e\geq 0$,
	and  $\frakp R'=\bigcap_{\frakP\in \phi^{-1}(\frakp)} \frakP$ for
	all $\frakp\in R^{(e)}$ 
	\cite[Tags \href{https://stacks.math.columbia.edu/tag/00HS}{00HS},
	\href{https://stacks.math.columbia.edu/tag/00GT}{00GT}]{DeJong_2018_stacks_project}.
	Given $\frakP\in R'^{(e)}$,    write
	$k'(\frakP)=\mathrm{Frac}(R'/\frakP)$, $\tilde{k}(\frakP)=\Ext^{e}_{R'_{\frakP}}(k'(\frakP),R'_\frakP)$,
	$A(\frakP)= A \otimes_{R'} k'(\frakP) $ and $\tilde{A}(\frakP)=  A \otimes_{R'} \tilde{k}'(\frakP) $. 
	For every $\frakP\in R'^{(e)}$ with $\frakp:=\phi(\frakP)$, define
	$
	\psi_{\frakp,\frakP}:\tilde{W}_\veps(A(\frakp))\to \tilde{W}_\veps(A(\frakP) ) 
	$ by $\psi_{\frakp,\frakP}[V,f]=[V_\frakP,f_\frakP]$. Here,
	we have identified $\tilde{A}(\frakp)_\frakP= (A \otimes_{R'}  R'_{\frakP}) \otimes \tilde{k}(\frakp)$ 
	with $\tilde{A}(\frakP)$ using the canonical isomorphism
	\[
	 R'_\frakP \otimes \tilde{k}(\frakp) =
	R'_\frakP \otimes  \Ext^e_{R_\frakp}(k(\frakp),R_\frakp) \cong
	\Ext^e_{R'_\frakP}(k'(\frakP),R'_\frakP)=\tilde{k}'(\frakP).
	\]
	Finally, let
	\[
	\psi_e=\sum_{\frakp\in R^{(e)}}\sum_{\frakP\in \iota^{-1}(\frakp)} \psi_{ \frakp,\frakP}:
	\bigoplus_{\frakp\in R^{(e)}}\tilde{W}_\veps(A(\frakp))
	\to
	\bigoplus_{\frakP\in R'^{(e)}}\tilde{W}_\veps(A(\frakP)).
	\]

	In order to show that $\psi_e$ is an isomorphism, it is enough
	to check that the sum $\sum_{\frakP\in \phi^{-1}(\frakp)}\psi_{\frakp,\frakP}:
	\tilde{W}_\veps(A(\frakp))\to 
	\bigoplus_{\frakP\in \iota^{-1}(\frakp)}\tilde{W}_\veps(A(\frakP))$ is an isomorphism.
	This follows readily from the fact
	that $R'(\frakp)=\prod_{\frakP\in \phi^{-1}(\frakp)} k'(\frakP)$,
	which means that the natural map $V\to \prod_{\frakP\in\iota^{-1}(\frakp)} V_\frakP$
	is an isomorphism of $R'(\frakp)$-modules for every $R'(\frakp)$-module $V$,
	and likewise for $A(\frakp)$-modules.

	It remains to check that $\psi=(\psi_e)_{e\geq -1}$
	is compatible with the differentials
	of $\aGW{A/R,\sigma,\veps}$ and $\aGW{A/R',\sigma,\veps}$.
	Let $\dd'_e$ denote the $e$-th differential of $\aGW{A/R',\sigma,\veps}$.
	Then we need to check that $\dd'_e\circ \psi_e=\psi_{e+1}\circ \dd_e$
	for all $e\geq -1$. We leave the straightforward case where $e=-1$  to the reader
	and assume  $e\geq 0$. In this case,
	it is enough to show that for all $\frakp\in R^{(e)}$ and $\frakQ\in R^{(e+1)}$,
	we have $\sum_{\frakP} \partial_{\frakP,\frakQ}\circ \psi_{\frakp,\frakP}=
	\psi_{\frakq,\frakQ}\circ \partial_{\frakp,\frakq}$, where
	$\frakq=\iota(\frakQ)$ and the sum is taken over all $\frakP\in \iota^{-1}(\frakp)$
	contained in $\frakQ$. We assume that $\frakp\subseteq \frakq$ otherwise both sides
	are $0$. We may further assume that $R$ is local and $\frakq$ is its maximal ideal.
	
	Fix $\frakp\in R^{(e)}$ and $(V,f)\in \tHerm[\veps]{A(\frakp)}$.
	Write $I=\frakp R'_\frakQ$. Then $R'_\frakQ/I$ is Cohen-Macaulay 
	\cite[\href{https://stacks.math.columbia.edu/tag/025Q}{Tag 025Q}]{DeJong_2018_stacks_project},
	so we may consider $\partial_{I ,\frakQ}$.
	We claim that 
	\begin{align}\label{EQ:partial-I-frakQ} 
	\partial_{I,\frakQ} [f_\frakQ]= (\partial_{\frakp, \frakq} [f])_\frakQ.
	\end{align}
	Indeed, choose an $A$-lattice $U$ in $V$ with $U^f\frakq\subseteq U \subseteq U^f$.
	As in the proof of Lemma~\ref{LM:partial-well-defined},
	we have $(U_\frakQ)^{f_\frakQ}=(U^f)_\frakQ$.
	Since $U^f_\frakQ$ is an $R'_\frakQ/I$-module and
	$U^f_\frakQ \frakQ=U^f_\frakQ R'_\frakQ \frakq\subseteq U^f_\frakQ$,
	it follows that  $ \partial_{I,\frakQ} [f_\frakQ]$
	is represented by $(U^f_\frakQ/U_\frakQ,\partial f_\frakQ)=(U^f/U,\partial f)_{\frakQ}$,
	hence our claim.
	
	Let $\frakQ\in R'^{(e+1)}$ and let
	$T= \{\frakP\in \phi^{-1}(\frakp)\suchthat \frakP\subseteq \frakQ\}$. 
	Then the natural maps $V_\frakQ\to \prod_{\frakP\in T} V_\frakP$ and
	$\tilde{A}(\frakp)_\frakQ\to \prod_{\frakP\in T}\tilde{A}(\frakp)_\frakP$  are isomorphisms.
	Given $\frakP\in T$, the (unique)   $A$-module section
	of $\tilde{A}(\frakp)_\frakQ\to   \tilde{A}(\frakp)_\frakP$	 is just
	the map $\iota_{\frakP,I}$ of Lemma~\ref{LM:change-of-ideals} (applied to $R'$),
	so $(V_\frakQ,f_\frakQ)\cong \oplus_{\frakP\in T}(V_\frakQ,\iota_{\frakP,I}\circ f_{\frakQ})$.
	By part (iii) of that lemma, 
	$\partial_{I,\frakQ}[f_\frakQ]=\sum_{\frakP\in T}\partial_{\frakP,\frakQ}[f_\frakP]$.
	Thanks to \eqref{EQ:partial-I-frakQ},
	this means that $\psi_{\frakq,\frakQ}[f]=\sum_{\frakP\in T}\partial_{\frakP,\frakQ}\psi_{\frakp,\frakP}[f]$,
	which is what we need to show.
\end{proof}

\section{Computing Residue Maps}
\label{sec:complete-intersection}

Suppose $R$, $(A,\sigma)$ and $\veps$ are as in Section~\ref{sec:second-res}.
Let $e\in \N\cup\{0\}$, $\frakp\in R^{(e)}$, $\frakq\in R^{(e+1)}$
and assume $\frakp\subseteq \frakq$.
We now show that  the definition of $\partial_{\frakp,\frakq}:\tilde{W}_\veps(A(\frakp))\to \tilde{W}_\veps(A(\frakq))$,
see   \ref{subsec:dd-description}, can be made even
more explicit, to the extent of allowing hands-on
computations.
This will give an elementary method to compute the differentials of $\GW{A,\sigma,\veps}$,
which will be used in the sequel.

To that end, we may assume that $R$ is local and $\frakq$ is its maximal ideal.
In fact, we shall further require that $R/\frakp$ is a   complete intersection ring,
because the general case can be deduced from this special case, see Remark~\ref{RM:res-in-general}.

\medskip

Similarly to  \ref{subsec:second-res}, we shall work in greater
generality, replacing $\frakp$ with any ideal
$I$ of $ R$ such that $S:=R/I$ is a $1$-dimensional \emph{complete intersection} ring.
Define $C$, $\frakm$, $k(I)$, $\tilde{S}$, $\tilde{\frakm}^{-1}$ and $\tilde{k}(I)$ as in \ref{subsec:second-res}.
By Proposition~\ref{PR:Koszul-properties}, $I/I^2$ is a free $S$-module of rank
$e$, and 
\[\tilde{S}\cong \Hom({\textstyle\bigwedge^e} (I/I^2),S)\]
canonically. In particular,  $\tilde{S}$ is a free  $S$-module of rank $1$.
For  the sake of brevity, given a regular sequence $r_1,\dots,r_e\in R$ generating $I$
and some $s\in R $,
we shall   write
$[r_1\wedge\dots\wedge r_e\mapsto s]$ to denote
the unique element of $\tilde{S}$ 
sending $(r_1+I^2)\wedge
\dots\wedge (r_e+I^2)$ to $s+I$. The same convention will be applied to other complete intersection
ideals, e.g., $IC^{-1}\subseteq RC^{-1}$ and $\frakq\subseteq R$.

Let us fix a regular sequence 
$\alpha_1,\dots,\alpha_e \in R$ generating 
$I$
and a regular sequence $\beta_1,\dots,\beta_{e+1}\in R$ generating $\frakq$.
Then $\tilde{S}$ is   generated by $[\alpha_1\wedge\dots\wedge \alpha_e\mapsto 1]$.
Since $I \subseteq \frakq$, we can find 
elements $ \gamma_{ji} \in R$ 
($1\leq i\leq e$, $1\leq j\leq e+1$)
such that $\sum_j \beta_j\gamma_{ji}=\alpha_i$
for all $i$.
In addition, since $\frakq C^{-1}=RC^{-1}$
(Lemma~\ref{LM:zero-divisors}),
there exist $\xi_1,\dots,\xi_{e+1}\in RC^{-1}$
such that $\sum_j \xi_j\beta_j=1$.

\begin{prp}\label{TH:description-of-tildeM}
	Let $d\in RC^{-1}$ denote the determinant of the matrix
	\[
	\left[
	\begin{matrix}
	\xi_1 & \gamma_{11} & \cdots & \gamma_{1e} \\
	\xi_2 & \gamma_{21} & \cdots & \gamma_{2e} \\
	\vdots & \vdots & & \vdots \\
	\xi_{e+1} & \gamma_{(e+1) 1} & \cdots & \gamma_{(e+1) e}
	\end{matrix}
	\right]
	\]
	Then 
	$\tilde{\frakm}^{-1}=\tilde{S}+[\alpha_1\wedge \dots\wedge \alpha_n\mapsto d]S$
	and the map $\calT_{I,\frakq,R}:\tilde{\frakm}^{-1}\to \tilde{k}(\frakq)$
	(see \eqref{EQ:dual-m-S-k-A}) sends 
	$\tilde{S}$ to $0$ and $[\alpha_1\wedge \dots\wedge \alpha_n\mapsto d]$
	to $[\beta_1\wedge\dots \wedge \beta_{e+1}\mapsto 1]\in\tilde{k}(\frakq)$.
\end{prp}

\begin{proof}
	We work in the bounded derived category of  $\rMod{R}$, denoted $\calD$.
	Recall that $M^\vee$ denotes $\Hom(M,R)$.
	If $P\in\calD$, we write the $i$-th differential
	of $P$ as $d_i^P$,
	and set $P^\vee=(P_{-i},(d_{1-i}^P)^\vee)_{i\in\Z}$
	and $\Delta_n P=T^n (P^\vee)$.
	We define $P^{\vee'}$ and $\Delta'_n P$ similarly, by
	dualizing relative to $RC^{-1}$.

	Let $E$ be a free $R$-module
	with basis $x_1,\dots,x_e$ and let $s\in E^\vee$
	be the $R$-module homomorphism sending $x_i$ to $\alpha_i$ for all $i$.
	Let $F$ be a free $R$-module
	with basis $y_1,\dots,y_{e+1}$ and let $t\in F^\vee$
	be the $R$-module homomorphism sending $y_j$ to $\beta_j$ for all $j$.
	Let $L$ and $K$ denote the Koszul complexes of $(E,s)$
	and $(F,t)$, respectively, see~\ref{subsec:Koszul-complex}.
	Then $\HH_0(E)=R/I=S$ and $\HH_0(F)=R/\frakq=k(\frakq)$.
	
	Let $\gamma:E\to F$
	denote the $R$-homomorphism
	determined by $\gamma x_i=\sum_j  y_j \gamma_{ji}$.
	Then $t=s\circ \gamma$, hence $\gamma$ determines a morphism $L\to K$,
	which is also denoted $\gamma$. Explicitly, $\gamma_i:L_i\to K_i$
	is just $\gamma^{\wedge i}:\wedge^i E\to \wedge^i F$ (and $\gamma^{\wedge i}=0$
	if $i<0$).
	The induced map $\HH_0(\gamma):S\to k(\frakq)$
	is the   quotient map.
	
	Let $G$ denote the cone of $\gamma:L\to K$. Then we have
	a  distinguished triangle
	\[
	T^{-1}G\xrightarrow{u}
	L\xrightarrow{\gamma} K\xrightarrow{v} G
	\]
	in $\calD$, where 
	$G_i=E_{i-1}\oplus F_i$ (viewed as column vectors) and
	\[d_i^G=\begin{bmatrix}-d_{i-1}^L & 0 \\ \gamma^{\wedge (i-1)} & d_i^K \end{bmatrix},
	\qquad
	u_i=\begin{bmatrix}\id_{E_i} & 0\end{bmatrix},
	\qquad
	v_i=\begin{bmatrix} 0 \\ \id_{F_{i}}\end{bmatrix}.
	\]
	The cochain complex $T^{-1}G$ is quasi-isomorphic to a projective resolution of $\frakm$ 
	(the image of $\HH_0(u):\HH_0(T^{-1}G)\to \HH_0(L)=S$ is precisely $\frakm$),
	hence
	$\HH_0(\Delta_e T^{-1}G)=\HH_0(\Delta_{e+1}G)$
	can be identified canonically
	with $\tilde{\frakm}^{-1}=\Ext^e_R(\frakm,R)$.
	Moreover, the   maps 
	$\tilde{S}\to\tilde{\frakm}^{-1}$ and	
	$\calT_{I,\frakq}:\tilde{\frakm}^{-1}\to \tilde{k}(\frakq)$
	are just $\HH_0(\Delta_e u)$ and $\HH_0(\Delta_{e+1}v)$, respectively.
	
	Let $L'$ and $G'$ denote the localizations of $L$ and $G$ at $C$, respectively.
	Then $L'$
	is a projective resolution of the $RC^{-1}$-module
	$k(I)$.
	Let $c:=\sum_i y_i\xi_i\in FC^{-1}$. 
	For $i\geq 0$, define $c_i:\wedge^i EC^{-1} \to \wedge^{i+1} FC^{-1}$
	by $c_i(z)=c\wedge \gamma^{\wedge i} z$.
	Let $f_i:L_iC^{-1}_i\to G_{i+1}C^{-1}$ denote $[\begin{smallmatrix} \id \\ c_i\end{smallmatrix}]$
	if $0\leq i\leq e$ and the zero map otherwise. It is routine
	to check that $f=(f_i)_{i\in \Z}$ is a morphism from $L'$ to $T^{-1}G'$ 
	(use $s_{RC^{-1}}(c)=1$ and $\gamma^{\wedge (i-1)}d_i^L=d_i^K\gamma^{\wedge i}$).
	
	Let $u'=u_{RC^{-1}}:T^{-1}G'\to L'$.
	Then $u'$ is a quasi-isomorphism,  
	and since $u' \circ f=\id_{L'}$,
	we see that $\HH_0( \Delta'_e f)$
	is the the inverse of $\HH_0(\Delta'_e u'):\tilde{ S}C^{-1}\to \tilde{\frakm}^{-1}C^{-1}$.
	Thus, writing the natural morphism $\Delta_{e+1} G\to \Delta'_{e+1} G'$ as $\iota$,
	the image of $\HH_0(\Delta'_e f\circ \iota)$ in $\HH_0(\Delta'_e L')=\tilde{k}(I)$
	is the copy of $\tilde{\frakm}^{-1} $ in $\tilde{k}(I)$.

	The morphisms $\Delta'_ef \circ \iota$ and $\Delta_{e+1}v$
	are illustrated, in degrees $0$ and $1$, in the following diagram 
	(the top row is degree $1$ and the bottom row is degree $0$),
	\begin{align}
	\label{EQ:Koszul-diag}
	\xymatrix{
	(\wedge^e F)^\vee \ar[d]|{(-1)^{e+1}(d^K_{e+1})^\vee}
	&
	{\begin{array}{c}
	(\wedge^{e-1} E)^\vee \\[-0.15cm] \oplus \\[-0.05cm] (\wedge^{e} F)^\vee 
	\end{array}}
	\ar[d]^{(*)} \ar[l]_{[0~\id]} \ar[r]^{[\id~c_{e-1}^\vee]}
	&	
	(\wedge^{e-1} EC^{-1})^{\vee'} \ar[d]|{(-1)^e(d^{L'}_{e})^{\vee' } }
	\\
	(\wedge^{e+1} F)^\vee  
	&
	{\begin{array}{c}
	(\wedge^{e} E)^\vee \\[-0.15cm] \oplus \\[-0.05cm] (\wedge^{e+1} F)^\vee 
	\end{array}}
	\ar[l]_{[0~\id]} \ar[r]^{[\id~c_{e}^\vee]}
	&
	(\wedge^{e} EC^{-1})^{\vee'}
	}
	\end{align}
	where $(*)$ is $
	(-1)^{e+1 }\left[\begin{smallmatrix} -(d^E_e)^\vee &  (\gamma^{\wedge e})^\vee \\
	0 & (d_{e+1}^K)^\vee \end{smallmatrix}\right]$.
	As noted
	in the proof of Proposition~\ref{PR:Koszul-properties}(iii),
	the isomorphism  $\HH_0(\Delta'L')=\coker( d_e^{L'})^{\vee' } \cong \tilde{k}(I)$
	sends  $[x_1\wedge\dots \wedge x_e\mapsto 1]$
	to $[\alpha_1\wedge\dots \wedge \alpha_e\mapsto 1]$,
	and the isomorphism  $\HH_0(\Delta_{e+1} K)=\coker( d_{e+1}^K )^\vee\cong \tilde{k}(\frakq)$
	sends $[y_1\wedge\dots \wedge y_{e+1}\mapsto 1]$
	to $[\beta_1\wedge\dots \wedge \beta_{e+1}\mapsto 1]$.
	It is routine to check
	that  $c_{e+1}^\vee$ 
	maps $[y_1\wedge\dots \wedge y_{e+1}\mapsto 1]$
	to $[x_1\wedge \dots\wedge x_e\mapsto d]$.
	The theorem follows readily from this observation
	and the bottom row of \eqref{EQ:Koszul-diag}.
\end{proof}

	In the remainder of this section,
	we show how Proposition~\ref{TH:description-of-tildeM}
	can be applied to compute $\partial_{\frakp,\frakq}$ in various situations.
	Given $a_1,\dots,a_n\in  \tilde{A}(\frakp)$ with $\veps a_i^{\tilde{\sigma}(\frakp)}=a_i$
	for all $i$,
	we write $\langle a_1,\dots,a_n\rangle_{A(\frakp)}$ to
	denote the $(\tilde{A}(\frakp),\tilde{\sigma}(\frakp))$-valued $\veps$-hermitian form on $A(\frakp)^n$
	given by 
	$((x_i),(y_i))\mapsto \sum_i x^\sigma_ia_i y_i$.
	The Witt equivalence relation is denoted $\sim$.
	The ring $R$ is assumed to be regular local of dimension $e+1$ if not otherwise indicated.

	\begin{example}\label{EX:dd-description-toy-example}
		Suppose that $R/\frakp$ is regular, and hence a discrete valuation ring.
		Then there exists $\alpha_{e+1}\in R$
		such that $\alpha_1,\dots,\alpha_e,\alpha_{e+1}$ generate $\frakq$.
		Using this generating set, we can take $\gamma_{ji}=\delta_{ji}$,
		$\xi_1=\dots=\xi_e=0$ and $\xi_{e+1}=\alpha_{e+1}^{-1}$.		
		By Proposition~\ref{TH:description-of-tildeM},
		$\tilde{\frakm}^{-1}$ is the $S$-submodule of $\tilde{k}(\frakp)$
		generated
		by $[\alpha_1\wedge\dots\wedge \alpha_e\mapsto \alpha_{e+1}^{-1}]$
		(this module already contains $\tilde{S}$),
		and $\calT:\tilde{\frakm}^{-1}\to \tilde{k}(\frakq)$
		maps $[\alpha_1\wedge\dots\wedge \alpha_e\mapsto \alpha_{e+1}^{-1}]$
		to $(-1)^e [\alpha_1\wedge\dots\wedge   \alpha_{e+1}\mapsto 1]$.
		
		We apply this to describe $\partial_{\frakp,\frakq}$
		in  the case $(A,\sigma,\veps)=(R,\id_R,1)$. It is enough
		to evaluate $\partial_{\frakp,\frakq}f_b$,
		where $f_b:=\langle [\alpha_1\wedge\dots\wedge   \alpha_{e }\mapsto b] \rangle_{k(\frakp)}$
		and $b\in \units{R_\frakp}$.

		Since $S$ is a discrete valuation ring
		with uniformizer $\alpha_{e+1}(\frakp)$, we can write $b\equiv \alpha_{e+1}^n r \bmod \frakp_\frakp$ 
		with $n\in\Z$
		and $r\in \units{R}$. We claim that, up to Witt equivalence, 
	\begin{align*}
	\partial_{\frakp,\frakq}\langle [\alpha_1\wedge\dots \wedge \alpha_e\mapsto \alpha_{e+1}^nr]\rangle_{k(\frakp)}
	=\left\{
	\begin{array}{ll}
	0 & n\in 2\Z \\
	\langle (-1)^{e } [\alpha_1\wedge\dots \wedge \alpha_{e+1}\mapsto r]\rangle_{k(\frakq)}
	& n\notin 2\Z.
	\end{array}
	\right.
	\end{align*}
	Indeed, write $f:=f_b$,
	let $m=\lfloor \frac{n}{2}\rfloor$
	and consider the submodule $U=\alpha_{e+1}^{-m}(\frakp)S $
	of $V =k(\frakp)$. One readily checks
	that $U^f=\alpha_{e+1}^{-m}(\frakp)S=U$ if $n$ is even
	and $U^f=\alpha_{e+1}^{-m-1}(\frakp)S=\alpha_{e+1}^{-1}(\frakp)U$
	if $n$ is odd. Thus, $\partial f$ is the zero form when $n$ even.
	On the other hand, when $n$ is odd, $U^f/U$ is a one-dimensional
	$k(\frakq)$-vector space generated by the image of $\alpha^{-m-1}_{e+1}(\frakp)$,
	and we have
	\begin{align*}
	 {\partial f}(\alpha_{e+1}^{-m-1}(&\frakp)+U,\alpha_{e+1}^{-m-1}(\frakp)+U)=
	 \calT(f(\alpha_{e+1}^{-m-1}(\frakp),\alpha_{e+1}^{-m-1}(\frakp)))=\\
	&
	 \calT[\alpha_1\wedge\dots \wedge \alpha_e\mapsto \alpha_{e+1}^{n-2m-2} r]=
	(-1)^{e } [\alpha_1\wedge\dots \wedge \alpha_{e+1}\mapsto r],
	\end{align*}
	so $\partial_{\frakp,\frakq} f_b
	\sim \langle (-1)^{e } [\alpha_1\wedge\dots \wedge \alpha_{e+1}\mapsto r]\rangle_{k(\frakq)}$.

	To illustrate this formula, let $F$ be a field,
	let  
	$R=F[x,y ] $ and consider the ideals  $\fraka=xR$, $\frakb= yR$ and $\frakc=xR+yR$.
	Then
	\begin{align*}
	\partial_{0,\fraka}\langle xy\rangle_{k(0)} & \sim \langle  [x\mapsto y]\rangle_{k(\fraka)}, \\ 
	\partial_{0,\frakb}\langle xy\rangle_{k(0)} & \sim \langle  [y\mapsto x]\rangle_{k(\frakb)}, \\
	\partial_{\fraka,\frakc}\langle  [x\mapsto y]\rangle_{k(\fraka)} 
	& \sim \langle - [x\wedge y\mapsto 1]\rangle_{k(\frakc)},\\
	\partial_{\frakb,\frakc}\langle  [y\mapsto x]\rangle_{k(\frakb)} 
	& \sim \langle - [y\wedge x\mapsto 1]\rangle_{k(\frakc)}
	=\langle   [x\wedge y\mapsto 1]\rangle_{k(\frakc)}.
	\end{align*}
	One can similarly check
	that all other components of $\dd_0$, resp.\ $\dd_1$,
	vanish on $\langle xy\rangle_{k(0)}$,
	resp.\ 
	$\langle   [x\mapsto y]\rangle_{k(\fraka)}$ and
	$\langle  [y\mapsto x]\rangle_{k(\frakb)}$.
	Thus, $\dd_1\dd_0\langle xy\rangle_{k(0)}
	\sim \langle   -[x\wedge y\mapsto 1],   [x\wedge y\mapsto 1]\rangle_{k(\frakc)}
	\sim 0$, as one would expect.
	\end{example}

\begin{example}\label{EX:second-residue-II}
	We continue to assume that $R/\frakp$
	is regular as in Example~\ref{EX:dd-description-toy-example},
	and choose $\alpha_{e+1}\in R$
	such that $\alpha_1,\dots,\alpha_{e+1}$ generate $\frakq$.
	Recall that $(A,\sigma)$ is a separable projective $R$-algebra
	with involution.
	We write $\frakm^{-1} =\{r\in k(\frakp)\suchthat
	\frakm r\subseteq S\}$ and $A \frakm^{-1}=  A \otimes \frakm^{-1} =\{a\in A(\frakp)\suchthat
	a\frakm  \subseteq A_S\} $.
	
	Let us identify
	$\tilde{k}(\frakp)$ with $k(\frakp)$ by sending $[\alpha_1\wedge \dots \wedge \alpha_e\mapsto 1]$ to $1$
	and $\tilde{k}(\frakq)$ with $\frakm^{-1}/S$
	by sending $[\alpha_1\wedge \dots \wedge \alpha_{e+1}\mapsto 1]$ to $\alpha_{e+1}^{-1}+S$.
	Then $\tilde{S}$ and
	$\tilde{\frakm}^{-1}$ correspond to $S$ and $\frakm^{-1}$, respectively,
	and we have induced identifications
	$\tilde{A}_S=A_S$ and $\tilde{A}(\frakq)=A \otimes (\frakm^{-1}/S)  =A\frakm^{-1} /A_S$.
	By Example~\ref{EX:dd-description-toy-example},
	the map $\calT_A: A\frakm^{-1} \to A \frakm^{-1}/A_S$ of \eqref{EQ:dual-m-S-k-A} is just $(-1)^e$ times 
	the   the quotient map. 
	Note that the  identifications just made are independent of $\alpha_{e+1}$,
	so they are canonical  when $e=0$.

	Write $K=k(\frakp)$, $k=k(\frakq)$ and $\tilde{k}=\frakm^{-1}/S$.
	Then  
	$\tilde{W}_\veps(A(\frakp))=W_\veps(A_K,\sigma_K)$, $\tilde{W}_\veps(A(\frakq))=\tilde{W}_\veps(A_k,\sigma_k;
	\tilde{k})$ (notation as in Example~\ref{EX:herm-line-bundle}),
	and  $\partial_{\frakp,\frakq}:W_\veps(A_K,\sigma_K)\to\tilde{W}_\veps(A_k,\sigma_k;\tilde{k})$
	can be described as follows:
	Given $(V,f)\in \Herm[\veps]{A_K,\sigma_K}$,
	choose an $A$-lattice $U\subseteq V$
	such that $ U^f \frakm\subseteq U\subseteq U^f$,
	and define
	$\partial_{\frakp,\frakq}[V,f]=[U^f/U,\partial f]$,
	where $\partial f$ is given by ${\partial f}(x+U,y+U)=(-1)^{e }
	{f}(x,y)+A_S$ (here we identify
	$A \frakm^{-1}/A_S$ with $ A \otimes\tilde{k} $).
	
	Taking $(A,\sigma,\veps)=(R,\id_R,1)$   and fixing an isomorphism
	$\frakm^{-1}/S\cong k$, the induced map 
	$W (K)\to \tilde{W} (k)\cong W (k)$ 
	is (up to sign)  the usual \emph{second residue
	map}, see \cite[Definition 6.2.5]{Scharlau_1985_quadratic_and_hermitian_forms}, for instance.
	
	The previous paragraphs also imply that
	$(-1)^e\cdot \partial_{\frakp,\frakq}^A$ can be identified non-canonically
	with $\partial_{0,\frakm}^{A_S}$ when $S=R/\frakp$ is regular.
\end{example}

\begin{example}\label{EX:singular-residue}
	Let $F$ be a field,  let $R$ denote
	the localization of $F[x,y]$ at the ideal generated
	by $x$ and $y$, and let $\frakp=(x^2-y^5)R$ and $\frakq=xR+yR$.
	Then $R/\frakp$ is a complete intersection ring which is not regular.
	Suppose $(A,\sigma,\veps)=(R,\id_R,1)$.
	We apply Proposition~\ref{TH:description-of-tildeM}  
	to compute $\partial_{\frakp,\frakq}\langle [x^2-y^5\mapsto xy]\rangle_{k(\frakp)}$.

	To that end,
	we identify
	$\tilde{k}(\frakp)$ with $k(\frakp)$ by mapping $[x^2-y^5\mapsto 1]$ to $1$
	and $\tilde{k}(\frakq)$ with $k(\frakq)$ by mapping $[x\wedge y\mapsto 1]$ to $1$.
	It is also convenient to identify
	$k(\frakp)$ with $F(z)$ via $x\mapsto z^5$, $y\mapsto z^2$;
	the ring $S=R/\frakp$ corresponds to the localization of $F[z^2,z^5]$
	at the ideal $M$ generated by $\{z^2,z^5\}$.
	
	Take $\alpha_1=x^2-y^5$, $\beta_1=x$, $\beta_2=y$,
	$\gamma_{11}=x$, $\gamma_{21}=-y^4$, $\xi_1=x^{-1}$ and $\xi_2=0$.
	Then, under the previous identifications,  Proposition~\ref{TH:description-of-tildeM}
	asserts that $\tilde{\frakm}^{-1}=S+S(-y^4x^{-1})=S+Sz^3 =F[z^2,z^3]_M$,
	and $\calT:\tilde{\frakm}^{-1}\to k(\frakq)$
	maps $S$ to $0$ and $-z^3=-y^4x^{-1}(\frakp)$ to $1$.
	
	Writing $f=\langle xy\rangle_{k(\frakp)} =\langle z^7\rangle_{k(\frakp)}$
	and $U=z^{-1}F[z]_M$, it is routine to
	check that 
	$U^f=\{u\in k(z)\suchthat z^6F[z]u\subseteq S\}=z^{-2}F[z]_M$.
	Thus, $U^f/U$ is a simple $S$-module generated
	by $z^{-2}+U$.
	Since $ {\partial f}(z^{-2}+U,z^{-2}+U)= \calT(z^3)=-1$,
	we conclude that
	$\partial_{\frakp,\frakq}\langle xy\rangle_{k(\frakp)}\sim \langle -1\rangle_{k(\frakq)}$.
\end{example}

\begin{remark}
	\label{RM:res-in-general}
	Let $R$ be a regular local ring of dimension $e+1$,
	let   $\frakq$ denote its maximal ideal,
	and let   $\frakp\in R^{(e)}$. {\it We do not assume that $R/\frakp$ is a complete intersection ring.}
	Given $(V,f)\in\tHerm[\veps]{A(\frakp)}$,
	we can use Proposition~\ref{TH:description-of-tildeM} to 
	describe $\partial_{\frakp,\frakq}[V,f]$ as follows. 	
	By \cite[Theorem~2.12(b)]{Bruns_1993_cohen_macaulay_rings}, $\frakp$ contains a regular
	sequence $\alpha_1,\dots,\alpha_e$. Write $I=\alpha_1R+\dots+\alpha_eR$.
	Then $R/I$ is a complete intersection ring of dimension $1$.
	Thanks to  Lemma~\ref{LM:change-of-ideals}(iii), we have
	\[\partial_{\frakp,\frakq}[V,f]=\partial_{I,\frakq}[V,\iota_{\frakp, I}\circ f] ,\]
	and the right hand side can be computed using Proposition~\ref{TH:description-of-tildeM}.
	
	This approach requires a description of $\iota_{\frakp,I}:\tilde{k}(\frakp)\to \tilde{k}(I)$, 
	which is given as follows.
	Defining $C$ as in \ref{subsec:second-res},
	let $\pi_1,\dots,\pi_e\in \frakp C^{-1}$ be a regular sequence generating $\frakp C^{-1}$ in $R C^{-1}$.
	Since $I\subseteq \frakp$, we can write $\alpha_i=\sum_j \pi_j u_{ji}$ for some 
	$\{u_{ji}\}_{j,i}\subseteq RC^{-1}$.
	Then $\iota_{\frakp, I}$ is determined by  
	\[\iota_{\frakp, I}[\pi_1\wedge\dots\wedge \pi_e\mapsto 1] =
	\det(u_{ij})\cdot [\alpha_1\wedge \dots\wedge \alpha_e\mapsto 1].\]
	Indeed, let $s:R^e\to R$, $t:R^e\to R$ and $u:R^e\to R^e$ be given by
	$s(r_1,\dots,r_e)=\sum_i \alpha_i r_i$, $t(r_1,\dots,r_e)=\sum_i \pi_i r_i$
	and $u(r_1,\dots,r_e)=(\sum_i u_{1i} r_i,\dots, \sum_i u_{ei} r_i)$.
	As in the proof of Proposition~\ref{TH:description-of-tildeM}, 
	$u$ determines a morphism $u =(u^{\wedge i})_{i\geq 0}: K(R^n,s)\to K(R^n,t)$
	such that $\HH^0(u)$ is the quotient map $RC^{-1}/IC^{-1}\to RC^{-1}/\frakp C^{-1}$.
	The assertion now follows  by examining $(u^{\wedge e})^\vee: (K(R^e,t)_e)^\vee\to
	(K(R^e,s)_e)^\vee$ and using
	Proposition~\ref{PR:Koszul-properties}(iii).
\end{remark}

\section{Functoriality}
\label{sec:functoriality}

In this section, we prove
that  the Gersten--Witt complex (see Section~\ref{sec:second-res}) is 
compatible with 
a number of natural operations, e.g.\   base change of $(A,\sigma)$.
We stress that the elementary proofs are possible thanks to the new 
construction of  the Gersten--Witt complex in Section~\ref{sec:second-res}.

Throughout, $(A,\sigma)$
and $(B,\tau)$
denote $R$-algebras with involution, $\veps \in \mu_2(R)$,
and $M$ is an invertible $R$-module. 
Ultimately, we will specialize to
the case  where $R$ is regular and $A$ and $B$ are separable projective over $R$.
Recall \cite[\S2B]{Lam_1999_lectures_on_modules_rings}
that every 
$V\in\rproj{A}$
admits a finite dual basis, i.e.,
a collection $\{(x_i,\phi_i)\}_{i=1}^n\subseteq V\times \Hom_A(V,A)$
such that
$x=\sum_i x_i\cdot \phi_ix$ for all $x\in V$.

\begin{remark}
Regard $\rproj{A}$ and $\rproj{B}$ as exact hermitian categories
using the involutions $\sigma$ and $\tau$
(see~\ref{subsec:herm-forms} and Example~\ref{EX:herm-ring-with-inv}).
The operations
that we consider in this section --- base change, involution trace and $e$-transfer ---
are induced by exact hermitian functors between 
$\rproj{A}$ and $\rproj{B}$ in the sense of Balmer \cite{Balmer_2005_Witt_groups}.
These functors also respect the $R$-codimension filtration on the bounded derived categories
of $\rproj{A}$ and $\rproj{B}$,
and therefore induce a morphism between the   Gresten--Witt complexes
of $A$ and $B$ \`a la Gille
\cite{Gille_2007_hermitian_GW_complex_I},
\cite{Gille_2009_hermitian_GW_complex_II}.
This morphism agrees with the respective operation in degrees $-1$ and $0$, but
it is  {\it a~priori} not  clear that the same 
holds in degrees $\geq 1$. The compatibility in all degrees, which
we establish in this section for the Gersten--Witt complex
of Section~\ref{sec:second-res}, will be important later on.
\end{remark}

\subsection{Base Change}
\label{subsec:base-change}

Let  
$\rho:(A,\sigma)\to (B,\tau)$
be a morphism of $R$-algebras with involution.
There is a \emph{base change  along $\rho$}
functor $\rho_* : \Herm[\veps]{A,\sigma;M}\to \Herm[\veps]{B,\tau;M}$
(notation as in Example~\ref{EX:herm-line-bundle})
given on objects by $\rho_*(V,f)=(V\otimes_A B, \rho f)$,
where $\rho f$ is determined by
\[
\rho f(x\otimes b,x'\otimes b')= b^\tau\cdot (\rho\otimes \id_M)(f(x,x'))\cdot b'
\qquad
(x,x'\in V,\, b,b'\in B),
\]
and on morphisms by $\rho_*\vphi=\vphi\otimes_A \id_B$. The hermitian form
$\rho f $ is unimodular by the following lemma,
which also tells us that that $\rho$ induces a group homomorphism
$ W_\veps(A,\sigma;M)\to W_\veps(B,\tau;M)$. The latter is also denoted $\rho_*$.

\begin{lem}\label{LM:base-change-well-def}
	In the previous notation, $\rho_*(V,f)\in \Herm[\veps]{B,\tau}$.
	If $(V,f)$ is metabolic, then so is $\rho_* (V,f)$.
\end{lem}

\begin{proof}
	The first statement follows by a simple adaptation of the argument in
	\cite[I.7]{Knus_1991_quadratic_hermitian_forms}.
	If $L\subseteq V$ is an $A$-submodule satisfying $L=L^\perp$ and $V/L\in\rproj{A}$, then there exists
	another $A$-submodule $L'$
	such that  $V=L\oplus L'$ and $L'=L'^\perp$ \cite[Proposition I.3.7.1]{Knus_1991_quadratic_hermitian_forms}. 
	It is easy to see
	that $\rho f(L\otimes_A B,L\otimes_A B)=\rho f(L'\otimes_A B,L'\otimes_A B)=0$.
	Since $V\otimes_AB=(L\otimes_A B)\oplus (L'\otimes_A B)$, the unimodularity of 
	$\rho f$ forces $(L\otimes B)^{\perp(\rho f)}=L\otimes B$, so $\rho f$ is metabolic.
\end{proof}

\begin{thm}\label{TH:base-change-is-compatible-with-GW}
	Suppose that $R$
	is regular 
	and that $A$ and $B$
	are separable  projective  $R$-algebras.
	Then $\rho :(A,\sigma)\to (B,\tau)$ induces
	a morphism of cochain complexes
	$\rho=(\rho_i)_{i\in \Z}:\aGW{A,\sigma,\veps}\to \aGW{B,\tau,\veps}$
	with  $\rho_{-1}=\rho_*$ 
	and   $\rho_e=\bigoplus_{\frakp\in R^{(e)}} {\rho}(\frakp)_*$.
\end{thm}

\begin{proof}
	Let $\dd_e$ and $\dd'_e$
	denote the $e$-th differentials
	of $\aGW{A,\sigma,\veps}$ and $\aGW{B,\tau,\veps}$,
	respectively.
	We need to show that
	$\dd'_e  \circ \rho_e=\rho_{e+1} \circ \dd_e$
	for all $e\geq -1$. The case $e=-1$
	is clear because $\rho_*$ is compatible with localization.
	Proving the case   $e\geq 0$
	amounts to showing that $
	\partial_{\frakp,\frakq}^B \circ  {\rho}(\frakp)_*
	= {\rho}(\frakq)_* \circ  \, \partial_{ \frakp,\frakq}^A $
	for all $\frakp\in R^{(e)}$ and $\frakq\in R^{(e+1)}$
	with $\frakp\subseteq\frakq$.
	To that end,
	we may assume
	that $R$ is local and $\frakq$ is its maximal ideal.
	We then apply the notation of \ref{subsec:second-res}, taking $I=\frakp$.
	
	Let $(V,f)\in \tHerm[\veps]{A(\frakp) }$
	and let $U$ be an $A $-lattice in $V$
	such that $ U^f \frakm \subseteq U\subseteq U^f$.
	Since $A$ is separable over $R$, the algebra $B$ is finite projective as a left (or right) 
	$A$-module
	(see \ref{subsec:Azumaya}).
	Thus,
	we may and shall view $U\otimes_A B $ and $U^f\otimes_A B $
	as submodules of $V\otimes_A B=V\otimes_{A(\frakp)} B(\frakp)$.
	Write $g= {\rho}(\frakp) f$.
	Provided that $(U\otimes_A B)^{g}=U^f\otimes_A B$,
	it is easy to see that
	the natural isomorphism $(U^f\otimes_A B)/(U\otimes_A B)\to (U^f/U)\otimes_{A(\frakq)} B(\frakq)$
	is an isometry
	from $\partial g$ to $ {\rho}(\frakq)(\partial f)$, 
	which is exactly what we want.
	We finish by showing that $(U\otimes_A B)^{g}=U^f\otimes_A B$.
	
	That $(U\otimes_A B)^{  g}\supseteq U^f\otimes_A B$ is straightforward,
	so  we turn to show the converse.
	We noted earlier that $B$ is finite projective as a right $A$-module,
	so it   has a finite dual basis $\{(b_i,\phi_i)\}_{i=1}^n\subseteq B\times \Hom_A(B,A)$.
	For all $i\in \{1,\dots,n\}$, let $b'_i=b_i(\frakp)\in B(\frakp)$
	and
	$\phi'_i=  \phi_i \otimes \id_{\tilde{k}(\frakp)} \in \Hom_A(\tilde{B}(\frakp),\tilde{A}(\frakp))$.
	Then $\phi'_i(\tilde{B}_S)\subseteq\tilde{A}_S$ and $b=\sum_i b'_i\cdot \phi'_i b$
	for all $b\in \tilde{B}(\frakp)$.
	Let $v\in (U\otimes_A B)^{g}$.
	Since $f$ is unimodular, for all $i\in \{1,\dots,n\}$,
	there exists $v_i\in V$
	such that $  f(v_i,x)=\phi'_i(g(v,x\otimes 1))$ for all $x\in V$.
	Since $ {g}(v,U\otimes 1)\subseteq \tilde{B}_S$, we have $\phi'_i(  g(v,U\otimes 1))\subseteq \tilde{A}_S$,
	hence  $v_i\in U^{  f}$.
	Now, for all $x\in V$ and $b\in B(\frakp)$, 
	we have $ {g}(v ,x \otimes b)=\sum_i b'_i \cdot \phi'_i(   g(v,x\otimes 1)) b=
	\sum_i b'_i \cdot   f(v_i,x)b=   g(\sum_i v_i \otimes b'^\tau_i, x \otimes b)$,
	so $v=\sum_i v_i \otimes b'^\tau_i\in U^f\otimes_A B$.
\end{proof}

\subsection{Involution Traces}
\label{subsec:Scharlau}

Throughout this subsection,
we assume that $B$
is an $R$-subalgebra of $A$ ($\sigma|_B\neq \tau$ is possible), and $A$ is a finite projective right 
$B$-module; the latter is automatic if $A$ and $B$ are separable projective over $R$ 
(see~\ref{subsec:Azumaya}). We further let $\gamma\in \mu_2(R)$.

\medskip

A function $\pi:A\to B$ is called an \emph{involution $\gamma$-trace}
(relative to $\sigma$ and $\tau$) if:
\begin{enumerate}[label=(T\arabic*)]
	\item  $\pi$ is additive and $\pi(b_1^\sigma a b_2)=b_1^\tau \pi(a)b_2$
	for all $b_1,b_2\in B$, $a\in A$;
	\item  $\pi\circ \sigma=\gamma \tau\circ \pi$; 
	\item \label{item:inv-trans:regularity} the map $a\mapsto [x\mapsto \pi(ax)]:A\to \Hom_B(A ,B )$
	is an isomorphism.
\end{enumerate}
Note that condition \ref{item:inv-trans:regularity} is 
equivalent to the unimodularity
of the $\gamma$-hermitian form $(a,a')\mapsto \pi(a^\sigma a'):A\times A\to B$
over $(B,\tau)$.
Under the additional assumptions $\tau=\sigma|_A$ and $\gamma=1$, the map $\pi$
is an involution trace in the sense of  \cite[I.7.2.4]{Knus_1991_quadratic_hermitian_forms}.

Every involution $\gamma$-trace $\pi:A\to B$ induces a functor
$\pi_*:\Herm[\veps]{A,\sigma;M}\to \Herm[\gamma\veps]{B,\tau;M}$
given on objects by $\pi_*(V,f)=(V_B,\pi f)$,
where $\pi f= (\pi\otimes \id_M) \circ f$, and by
$\pi_*\vphi=\vphi$ on morphisms. This is well-defined by the following lemma,
which also tells us that we have an induced
group homomorphism $\pi_*:W_\veps(A,\sigma;M)\to W_{\gamma\veps}(B,\tau; M)$.

\begin{lem}
	In the previous notation, $\pi_*(V, f)\in \Herm[\gamma\veps]{B,\tau;M}$.
	If $(V,f)$ is metabolic, then so is $\pi_*(V, f)$.
\end{lem}

\begin{proof}
	We have $V_B\in \rproj{B}$ because $A_B\in\rproj{B}$ and
	$V$ is an $A$-module summand of $A^n $ for some $n$. That $\pi f$ is a $(B\otimes M,\tau \otimes \id_M)$-valued 
	$\gamma\veps$-hermitian form over $(B,\tau)$ follows readily from 
	(T1) and (T2).
	
	Write $\hat{f}$ for $x\mapsto f(x,-):V\to \Hom_A^\sigma(V,A)$,
	and define $\what{\pi f}:V_B\to \Hom^\tau_B(V,B)$ similarly.
	Then $ \what{\pi f}  = j_V \circ \hat{f}$,
	where $j_V:\Hom_A(V,A\otimes M)\to \Hom_B(V_B,B\otimes M)$ is given by $j_V(\phi)=(\pi\otimes \id_M)\circ \phi$.
	Thus, in order to show that $\pi f$ is unimodular, it is enough to show that $j_V$ is an isomorphism.
	It is routine to check that $j_V$ is natural in $V$.
	Since  $V$ is a summand of $(A\otimes M)^n$ for some $n\in \N$, it is enough
	to consider the case $V=A\otimes M$.
	Since the map $\phi\mapsto \phi\otimes \id_M: \Hom_A(U,W)\to \Hom_A(U\otimes M,W\otimes M)$ is an isomorphism for 
	all $U,W\in\rproj{A}$ 	(\cite[Theorem~1.3.26]{Ford_2017_separable_algebras}, $\End(M)=R$),	 and likewise
	for $B$-modules, we are reduced  to
	proving that $\phi\mapsto \pi\circ \phi:\Hom_A(A,A)\to \Hom_B(A_B,B)$
	is an isomorphism. This holds because the composition of this map with the isomorphism $A\to \End_A(A)$ given
	by sending $a\in A$ to left multiplication by $a$ is the map considered in (T3).
	
	The second assertion is shown exactly as in the proof of Lemma~\ref{LM:base-change-well-def}.
\end{proof}

If $S$ is an $R$-ring, then $\pi_S:A_S\to B_S$
is also 
a $\gamma$-involution trace
relative to $\sigma_S$ and $\tau_S$ (use \cite[Corollary~1.3.27]{Ford_2017_separable_algebras}
to check (T3)).

\begin{thm}\label{TH:pi-transfer-is-compatible}
	Keeping the previous notation, suppose that $R$ is regular  
	and   $A$ and $B$ are separable projective over $R$.
	Then $\pi$ induces a morphism of cochain complexes
	$\pi=(\pi_i)_{i\in \Z}:\aGW{A,\sigma,\veps}\to \aGW{B,\tau,\gamma\veps}$
	with $\pi_{-1}=\pi_*$
	and $\pi_e=\bigoplus_{\frakp\in R^{(e)}} {\pi}(\frakp)_*$
	for $e\geq 0$.
\end{thm}

\begin{proof}
	As in the proof of Theorem~\ref{TH:base-change-is-compatible-with-GW},
	we may assume that $R$ is local of dimension $e+1$ with maximal ideal
	$\frakq$, and the proof
	reduces to showing that $ {\pi}(\frakq)_*\circ \partial^A_{\frakp,\frakq}=
	\partial^B_{\frakp,\frakq}\circ  {\pi}(\frakp)_*$
	for all $\frakp\in R^{(e)}$. We use the notation of \ref{subsec:second-res} with $I=\frakp$.

	Let $(V,f)\in \tHerm[\veps]{A(\frakp) }$
	and let $g= {\pi}(\frakq) f$.
	Choose an $A $-lattice $U$ in $V$ with $ U^f \frakm \subseteq U\subseteq U^f$.
	Provided that $U^g=U^f$, it is easy to see
	that $ \partial g = {\pi}(\frakq)(\partial f)$, which would finish the proof.
	
	It is clear that $U^f\subseteq U^g$.
	To see the converse, let $\{(a_i,\phi_i)\}_{i=1}^n$
	be a dual basis for $A_B$. By~\ref{item:inv-trans:regularity},
	there exist $\{c_i\}_{i=1}^n\subseteq A$ such that $\pi(c_i x)=\phi_i x$
	for all $x\in A$, hence $\sum_i a_i\pi(c_i x)= x$.
	Let $\tilde{\pi}=  \pi\otimes \id_{\tilde{k}(\frakp)}$.
	Then $\sum_i a_i \tilde{\pi}(c_i x)=x$ for all $x\in \tilde{A}(\frakp)$.
	Now, if $v\in U^g$, then for all $x\in U$,
	we have $  f(x,v)=\sum_i a_i \tilde{\pi}(c_i   f(x,v))=
	\sum_i a_i \tilde{\pi} (  f(x c_i^\sigma,v))=
	\sum_i a_i   g(x c_i^\sigma,v) \in \sum_i a_i \tilde{B}_S\subseteq \tilde{A}_S$,
	so $v\in U^f$.
\end{proof}

\begin{example}\label{EX:conjugation-is-compatible-with-GW}
	Let $u\in \units{A}\cap \Sym_{\gamma}(A,\sigma)$
	and let $\Int(u)$ denote the inner automorphism
	$a\mapsto uau^{-1}:A\to A$.
	Then $\tau:=\Int(u)\circ \sigma$ is an involution
	and the map $\pi_u:A\to A$ given by $\pi_u(a)=ua$
	is an involution $\gamma$-trace relative to $\sigma$ and $\tau$.
	We write $(\pi_u)_*$   as $u_*$, and $\pi_u f$ as $u f$ when $(V,f)\in \Herm[\veps]{A,\sigma;M}$.
	The functor $u_*:\Herm[\veps]{A,\sigma;M}\to
	\Herm[\gamma\veps]{A,\tau;M}$ is called \emph{$u$-conjugation}.
	It is clearly an equivalence, the inverse being $(u^{-1})_*$,
	so the induced map on the corresponding Witt groups is an isomorphism.

	By Theorem~\ref{TH:pi-transfer-is-compatible}, if $R$ is regular
	and $A$ is separable projective over $R$, then $u$-conjugation
	induces an isomorphism $u_*:\aGW{A,\sigma,\veps}\to \aGW{A,\Int(u)\circ \sigma,\gamma\veps}$.
\end{example}

\subsection{Hermitian Morita Equivalence}
\label{subsec:herm-Morita}

We show that the Gersten--Witt complex is compatible with
a special kind of hermitian Morita equivalence.

\medskip

Let $e\in A$ be an idempotent satisfying $e^\sigma=e$ and $AeA=A$.
Put $A_e=eAe$.
Then $\sigma$ restricts to an involution on $\sigma_e:A_e\to A_e$.
Since $A_e$ is an $R$-summand of $A$, we may and shall
regard $A_e\otimes M$ as a subset of $A\otimes M$.

Following  \cite[\S2.7]{First_2022_octagon} and \cite[Proposition~2.5]{First_2015_Witts_extension_theorem},
define the \emph{$e$-transfer} functor $e_*:\Herm[\veps]{A,\sigma;M}\to \Herm[\veps]{A_e,\sigma_e;M}$
by setting $e_*(V,f)=(Ve,f_e:=f|_{Ve\times Ve})$ for objects and $e_*\vphi = \vphi|_{Ve}$
for every morphism $\vphi:(V,f)\to (V',f')$. 

\begin{lem}
	$e_*:\Herm[\veps]{A,\sigma;M}\to \Herm[\veps]{A_e,\sigma_e;M}$ is an equivalence of categories
	taking metabolic spaces to metablic spaces.
\end{lem}

\begin{proof}
	Since we assume $2\in\units{R}$, the
	proof of \cite[Proposition~2.5]{First_2015_Witts_extension_theorem} applies
	to our situation with the following modification:
	Replace $\phi$ in {\it op.\ cit.} with 
	the natural transformation 
	$i_V:  \Hom^\sigma_A(V,A\otimes M)\cdot e=\Hom_A(V,eA\otimes M)\to \Hom^{\sigma_e}_{A_e}(Ve,A_e\otimes M)$
	given by $i_V(\psi)=\psi|_{Ve}$; this  is an isomorphism by Morita Theory 
	\cite[Example~18.30]{Lam_1999_lectures_on_modules_rings}.
\end{proof}

The  isomorphism $W_\veps(A,\sigma;M)\to W_\veps(A_e,\sigma_e;M)$ induced by $e_*$
will also be denoted   $e_*$.

\begin{thm}\label{TH:e-transfer-is-compatible-GW}
	In the previous notation, 
	suppose that $R$ is regular 
	and $A$ is  separable projective over $R$.
	Then $e$-transfer 
	induces an isomorphism
	of cochain complexes $e=(e_i)_{i\in \Z}:\aGW{A,\sigma,\veps}\to \aGW{A_e,\sigma_e,\veps}$
	with $e_{-1}=e_*$ and $e_{\ell}=\bigoplus_{\frakp\in R^{(\ell)}}  {e}(\frakp)_*$
	for $\ell\geq 0$.	
\end{thm}

\begin{proof}
	As in the proof of Theorem~\ref{TH:base-change-is-compatible-with-GW},
	we may assume that $R$ is local of dimension $\ell+1$ with maximal ideal
	$\frakq$, and the proof
	reduces to showing that $ {e}(\frakq)_*\circ \partial^A_{\frakp,\frakq}=
	\partial^{eAe}_{\frakp,\frakq}\circ  {e}(\frakp)_*$
	for all $\frakp\in R^{(\ell)}$. We use the notation of \ref{subsec:second-res} with $I=\frakp$.
	
	Let $(V,f)\in \tHerm[\veps]{A(\frakp) }$
	and let $U$ be an $A $-lattice in $V$
	such that $ U^f \frakm\subseteq U\subseteq U^f$.
	Then $Ue$ is an $ A_e $-lattice in $Ve$.
	Write $g=f_e$. Provided that $(Ue)^g=U^fe$, the natural
	map $(U^f/U)e\to U^fe /Ue$
	is an isometry
	from $(\partial f)_e$ to $\partial g$, and therefore $ {e}(\frakq)_*\partial^A_{\frakp,\frakq} [V,f]=
	\partial^{eAe}_{\frakp,\frakq} [Ve,f_e]$.
	
	That $U^f e\subseteq (Ue)^g$ is straightforward.
	Conversely, if $v\in (Ue)^g$,
	then $f(U,v)=f(UAeA,v)=A\cdot f(Ue,v)=A\cdot g(Ue,v)\subseteq \tilde{A}_S$,
	so $v\in U^f\cap Ve=U^fe$.
\end{proof}

\section{Surjectivity of The Last Differential}
\label{sec:surjectivity}

Throughout this section,  $R$ is a regular 
ring,  $(A,\sigma)$ is an Azumaya
$R$-algebra  with involution and  $\veps\in \mu_2(R)$.
We show that
some   cohomologies of $\aGW{A,\sigma,\veps}$
vanish under certain assumptions.
Our first and main result of this kind is:

\begin{thm}\label{TH:surj-at-last-term}
	If $R$ is semilocal of dimension $d$, then $\HH^d(\aGW{A,\sigma,\veps})=0$.
\end{thm}

As in Section~\ref{sec:functoriality},
the new construction of $\aGW{A,\sigma,\veps}$
in Section~\ref{sec:second-res} enables us
to   give a proof based on  classical methods.
A different proof   was given independently by Gille in \cite[Theorem~8.4]{Gille_2020_hermitian_Springer}.

\medskip

We begin with  the following well-known lemma.

\begin{lem}\label{LM:nice-prime-DVR-quotient}
	Let $R$ be a regular semilocal ring of dimension $d>1$
	and let $\frakq\in R^{(d)}$.
	Then there exists $\frakp\in R^{(d-1)}$
	contained in $\frakq$  such that $R/\frakp$
	is a discrete valuation ring.
	In particular, $\frakq$ is the only height-$d$ prime ideal
	containing $\frakp$.
\end{lem}

%

\begin{lem}\label{LM:primitive-invariant-idempotent}
	Let $(A,\sigma)$ be a central simple algebra
	with involution over a field $F$ such that $\Cent(A)\ncong F\times F$.
	If $\sigma$ is symplectic, we also require that $\ind A$ is even.
	Then every unimodular $1$-hermitian space
	over $(A,\sigma)$ (see Example~\ref{EX:herm-ring-with-inv}) is isomorphic to an orthogonal 
	sum of $1$-hermitian spaces of $A$-length $1$.
\end{lem}

\begin{proof}
By \cite[Proposition~1.27]{First_2022_octagon},
there exists a primitive idempotent $u\in A$ with $u^\sigma=u$.
Let $B=u Au$
and let $\tau=\sigma|_B$. Since $A$ is a simple artinian ring,
$B$ is a division ring
and  $Au A=A$. 
As explained in \ref{subsec:herm-Morita}, 
we have an equivalence of categories $u_*:\Herm[1]{A,\sigma}\to \Herm[1]{B,\tau}$,
which is easily seen to respect orthogonal sums and preserve length.
It is therefore enough to prove
the claim when $A$ is a division ring.
This this is well-known, see
\cite[Theorem~7.6.3]{Scharlau_1985_quadratic_and_hermitian_forms},
for instance. 
\end{proof}

\begin{proof}[Proof of Theorem~\ref{TH:surj-at-last-term}]

\Step{1} 
Since $R$ is regular semilocal, it is a finite product of regular domains.
By working over each factor separately,
we may assume that $R$ is a domain.

Let $d=\dim R$.
The theorem is clear if $d=0$, so assume $d>0$.

The case $d>1$ can be reduced
to the   case $d=1$ as follows: 
By Lemma~\ref{LM:nice-prime-DVR-quotient}, for every
$\frakq\in R^{(d)}$, there is $\frakp\in R^{(d-1)}$ such that $R/\frakp$ 
is a discrete valuation ring and $\frakq$ is the only height-$d$ prime containing $\frakp$.
As a result, the theorem will follow if we verify
that $\partial_{\frakp,\frakq}:\tilde{W}(A(\frakp))\to \tilde{W}(A(\frakq))$
is surjective for all  such $\frakp$ and $\frakq$. 
By the last paragraph of Example~\ref{EX:second-residue-II},
we may replace $R$, $\frakp$, $\frakq$, $A$ with $R/\frakp$, $\frakp/\frakp$, $\frakq/\frakp$, $A_{R/\frakp}$ 
and   reduce into proving the surjectivity
of $\partial_{0,\frakm}: W(A_K)\to \tilde{W}(A(\frakm))$
when
$R$ is a discrete valuation ring with maximal ideal $\frakm$
and fraction field $K$.

Note that the case $d=1$ cannot be similarly reduced to the case where
$R$ is a discrete valuation ring.

To conclude the previous discussion, we may
assume that $R$ is a semilocal Dedekind
domain. Write $K=\mathrm{Frac}(R)$.
For $\frakm\in \Max R$, we abbreviate $\partial_{0,\frakm}$
to $\partial_\frakm$. We identify
$\tilde{k}(\frakm)$ with $\frakm_\frakm^{-1}/R_\frakm\cong \frakm^{-1}/R$
and $\tilde{A}(\frakm)$ with $A_\frakm \frakm_\frakm^{-1} /A_\frakm\cong A\frakm^{-1} /A$
as in Example~\ref{EX:second-residue-II}.
The map   $\calT =\calT_{0,\frakm,A} $ is  just    
the quotient map $A_\frakm \frakm^{-1} \to A_\frakm \frakm^{-1} /A_\frakm$.

\medskip

\Step{2} Fix some   $\frakm\in \Max R$ and let $(W,g)\in \tilde{W}(A(\frakm))$.
It is enough to construct  
$(V,f)\in \Herm[\veps]{A_K,\sigma_K}$
such that $\partial_\frakm[V,f]=[W,g]$ and
$\partial_\frakq[V,f]=0$ for all $\frakq\in \Max R-\{\frakm\}$.

If  $[A(\frakm)]=0$ and $(\sigma(\frakm),\veps(\frakm))$ is symplectic
(see~\ref{subsec:Azumaya}), or  
$\Cent(A(\frakm))\cong k(\frakm)\times k(\frakm)$, then $W_\veps(A(\frakm))=0$
(\cite[Example~2.4, Proposition~6.8(ii)]{First_2022_octagon},
for instance)
and we can take $f=0$.
We therefore exclude these cases until the end of the proof.
As the order of $[A(\frakm)]$
in $\Br \Cent(A(\frakm))$ divides $\ind A(\frakm)$,
this means that $\ind A(\frakm)$ is even
when $(\sigma(\frakm),\veps(\frakm))$
is symplectic. Consequently, $\deg A$ is even when $(\sigma,\veps)$
is symplectic.

We may assume that $\veps=1$. 
Indeed, if $\veps=-1$, then by 
\cite[Lemma~1.26]{First_2022_octagon},
there exists $v\in \Sym_{-1}(A,\sigma)\cap \units{A}$.
Applying $v$-conjugation (Example~\ref{EX:conjugation-is-compatible-with-GW}),
we may replace $\sigma$ and $g$
with $\Int(v)\circ \sigma$ an $vg$, and assume that $\veps=1$.
The type
of $(\sigma,\veps)$ is unchanged by this transition \cite[Corollary~1.22(i)]{First_2022_octagon}.

Now, by Lemma~\ref{LM:primitive-invariant-idempotent}, 
$(W,g)$ is the orthogonal sum of $1$-hermitian spaces
with a simple underlying module.
It is therefore enough to consider the case
where $W$ is a simple $A(\frakm)$-module.

By  
\cite[Proposition~1.27]{First_2022_octagon}, 
there exists a $\sigma$-invariant primitive idempotent
$u\in A(\frakm)$. 
We may assume that $W=u A(\frakm)$.  
Writing $a= {g}(u,u)\in u(A\frakm^{-1}  /A)u$,
we have  $ {g}(x,y)= {g}(ux,uy)=x^\sigma a y$ for all $x,y\in u A(\frakm)$.
Fixing a generator $\pi$ to $\frakm$, it is also easy to see that the image
of $a$ under $ x+A \mapsto \pi x+\frakm A: u(A\frakm^{-1}/A)u \to u (A / A \frakm)u$
must be in $\units{(u(A/ A \frakm)u)}$, otherwise $g$ would not be unimodular.

It is well-known that $u\in A/A \frakm $ can be lifted
to an idempotent $u_1\in A/ A \frakm^2$. By  \cite[Lemma~1.28]{First_2022_octagon},
$u_1$ can be chosen so that $u_1^\sigma= u_1$.
Choose some $a_1\in u_1(A\frakm^{-1} /A \frakm )u_1$
projecting onto $a$. Replacing $a_1$ with $\frac{1}{2}(a_1+a_1^\sigma)$, we
may  assume that $a_1^\sigma=a_1$.
By 
\cite[Lemma~1.26]{First_2022_octagon},
there exists 
a $\sigma$-invariant  element $b\in \units{((1-u)(A/ A \frakm)(1-u))}$.
Then $a_2:=a_1+b\in A\frakm^{-1} /A \frakm $ satisfies
$a_2^\sigma=a_2$.

For all $\frakq\in \Max R-\{\frakm\}$,
the natural map $A \frakm^{-1} / A \frakm^{-1}\frakq \to (A \frakm^{-1})_\frakq/
( A \frakm^{-1}\frakq)_\frakq=A_\frakq/\frakq A_\frakq$ is an isomorphism.
This allows us to apply the Chinese Remainder Theorem and
choose
$c\in A \frakm^{-1}$   projecting onto $a_2$ and such
that  the image of $c$ in $A_\frakq/\frakq A_\frakq$ is $1$ for all $\frakq\in \Max R-\{\frakm\}$. Replacing
$c$ with $\frac{1}{2}(c+c^\sigma)$, we may assume that $c^\sigma=c$.

We claim that $c\in \units{A_K}$,
$c^{-1}\in A_\frakm$,
$u\cdot c^{-1}(\frakm)=0$
in $A(\frakm)$
and $A_\frakm/c^{-1}A_\frakm$ is a simple $A(\frakm)$-module.
It is enough to check  this after base-changing to the completion of $R$
at $\frakm$, so \emph{we assume that $R$ is a complete  discrete valuation ring until the end of the paragraph}.
In this case, $u_1$ can be lifted to an idempotent $u_2\in A$ 
\cite[Theorem~6.18]{Reiner_2003_maximal_orders_reprint}.
Let $c'=\pi u_2+(1-u_2)\in A$.
Then, working in the $A$-bimodule $A_K/A \frakm $,
we have $cc'+ A\frakm =\pi a_1+b \in \units{A(\frakm)}$.
Thus, $cc'\in \units{A}$ and it follows that $c\in \units{A_K}$
and $c^{-1}\in c'\units{A}\subseteq A$.
Moreover, 
we have $c^{-1}A=c'A$, so $A/c^{-1}A=A/c'A\cong u A(\frakm)$, which is a simple
$A(\frakm)$-module. 
Finally, $u_2c^{-1}=u_2c'(cc')^{-1}\in \pi u_2\units{A}\subseteq  A \frakm$,
so $u \cdot c^{-1}(\frakm)=0$.

Let $ {f}:A_K\times A_K\to A_K$ be given by $ {f}(x,y)= x^\sigma cy$.
Since $c\in\units{A_K}$, we have $(A_K,f)\in \Herm[1]{A_K,\sigma_K}$.
Let $U=c^{-1} A_\frakm$. Using $c^\sigma=c$, one readily checks that
$U^f=\{v\in A_K\suchthat  {f}(U,v)\subseteq A_\frakm\}= \{v\in A_K\suchthat
 A_\frakm v\subseteq A_\frakm\}=A_\frakm$.
Since $U^f/U=A_\frakm/c^{-1}A_\frakm$ is a simple $A_\frakm$-module, $\partial_\frakm [A_K,f]$
is represented by the $1$-hermitian space 
$(A_\frakm /c^{-1}A_\frakm,\partial f)$,
where 
\[ {\partial f}(x+c^{-1} A_\frakm,y+c^{-1} A_\frakm)=x^\sigma cy+A_\frakm=
(x^\sigma+A_\frakm) a (y+A_\frakm)
.\]
Since $u\cdot c^{-1}(\frakm)=0$ and $uau=a$,
it follows that   $(x+c^{-1} A_\frakm)\mapsto u\cdot x(\frakm):A_\frakm/c^{-1} A_\frakm 
\to u A(\frakm)$
defines an isometry from $(U^f/U,\partial f)$ to $(W,g)$ (it is invertible because
it is surjective and the source and target are simple).

To finish, note that for all $\frakq\in \Max R-\{\frakm\}$,
we have $c\in\units{A_\frakq}$  because $c+\frakq A_\frakq=1+\frakq A_\frakq$.
Computing $\partial_\frakq[A_K,f]$ using  $U=c^{-1}A_\frakq=A_\frakq$,
we get $U^f=A_\frakq=U$, so $\partial_\frakq[A_K,f]=0$.  
\end{proof}

\begin{lem}\label{LM:purity-and-dd-zero}
	Assume that $R$ is a regular domain with fraction field $K$.
	Let $(V,f)\in \Herm[\veps]{A_K,\sigma_K}$ and  $\frakp\in R^{(1)}$.
	If   
	$\partial^A_{0,\frakp}[V,f]=0$, then 
	there exists $(U,g )\in \Herm[\veps]{A_\frakp,\sigma_\frakp}$
	such that $(U_K,g_K)\cong (V,f )$.
\end{lem}

\begin{proof}
	Since $\partial_{0,\frakp}[V,f]=0$,
	there exists an $A_\frakp$-lattice $U\subseteq V$
	such that $U^f\frakp\subseteq U\subseteq U^f$ and $[U^f/U,\partial f]=0$.
	By \cite[Theorem~2.8(ii)]{First_2022_octagon} (for instance),
	$\partial f$ is metabolic. Let $L$ be an
	$A$-submodule such that $L=L^\perp:=\{v\in U^f/U\suchthat \partial f(U^f/U,v)=0\}$.
	Then $L=M/U$ for some $A_\frakp$-module $M$ lying between $U$ and $U^f$.
	That $L^\perp=L$ implies readily that $M^f=M$.
	Since $R_\frakp$ is a discrete valuation ring, 
	$M\in \rproj{A_\frakp}$, so  the restriction of $f$ to $M\times M$ --- 
	call it $g$ --- is a unimodular $\veps$-hermitian
	form over $(A_\frakp,\sigma_\frakp)$. Since $(M_K,g_K)\cong (V,f)$,
	we are done.
\end{proof}

The following theorem is 
a consequence of 
a   purity theorem of Colliot-Th\'el\`ene and Sansuc
\cite[Corollaire~2.5]{Colliot_1979_quadratic_fiberations}
and Lemma~\ref{LM:purity-and-dd-zero}.
(Colloit-Th\'el\`ene and Sansuc consider only the case
	$(A,\sigma,\veps)=(R,\id_R,1)$, but their proof applies, essentially
	verbatim, to Azumaya algebras with involution; 
	use \cite[Proposition~2.14]{Saltman_1999_lectures_on_div_alg}.
	See also  \cite[\S4]{Auel_2015_quadratic_surface_bundles}  
	for relevant generalizations of  patching arguments
	from \cite[\S2]{Colliot_1979_quadratic_fiberations}.)

\begin{thm}\label{TH:purity-dim-two}
	If $\dim R\leq 2$, then $\HH^0(\aGW{A,\sigma,\veps})=0$.
\end{thm}

\section{An Octagon of Witt Groups}
\label{sec:octagon}

In  \cite[\S6]{Grenier_2005_octagon_of_Witt_grps},
Grenier and Mahmoudi
associated with a  
central simple algebra 
and some auxiliary data an $8$-periodic cochain complex ---
octagon, for short --- of Witt groups; special cases
have been observed before, e.g., \cite[Appendix]{Bayer_1995_Serre_conj_II} and \cite{Lewis_1982_improved_exact_sequences}.
This   was extended to Azumaya algebras in  \cite{First_2022_octagon}.
In this section, we recall the    octagon's construction  
and prove that it is compatible with the Gersten--Witt complex.
	
\medskip

Following \cite[\S3.1]{First_2022_octagon},
suppose that:
\begin{enumerate}[label=(G\arabic*)]
		\item \label{item:octagon:given:first}
		$(A,\sigma)$ is an Azumaya $R$-algebra with involution;
		\item $\veps\in \mu_2(R)$;
		\item \label{item:octagon:given:last}
		$\lambda,\mu\in \units{A}$ and satisfy
		$\lambda^\sigma=-\lambda$, $\mu^\sigma=-\mu$, $\lambda\mu=-\mu\lambda$,
		$\lambda^2\in \Cent(A)$.
	\end{enumerate}
	Now  define the following:
	\begin{enumerate}[label=(N\arabic*)]
	\item  $S=\Cent(A)$;
	\item  $B$ is the commutant of $\lambda$ in $A$;
	\item  $T=\Cent(B)$;
	\item  $\tau_1:=\sigma|_B$;
	\item  $\tau_2:=\Int(\mu^{-1})\circ \sigma|_B$ (i.e.\ $x^{\tau_2}=\mu^{-1}x^\sigma \mu$).
	\end{enumerate}
	It is shown in \cite[\S3.1]{First_2022_octagon}
	that $A$ is an Azumaya   $S$-algebra, $B$ is an Azumaya   $T$-algebra, $T$ is a quadratic
	\'etale $S$-algebra, $\deg A=\frac{1}{2}\deg B$ (if $\deg A$ is constant), 
	$\mu B=B\mu$
	and $A=B\oplus \mu B$.
	Using the last fact:
	\begin{enumerate}[label=(N\arabic*), start=6]
	\item  $\pi,\pi':A\to B$  are defined by
	$\pi(b_1+\mu b_2)=b_1$ and $\pi'(b_1+\mu b_2)=b_2$
	for all
	$b_1,b_2\in B$.
	\end{enumerate}
	
	Note that $\pi$ is an involution $1$-trace from $(A,\sigma)$
	to $(B,\tau_1)$ and $\pi'$ is an involution
	$(-1)$-trace from $(A,\sigma)$ to $(B,\tau_2)$; condition \ref{item:inv-trans:regularity} of \ref{subsec:Scharlau}
	is verified in the proof of \cite[Lemma~3.3]{First_2022_octagon}.
	
	Write $\iota$ for the inclusion morphism $B\to A$.
	We shall view $\iota$ as morphism of $R$-algebras
	with involution from $(B,\tau_1)$ to $(A,\sigma)$,
	or from  $(B,\tau_2)$ to $(A,\sigma_2)$, where $\sigma_2:=\Int( (\lambda \mu)^{-1})\circ \sigma$.
	Recall from  \ref{subsec:base-change} that $\iota_*$
	denotes the induced map between the respective Witt groups.
	We write
	\begin{align*}
	\rho_*&:=\iota_*\circ \lambda_*  :W_\veps(B,\tau_1)\to W_{-\veps}(A,\sigma),\\
	\rho'_*&:= (\lambda\mu)_*\circ \iota_*  :W_\veps(B,\tau_2)\to W_{-\veps}(A,\sigma),
	\end{align*}
	where $\lambda_*$ and $(\lambda\mu)_*$ denote $\lambda$-conjugation 
	and $\lambda\mu$-conjugation, respectively; see Example~\ref{EX:conjugation-is-compatible-with-GW}.

	The octagon of Witt groups associated with the data \ref{item:octagon:given:first}--\ref{item:octagon:given:last}
	is:
\begin{equation}\label{EQ:octagon}
\xymatrix{
W_\veps(A,\sigma) \ar[r]^{\pi_*} &
W_{\veps}(B,\tau_1) \ar[r]^{\rho_*} &
W_{-\veps}(A,\sigma) \ar[r]^{\pi'_*} &
W_{\veps}(B,\tau_2) \ar[d]^{\rho'_*} \\
W_{-\veps}(B,\tau_2) \ar[u]^{\rho'_*}&
W_{\veps}(A,\sigma) \ar[l]^{\pi'_*}&
W_{-\veps}(B,\tau_1) \ar[l]^{\rho_*}  &
W_{-\veps}(A,\sigma) \ar[l]^{\pi_*}
}
\end{equation}

\begin{thm}[{\cite[Theorem~3.4, Proposition~3.5]{First_2022_octagon}}]
	\label{TH:octagon-is-exact}
	The octagon \eqref{EQ:octagon} is a cochain complex.
	If $R$ is semilocal, then it is   exact.
\end{thm}

\begin{thm}\label{TH:octagon-is-compatible}
	Suppose $R$   is regular.
	Then there is an octagon of cochain complexes:
	\[
	\xymatrix{
\aGW{A/R,\sigma,\veps} \ar[r]^{\pi_*} &
\aGW{B/R,\tau_1,\veps} \ar[r]^{\rho_*} &
\aGW{A/R,\sigma,-\veps} \ar[r]^{\pi'_*} &
\aGW{B/R,\tau_2,\veps} \ar[d]^{\rho'_*} \\
\aGW{B/R,\tau_2,-\veps} \ar[u]^{\rho'_*}&
\aGW{A/R,\sigma,\veps} \ar[l]^{\pi'_*}&
\aGW{B/R,\tau_1,-\veps} \ar[l]^{\rho_*}  &
\aGW{A/R,\sigma,-\veps} \ar[l]^{\pi_*} 
}
	\]
	It is a cochain complex of cochain complexes, and its $e$-level
	is exact for all $e\geq 0$. If $R$ is semilocal, then all its levels are exact.
\end{thm}

\begin{proof}
	The existence of the octagon follows from 
	Theorem~\ref{TH:base-change-is-compatible-with-GW}, 
	Theorem~\ref{TH:pi-transfer-is-compatible}
	and Example~\ref{EX:conjugation-is-compatible-with-GW}.
	By Theorem~\ref{TH:octagon-is-exact},  the $(-1)$-level of
	the octagon is a cochain complex, which is also exact
	when $R$ is semilocal. To see that
	the $e$-level is exact for $e\geq 0$, fix isomorphisms $\tilde{k}(\frakp)\to k(\frakp)$
	for all $\frakp\in R^{(e)}$
	and use them to identify the $\frakp$-component of the $e$-level of the octagon
	with the  octagon of Witt groups associated to $A(\frakp),\sigma(\frakp),\mu(\frakp),\lambda(\frakp)$.
	The latter is exact by Theorem~\ref{TH:octagon-is-exact} (or, alternatively,
	\cite[Corollary~6.1]{Grenier_2005_octagon_of_Witt_grps}).
\end{proof}

\begin{remark}\label{RM:base-ring-for-oct}
	In general,  $(B,\tau_{i})$ ($i=1,2$)
	is not always Azumaya over $R$.
	However, it is Azumaya over $R_i:=\Cent(B )^{\{\tau_i\}}$,
	which is a finite \'etale
	$R$-algebra (see~\ref{subsec:Azumaya}),
	and 
	by  Theorem~\ref{TH:base-does-not-matter},
	we have $\aGW{B/R,\tau_i,\veps}\cong \aGW{B/R_i,\tau_i,\veps}$.
\end{remark}

Overriding previous notation,
let $S$ be a quadratic \'etale $R$-algebra, 
let $\theta$
be its standard $R$-involution (see~\ref{subsec:Azumaya})  and suppose that
there exists $\lambda\in S$ such that $\{1,\lambda\}$
is an $R$-basis of $S$
and $\lambda^2\in\units{R}$; such $\lambda$ always exists if $R$ is semilocal
\cite[Lemma~1.19]{First_2022_octagon}.
Write $\Tr=\Tr_{S/R}$. It is easy to check that $\Tr:S\to R$
is an involution $1$-trace relative to both $(\theta,\id_R)$
and $(\id_S,\id_R)$. Writing $\iota$ for the inclusion map $R\to S$,
consider the sequence
\begin{align*}\label{EQ:five-terms}
		0 \to 
		W_{1}(S,\theta)\xrightarrow{\Tr_* }
		W_{1}(R,\id )\xrightarrow{\lambda_*\iota_*}
		W_{1}(S,\id )\xrightarrow{\Tr_* }
		W_{1}(R,\id )\xrightarrow{\lambda_*\iota_*}
		W_{-1}(S,\theta) 
		\to 0 .
\end{align*}
This  is in fact a special case of the octagon \eqref{EQ:octagon};
see \cite[Corollary~8.3]{First_2022_octagon} and its proof.
Thus, 
the sequence is a cochain complex and it is exact when $R$ is semilocal.
Arguing as in the proof of Theorem~\ref{TH:octagon-is-compatible}, we get:

\begin{thm}\label{TH:five-term-is-compatible}
	In the previous notation, when $R$ is regular,
	there exists a $5$-term  cochain complex of cochain complexes
	\[
	0 \to 
	\aGW{S,\theta,1}\xrightarrow{\Tr_* }
	\aGW{R,\id ,1}\xrightarrow{\lambda_*\iota_*}
	\aGW{S,\id ,1} \xrightarrow{\Tr_* }
	\aGW{R,\id ,1} \xrightarrow{\lambda_*\iota_*}
	\aGW{S,\theta,-1} 
	\to 0 .
	\] 
	The $e$-level of the complex  
	is exact for all $e\geq 0$. If $R$ is semilocal, then all  levels are exact.
\end{thm}

\section{Springer's Theorem on Odd-Rank Extensions}
\label{sec:Springer}

We recall the Scharlau transfer, which is 
a special kind of involution $1$-trace (see \ref{subsec:Scharlau}), and establish 
a version of the weak Springer Theorem on 
odd-rank extensions for  algebras with involution.

As before, $(A,\sigma)$ denotes an $R$-algebra with involution.
Recall that an $R$-algebra
$S$ is called \emph{monogenic} if there exists $x\in S$  and $n\in\N$
	such that $\{1,x,\dots,x^{n-1}\}$
	is an $R$-basis of $S$, or equivalently, if $S\cong R[X]/(f)$
	for some monic polynomial $f$.

\medskip

	Let $S$ be a monogenic $R$-algebra of odd  rank $n$ and let
	$x\in S$ be an element such that $\{1,x,\dots,x^{n-1}\}$
	is an $R$-basis of $S$. Denote the inclusion
	$A\to A_S$  by $\rho$.
	The \emph{Scharlau transfer} is the map
	$\pi=\pi_x:A_S\to A$  defined by
	\[
	\pi(a_0+a_1x+\dots+a_{n-1}x^{n-1})=a_{0}
	\] 
	for all $a_0,\dots,a_{n-1}\in A$
	(Scharlau  \cite{Scharlau_1969_transfer} considered the case $A=R$).
	Arguing as in 
	the proof of \cite[Proposition~1.2]{Bayer_1990_odd_degree_extensions}
	or \cite[Proposition~2.1]{Fainsilber_1994_formes_hermitiennes},
	one sees that  $\pi$
	is an involution $1$-trace from $(A_S,\sigma_S)$ to $(A,\sigma)$
	in the sense of \ref{subsec:Scharlau}, and  the following 
	version of the weak Springer  Theorem holds.

	\begin{prp}\label{PR:Springer}
		In the previous notation,
		the composition
		\[
		W_\veps(A,\sigma)\xrightarrow{\rho_*} W_\veps(A_{S},\sigma_{S})
		\xrightarrow{\pi_*} W_\veps(A,\sigma)
		\]
		is the identity. In particular,
		the base change
		map $\rho_*:W_\veps(A,\sigma)\to W_\veps(A_S,\sigma_S)$
		is injective.
	\end{prp}

	We would like to have an analogue of  Proposition~\ref{PR:Springer}
	when $S$ is a finite \'etale $R$-algebra of odd rank.
	Such algebras are   not 
	monogenic in general.
	When $R$ is semilocal, we   achieve  this  
	by showing
	that every odd-rank finite \'etale $R$-algebra can be embedded
	in an odd-rank monogenic \'etale $R$-algebra.

	We begin with the following lemma, which will also be needed
	in Section~\ref{sec:odd-extension}. The standard  proof is omitted.

	\begin{lem}\label{LM:free-generation}
		Assume $R$ is semilocal. Let $M\in \rproj{R}$ and let $m_1,\dots,m_n\in M$.
		\begin{enumerate}[label=(\roman*)]
			\item 
			If $m_1(\frakm),\dots,m_n(\frakm)$ form
			a $k(\frakm)$-basis to $M(\frakm)$ for all 
			$\frakm\in \Max R$, then $m_1,\dots,m_n$
			form an $R$-basis to $M$.
			\item 
			If $m_1(\frakm),\dots,m_n(\frakm)$  
			are $k(\frakm)$-linearly independent in  $M(\frakm)$ for all 
			$\frakm\in \Max R$, then $m_1,\dots,m_n$
			are $R$-linearly independent in $M$
			and $\sum_i m_i R$ is a summand of $M$.
		\end{enumerate}
	\end{lem}

%

	\begin{prp}\label{PR:simple-extension}
		Suppose that $R$ is semilocal,
		let $S$ be a finite \'etale $R$-algebra of constant   rank
		and let $c\in\N$.
		Then there exists an odd-rank finite \'etale $R$-algebra $T$
		such that $S\otimes T$ is monogenic
		and $|(S\otimes T)/\frakm|\geq c$ for all $\frakm\in\Max(S\otimes T)$.
	\end{prp}
	
	\begin{proof}
		Suppose first that $R$ is a finite field of cardinality $q$.
		Then $S=F_1\times\dots\times F_t$, where $F_1,\dots,F_t$ are finite $R$-fields.
		Let $\ell$ be a prime number and let $T$ be the field with $q^\ell$
		elements. 
		We claim that $T$ satisfies the requirements of the lemma   for all sufficiently
		large $\ell$.		
		To see this, write $n_i:=\dim_RF_i$ ($1\leq i\leq t$).
		Observe first that if $\ell>\max\{n_1,\dots,n_t\}$,
		then $F_i\otimes T$ is an $R$-field of dimension $n_i\ell$.
		For all $\ell$ large enough, we shall have $|F_i\otimes T|\geq c$.
		Choose a monic polynomial $g_i\in R[X]$ such that $F_i\otimes T\cong R[X]/(g_i)$.
		It is well-known that the number of monic prime polynomials of degree
		$n_i\ell$ over $R$
		is $(n_i\ell)^{-1}q^{n_i\ell}+O_{\ell}(q^{n_i\ell/2})$.
		Thus, for every $\ell$ sufficiently large, the polynomials
		$g_1,\dots,g_t$ can be chosen to be distinct (even when
		$n_1,\dots,n_t$ are not distinct).
		Let $f=\prod_i g_i$. By the Chinese Remainder Theorem,
		$S\otimes T=\prod_i F_i\otimes T\cong \prod_i R[X]/(g_i)\cong R[X]/(f)$,
		so $S\otimes T$ is monogenic.

		Suppose now that $R$ is an arbitrary semilocal ring
		with maximal ideals $\frakm_1,\dots,\frakm_s$.
		For every $i\in \{1,\dots, s\}$ such that $k(\frakm_i)$ is finite,
		use the previous paragraph  to choose a prime number $\ell_i$
		and a $k(\frakm_i)$-field   $T_i$  
		of dimension $\ell_i$
		such that  $S\otimes T_i$
		is a monogenic $k(\frakm_i)$-algebra
		which cannot surject onto a field of cardinality less than $c$.
		Choose also  a monic   polynomial $h_i\in k(\frakm_i)[X]$
		such that $T_i\cong k(\frakm_i)[X]/(h_i)$.
		Enlarging $\ell_i$ if needed, we may assume that $\ell_i$ is odd and independent of $i$;
		write $\ell=\ell_i$. If $k(\frakm_i)$ is infinite for all $i$, take $\ell=1$.
		
		For every $i\in \{1,\dots, s\}$ such that $k(\frakm_i)$
		is infinite, choose an arbitrary separable
		polynomial $h_i\in k(\frakm_i)[X]$ of degree $\ell$ and set
		$T_i=k(\frakm_i)[X]/(h_i)$.
		Then $S\otimes T_i$ is \'etale over the infinite field $k(\frakm_i)$,
		hence monogenic (see \cite[\S4]{First_2017_number_of_generators},
		for instance).
		
		By the Chinese Remainder Theorem,
		there exists a   monic polynomial $h\in R[X]$ such that 
		$\deg h=\ell$ and $h(\frakm_i)=h_i$
		for all $i\in \{1,\dots, s\}$.
		Let $T=R[X]/(h)$.
		Then $T(\frakm_i)=T_i$ is \'etale
		over $k(\frakm_i)$ for all $i$, hence $T$ is a finite \'etale $R$-algebra
		of rank $\ell$.
		Furthermore, $(S\otimes T)(\frakm_i)\cong S\otimes T_i$ is monogenic for all $i$.
		Write $m=\rank_RS$ and choose $x_i\in (S\otimes T)(\frakm_i)$
		such that $\{1,x_i,\dots,x_i^{\ell m-1}\}$
		is a $k(\frakm_i)$-basis of $(S\otimes T)(\frakm_i)$.
		By the Chinese Remainder Theorem, there exists
		$x\in S\otimes T$
		such that $x(\frakm_i)=x_i$ for all $i$,
		and 
		by Lemma~\ref{LM:free-generation}(i),
		$\{1,x,\dots,x^{\ell m-1}\}$ is an $R$-basis for $S\otimes T$.
	\end{proof}

	\begin{cor}\label{CR:Springer-odd-etale}
		Let $(A,\sigma)$ be an algebra with involution over a semilocal
		ring $R$ and let $S$ be a finite \'etale $R$-algebra of odd rank.
		Then the base change map
		\[
		\rho_*:W_\veps(A,\sigma)\to W_\veps(A_S,\sigma_S)
		\]
		admits a splitting. In particular, this map is injective.
		
		If, in addition, $R$ is regular and   
		$(A,\sigma)$ is Azumaya
		over $R$,
		then the base change morphism
		$\rho_*:\aGW{A/R,\sigma,\veps}\to \aGW{A_S/R,\sigma_S,\veps}$
		of   Theorem~\ref{TH:base-change-is-compatible-with-GW}
		admits a splitting  in the category of abelian  cochain complexes.
	\end{cor}
	
	\begin{proof}
		We may assume $R$ is connected,
		otherwise write $R$ as a finite product of connected rings and work  over each
		factor separately. Then $\rank_RS$ is constant.
		By   Proposition~\ref{PR:simple-extension}, there exists an odd-rank finite \'etale
		$R$-algebra $T$  such that 
		$S\otimes T$ is monogenic.
		By Proposition~\ref{PR:Springer},
		the composition
		\[
		W_\veps(A,\sigma)\to W_\veps(A_S,\sigma_S)\to  W_\veps(A_{S\otimes T},\sigma_{S\otimes T})
		\xrightarrow{\pi_*} W_\veps(A,\sigma)
		\]
		is the indentity, hence the first part of the corollary.
		The second part   follows from
		Theorems~\ref{TH:base-change-is-compatible-with-GW}
		and~\ref{TH:pi-transfer-is-compatible}.
	\end{proof}

\section{Applying the Octagon After Odd-Rank Extensions}
\label{sec:odd-extension}

Let $(A,\sigma)$ be an Azumaya $R$-algebra with involution.
In order to apply the octagon of Section~\ref{sec:octagon}, one needs
to find elements $\lambda,\mu\in \units{A}$ satisfying condition \ref{item:octagon:given:last}. 
In this section, we shall see that after applying operations
such as conjugation (Example~\ref{EX:conjugation-is-compatible-with-GW}), 
$e$-transfer (see \ref{subsec:herm-Morita}) and tensoring with an odd-rank finite \'etale $R$-algebra, 
we can guarantee the existence of $\lambda$ and $\mu$.

Some  proofs in this section      use   Galois theory of commutative rings.
We refer the reader to \cite[Chapter 12]{Ford_2017_separable_algebras}
for the necessary definitions and an extensive discussion.

	\begin{lem}\label{LM:splitting-with-fin-et-II}
		Assume $R$ is semilocal  and let $A$ be a separable
		projective $R$-algebra with center $S$.
		Suppose that $\ell:=\rank_RS$ and $n:=\deg A$ are constant.
		Then there exists a finite \'etale $R$-algebra $T$
		such that 
		$A_T\cong  \nMat{S_T}{n}$ as $S_T$-algebras.
	\end{lem}
	
	\begin{proof}
		By  \cite[Proposition~2.18]{Saltman_1999_lectures_on_div_alg},
		there is an finite \'etale $R$-algebra $T_0$ such that $S_{T_0}\cong T_0\times \dots \times T_0$
		($\ell$ times). 
		There is a corresponding decomposition $A_{T_0}=A_1\times \dots\times A_\ell$ where
		each $A_i$ is Azumaya of degree $n$ over $T_0$. 
		By 	\cite[III.5.1.17]{Knus_1991_quadratic_hermitian_forms}
		and \cite[Theorem~7.4.4]{Ford_2017_separable_algebras},
		for each $i\in\{1,\dots,\ell\}$, there is a finite \'etale $T_0$-algebra
		$T_i$ such that $A_i\otimes_{T_0}T_i\cong \nMat{T_i}{n}$ 
		(the first source assumes that $R$ is local, but the proof
		also applies in the semilocal case by Lemma~\ref{LM:free-generation}(ii)).
		Take $T=T_1\otimes_{T_0} \cdots\otimes_{T_0} T_\ell$.
	\end{proof}

	\begin{lem} 
	\label{LM:index-bound}
		Let $A$ be an Azumaya $R$-algebra and 
		let $S$ be a finite \'etale $R$-algebra
		such that $[A_S]=0$ in $\Br S$.
		Then $\ind A\mid \rank_R S$.
	\end{lem}
	
	\begin{proof}
		By \cite[Theorem~7.4.3]{Ford_2017_separable_algebras},
		there exists $B\in [A]$ such that $\deg B=\rank_RS$,
		hence the lemma.
	\end{proof}

	\begin{lem}\label{LM:quad-etale-after-odd-rank}
		Suppose that $R$ is 
		connected  
		semilocal
		and let $S$ be a 
		connected 
		finite \'etale
		$R$-algebra of rank $2^n$ ($n\geq 1$).
		Then there is a  connected 
		odd-rank finite \'etale
		$R$-algebra $T$ such that $S_T$ 
		is connected and 
		contains a quadratic \'etale $T$-subalgebra.
	\end{lem}
	
	\begin{proof}
		The proof goes by a standard argument similar to the case of fields
		using Galois theory of fields.
		Use \cite[Theorems~12.5.4, 12.6.3]{Ford_2017_separable_algebras}
		and \cite[Lemma~1.3]{First_2022_octagon}.
	\end{proof}

	\begin{lem}\label{LM:involution-inv-max-etale-alg}
		Suppose that $R$ is semilocal 
		and let $(A,\sigma)$ be an Azumaya
		$R$-algebra with involution such that
		$\sigma$ is orthogonal or unitary. 
		Assume $n:=\deg A$ is constant and let $S=\Cent (A)$.
		Then there exists an odd-rank finite \'etale $R$-algebra 
		$T$ and $x\in\Sym_1(A_T,\sigma_T)$
		such that $S_T[x]$ is a finite \'etale $S_T$-algebra
		of rank $n$.
	\end{lem}
	
	\begin{proof}
		By Proposition~\ref{PR:simple-extension} (applied with $S=R$),
		there exists an odd-rank finite \'etale $R$-algebra
		$T$ such that $|T/\frakm|>n$ for all $\frakm\in\Max T$.
		We may replace $R,A,\sigma$ with $T,A_T,\sigma_T$ and assume
		that $|k(\frakm)|>n$ for all $\frakm\in \Max R$.
		
		Suppose first that $R$ is a field.
		By \cite[Theorem~4.1]{Becher_2018_involutions_and_stable_subalgebras},
		$A$ contains a finite \'etale $R$-subalgebra $E$ of rank $n$ fixed pointwise by $\sigma$
		(here we need $\sigma$ to be orthogonal or unitary).
		By \cite[Corollary~4.2]{First_2017_generators_of_sep_algebras_fin_fld_preprint}
		(see also \cite[Theorem~6.3]{Kravchenko_2012_number_of_generators_of_alg})
		and our assumption that $|R|>n$, the $R$-algebra $E$ is monogenic, say
		$E=R+xR+\dots+x^{n-1} R$. It is therefore enough to show that 
		$ 1,x,\dots,x^{n-1} $ are linearly independent over $S$, or equivalently,
		that $\dim_R ES=n\cdot \dim_R S$. This is clear if $S=R$. Otherwise,
		$S$ is quadratic \'etale over $R$, so there is $\lambda\in \units{S}$
		such that $\{1,\lambda\}$ is an $R$-basis of $S$ and $\lambda^{\sigma}=-\lambda$.
		Since $ E \cap \lambda E\subseteq   \Sym_{1}(A,\sigma)\cap  \Sym_{-1}(A,\sigma)=0$,
		we have $ES= E\oplus \lambda E$, hence our claim.
		
		The case where $R$ is a general semilocal ring
		can be deduced from the previous paragraph using the Chinese Remainder
		Theorem and Lemma~\ref{LM:free-generation}(ii).
	\end{proof}

	The following generalization of the Skolem-Noether theorem
	is known to experts. We include a proof for the sake of completeness.

	\begin{thm}\label{TH:Skolem-Noether}
		Suppose that $R$ is connected semilocal,
		let $A$ be an Azumaya $R$-algebra
		and let $B$ be a separable projective  subalgebra of $A$
		with connected center.
		Then any $R$-algebra homomorphism $\vphi:B\to A$
		is the restriction of an inner automorphism of $A$.
	\end{thm}
	
	\begin{proof}
		Let $S=\Cent(B)$.
		View $A$ as a right $A\otimes B^\op$-module
		by setting $a\star (a'\otimes b^\op)=baa'$
		and let $M$ denote $A$
		with the right $A\otimes B^\op$-module
		structure given by
		$a\ast (a'\otimes b^\op)=\vphi(b)aa'$.
		Since $R$ and $S$ are connected,
		$\rank_S A_{A\otimes B^\op}=\frac{\rank_R A}{\rank_R S}=
		\rank_S M$, and since $A\otimes B^\op$
		is separable   over $R$, we have $A ,M\in\rproj{A\otimes B^\op}$ (see \ref{subsec:Azumaya}).
		Thus, by \cite[Lemma~1.24]{First_2022_octagon},
		there is an $A\otimes B^\op$-module isomorphism $\psi:A\to M$.
		Let $a=\psi(1)$. Then for all $b\in B$, we have
		$\vphi(b)a =\psi(1)\ast (1\otimes b^\op)=\psi(1\star(1\otimes b^\op))=\psi(b)=\psi(1)b=ab$.
		It is easy to see that $a\in \units{A}$, so the theorem follows.
	\end{proof}
	
	\begin{thm}\label{TH:involution-ext}
		Suppose that  $R$ is connected semilocal.
		Let  $(A,\sigma)$ be an Azumaya
		$R$-algebra with involution and let $S=\Cent(A)$. 
		Let $B$ be a separable projective
		$S$-subalgebra of $A$ with connected center $T$,
		and let $\tau:B\to B$ be an involution
		such that $\tau|_S=\sigma|_S$.
		Then there
		exists $\veps\in \{\pm 1\}$ and $a\in \units{A}\cap\Sym_\veps(A,\sigma)$
		such that $\tau =\Int(a)\circ \sigma|_B$. When    $\tau|_T\neq\id_T$,
		one can take any prescribed $\veps\in \{\pm1\}$.
	\end{thm}
	
	\begin{proof}
		Suppose first that  $T=\Cent_A(B)$.
		By Theorem~\ref{TH:Skolem-Noether},
		$\sigma\circ \tau$ is the restriction of an inner
		automorphism of $A$. This implies that there is $x\in \units{A}$
		such that $\tau=\Int(x)\circ \sigma|_B$.
		Set $t=x^\sigma x^{-1}$. 
		Since $\tau$ is an involution,
		for all $b\in B$, we have
		$b=b^{\tau\tau}=(x^\sigma x^{-1})^{-1}b(x^\sigma x^{-1})$,
		so $t\in\Cent_A(B)=T$.
		Furthermore,
		$t^\tau t=xt^\sigma x^{-1} t=x(x^{-1})^\sigma xx^{-1}x^\sigma x^{-1}=1$.
		
		If $\tau|_T=\id_T$, then $t\in\mu_2(T)=\{\pm 1\}$, hence $x^\sigma\in \{\pm x\}$,
		so take $x=a$. 
		
		If $\tau|_T\neq\id_T$, choose
		some $\veps\in \{\pm 1\}$.
		Since $T$ is connected and finite \'etale over
		$R$, and since $\tau|_T\neq\id_T$, the algebra $T$
		is quadratic \'etale over $T^{\{\tau\}}$ (see~\ref{subsec:Azumaya}).
		As $T^{\{\tau\}}$ is semilocal, we may apply Hilbert's Theorem 90 to find $s\in \units{T}$
		such that $s^{-1}s^\tau=\veps t^{-1}$.
		Then $(sx)^\sigma=x^\sigma s^\sigma=x^\sigma x^{-1} s^\tau x=
		ts^\tau x=\veps s x $.
		Since $\Int(sx)\circ\sigma$ agrees with $\Int(x)\circ \sigma$
		on $B$, we can take $a=sx$.
		
		Now assume $B$ is arbitrary, and let
		$B'=\Cent_A(T)$ and $C=\Cent_{B'}(B)$. Then $B\otimes_T C\cong B'$
		via $b\otimes c\mapsto bc$
		and $[B']=[A\otimes_S T]$ in $\Br T$,
		hence $[C]=[A\otimes_ST]-[B]$.
		Note that both
		$A\otimes_ST$
		and $B$ carry involutions
		restricting to $\tau|_T$ on the center
		(for $A\otimes_ST$, take $\sigma\otimes_S (\tau|_T)$).	
		Thus, theorems of Saltman \cite[Theorems~3.1b, 4.4b]{Saltman_1978_Azumaya_algebras_w_involution}
		imply that $C$ also admits an involution $\theta$ with
		$\theta|_T=\tau|_T$.
		Now apply the previous paragraphs to $B'\cong B\otimes_T C$
		and the involution $\tau\otimes_T\theta$.
	\end{proof}

	We are now ready to prove that the octagon
	of Section~\ref{sec:octagon} can be applied after an odd-degree extension.

	\begin{thm}\label{TH:reduction-for-octagon}
		Suppose that $R$ is    a   regular semilocal domain, let $(A,\sigma)$
		be an Azumaya $R$-algebra with involution  
		and let $\veps\in \{\pm 1\}$.
		Then there exist a connected odd-rank finite \'etale $R$-algebra
		$R_1$,  an Azumaya $R_1$-algebra
		with involution $(A_1,\sigma_1)$
		and $\veps_1\in\{\pm 1\}$
		such that 
		\begin{enumerate}[label=(\roman*)]
			\item $[A_{R_1}]=[A_1]$ in $\Br R_1$
			and $(\sigma,\veps)$
			has the same type as $(\sigma_1,\veps_1)$
			(see \ref{subsec:Azumaya});
			\item $\aGW{A,\sigma,\veps}$
			is isomorphic to a summand
			of $\aGW{A_1,\sigma_1,\veps_1}$;
		\end{enumerate}
		and at least
		one of the following hold:
		\begin{enumerate}[label=(iii-\arabic*)]
			\item \label{item:TH:oct-reduction:S-not-connected} $\Cent(A_1)\cong R_1\times R_1$;
			\item \label{item:TH:oct-reduction:deg-one} $\deg A_1=1$;
			\item \label{item:TH:oct-reduction:lambda-mu} $\ind A_1=\deg A_1$,
			$\deg A_1$ is a power of $2$
			dividing $\ind A$ 
			and there exist $\lambda,\mu\in \units{A_1}$
			such that $\lambda^2\in\units{R_1}$, $\lambda^{\sigma_1}=-\lambda$,
			$\mu^{\sigma_1}=-\mu$
			and $\lambda\mu=-\mu\lambda$.
		\end{enumerate}
	\end{thm}

	\begin{proof}
		Write $S=\Cent(A)$, $\ell=\rank_RS\in\{1,2\}$ and $n=\deg A$.
		
		We may assume throughout that $\sigma$ is orthogonal or unitary.
		Indeed, if $\sigma$ is symplectic, choose $u\in \Sym_{-1}(A,\sigma)\cap \units{A}$
		(use \cite[Lemma~1.26]{First_2022_octagon}) and replace $ \sigma,\veps $
		with $\Int(u)\circ\sigma,-\veps$. This does not
		affect the isomorphism class of $\aGW{A,\sigma,\veps}$ 
		(Example~\ref{EX:conjugation-is-compatible-with-GW}) or
		the type of $(\sigma,\veps)$ \cite[Corollary~1.22(i)]{First_2022_octagon}.		
	
		We   prove the theorem by induction on $n=\deg A$.
		The case $n=1$ is clear, so assume that $n>1$ and the theorem
		holds for Azumaya algebras of  degree smaller than $n$.
		Note if $T$ is an odd-rank
		connected finite \'etale $R$-algebra, then  Corollary~\ref{CR:Springer-odd-etale}
		allows us to replace $R,A,\sigma$
		with $T,A_T,\sigma_T$.
		
\medskip

		\Step{1} 
		We first show that the induction hypothesis implies
		the theorem  
		if at least one of the following conditions \emph{fail}:
		(1) $S$ is connected, 
		(2) $\ind A=\deg A$, 
		(3) $A$ contains no nontrivial
		idempotents.

		Indeed, if $S$ is not connected, then \ref{item:TH:oct-reduction:S-not-connected}
		holds for $A_1=A$ by \cite[Lemma~1.16]{First_2022_octagon}.

		Next, if   $\ind A<\deg A$,
		then
		by \cite[Theorem~1.30]{First_2022_octagon},
		there exists a full idempotent $e\in A $
		such that $e^\sigma=e$ and $\deg eA e=\ind A $
		(here we need $\sigma$ to be non-symplectic).
		By Theorem~\ref{TH:e-transfer-is-compatible-GW}, $\aGW{A,\sigma,\veps}\cong \aGW{eAe,\sigma_e, \veps}$,
		and since $\deg eAe=\ind A\mid \deg A$,
		we may apply the induction hypothesis to $(eAe,\sigma_e,\veps)$ and finish.
		The type of $(\sigma,\veps)$ remains unchanged
		by \cite[Corollary~1.22(ii)]{First_2022_octagon}.
		
		If $S$ is connected and $A$ contains a nontrivial idempotent
		$e$, then $eAe$ is an Azumaya $S$-algebra with $[eAe]=[A]$
		and $\deg eAe<\deg A$ \cite[Corollary~1.12]{First_2022_octagon}, so $\ind A<\deg A$ and we can proceed
		as in the previous paragraph.
		
\medskip

		\Step{2} We claim that the 	theorem holds if $\deg A$ is not  a power of $2$.
		
		Indeed,
		by Lemma~\ref{LM:splitting-with-fin-et-II}, 
		there exists a finite \'etale $R$-algebra $T$ such that
		$A_{T}\cong \nMat{S_{T} }{n}$ as $S_{T}$-algebras.
		By \cite[Theorem~12.6.1]{Ford_2017_separable_algebras},
		there exists a finite group 
		$G$ such that $T$ can
		be embedded in a $G$-Galois $R$-algebra.
		Replace $T$ with this $G$-Galois algebra.
		Let $P$ be a $2$-Sylow subgroup of $G$ and write
		$E:=T^P$.
		Then $E$ is   a finite \'etale
		$R$-algebra of rank $|G/P|$, which is odd.
		Furthermore, $T$ is a $P$-Galois $E$-algebra
		such that $[(A_{E})\otimes_{E} T]=[A_T]=0$.
		By Lemma~\ref{LM:index-bound}, 
		$\ind A_{E}\mid \rank_{E} T=|P|$,
		so $\ind A_{E}$ is a power of $2$.
		
		Write $E$ as a product of connected finite \'etale $R$-algebras.
		At least one of these algebras has odd $R$-rank.
		Replacing $E$ with that algebra, we may assume that
		$E$ is connected.
		As explained above, 
		we may now replace replace $R,A,\sigma$ with $E,A_{E},\sigma_{E}$
		to assume that $\ind A$ is a power of $2$.
		Since we assumed that $\deg A$ is not a power
		of $2$, we have $\ind A <\deg A $ and the theorem holds
		by Step~1.

\medskip

		\Step{3}
		By Lemma~\ref{LM:involution-inv-max-etale-alg},
		there exists an odd-rank finite \'etale $R$-algebra
		$T$ and $x\in \Sym_1(A_T,\sigma_T)$
		such that $L:=S_T[x]$
		is a finite \'etale $S_T$-algebra
		satisfying  $\rank_{S_T} L=\deg A_T$.
		If $T$ is not connected, express it as a product
		of connected  $R$-algebras  and replace $T$
		with one of the odd-rank factors $T_1$ and $x$ with its image in $A_{T_1}$.
		We may replace $R,A,\sigma$ with $T,A_T,\sigma_T$. 
		By Steps~1 and~2, we may  assume that 
		$\deg A$ is  a  power of $2$ greater than $1$
		and $A$ contains no nontrivial idempotents.
		In particular, all commutative $R$-subalgebras of $A$ are connected.

		Let $M=L^{\{\sigma\}}$. 
		If $S=R$, then $M=L$.
		Otherwise, $\sigma|_L\neq \id_L$,
		hence  $L$ is quadratic \'etale over $M $
		and $M$ is finite \'etale over $R$
		(see~\ref{subsec:Azumaya}).
		In any case, $\rank_ML=\rank_RS$.
		Since $M$ and $S$ are connected,  
		\[
		\rank_R M\cdot \rank_ML=\rank_R L=\rank_R S\cdot \rank_SL = \rank_R S \cdot \deg A,
		\]
		so $\rank_R M=\deg A$ is a  power of $2$ greater than $1$.

		By Lemma~\ref{LM:quad-etale-after-odd-rank},
		there exists a connected  odd-rank finite \'etale
		$R$-algebra $E$
		such that $M_{E}$ contains a quadratic \'etale $E$-algebra
		$Q$. We may replace $R,A,\sigma$ with $E,A_E,\sigma_E$
		and, thanks to Steps~1 and~2, continue to assume that $\deg A$
		is an even power of $2$ and $A$ contains no nontrivial idempotents.

		We claim that the map $q\otimes s\mapsto qs: Q\otimes S\to Q\cdot S$
		is an isomorphism. This is immediate if $S=R$, so assume that
		$S$ is quadratic \'etale over $R$.
		Note first that $QS$ is an epimorphic image of $Q\otimes S$, hence
		separable over $R$ \cite[Proposition~4.3.6]{Ford_2017_separable_algebras}.
		By \cite[Lemma~1.3]{First_2022_octagon} (applied with $A=L$),
		$QS$ is also projective over $R$.
		This means that $\ker (Q\otimes S\to QS)$ is a projective $R$-module
		of rank $\rank_R Q\otimes S-\rank_R QS$, so we need to show that $\rank_R Q\otimes S=\rank_R Q  S$.
		Clearly, $\rank_R QS\leq \rank_R Q\otimes S=4$.
		On the other hand, $QS\neq S$ because $Q\cap S\subseteq \Sym_1(S,\sigma)=R$,
		so $\rank_R QS=\rank_S QS\cdot \rank_R S\geq 2\cdot 2=4$,
		forcing   $\rank_R QS=4=\rank_R Q\otimes S$.
			
		We identify $Q\otimes S$ with $QS$ henceforth.
		
\medskip

		\Step{4} By \cite[Lemma~1.19]{First_2022_octagon},
		there exists $\lambda\in Q$ such that $Q=R\oplus  \lambda R	$
		and $\lambda^2\in \units{R}$. We have $\lambda^\sigma=\lambda$ because $Q\subseteq M$.
		Let $\theta$ denote the standard   $R$-involution
		of $Q$ and let $\tau:=\theta\otimes (\sigma|_S):QS\to QS$.
		Then $\tau $ is an involution of $QS$ agreeing
		with $\sigma$ on $S$ and satisfying $\lambda^\tau=-\lambda$.

		By Theorem~\ref{TH:involution-ext},
		there exists $\mu\in \Sym_{-1}(A,\sigma)\cap\units{A}$
		such that $\Int(\mu)\circ \sigma|_{QS}=\tau$.
		Let $\sigma_1=\Int(\mu)\circ \sigma$
		and $\veps_1=-\veps$.
		Then 
		\begin{align*}
		\lambda^{\sigma_1}&=\lambda^\tau=-\lambda,\\
		\mu^{\sigma_1}&=\mu\mu^\sigma \mu^{-1}=-\mu,\\
		\mu\lambda\mu^{-1}&=\mu\lambda^\sigma \mu^{-1}=\lambda^{\sigma_1}=-\lambda,
		\end{align*}
		and we have established \ref{item:TH:oct-reduction:lambda-mu} with $A_1=A$.
		By Example~\ref{EX:conjugation-is-compatible-with-GW},
		$\aGW{A,\sigma,\veps}\cong \aGW{A,\sigma_1,\veps_1}$,
		and the type of $(\sigma_1,\veps_1)$
		is the same as the type of $(\sigma,\veps)$ by \cite[Corollary~1.22(i)]{First_2022_octagon},
		so (i) and (ii) also hold.
	\end{proof}
	
	\begin{remark}
	Without condition (ii),   Theorem~\ref{TH:reduction-for-octagon} holds
	under the milder assumption that $R$ is connected semilocal.
	\end{remark}

\section[The Grothendieck--Serre Conjecture]{The Grothendieck--Serre Conjecture and Exactness of The Gersten--Witt Complex in Dimension $2$}
\label{sec:exactness}

Let
$R$ denote   a regular ring, let $(A,\sigma)$ be an Azumaya $R$-algebra with involution and let
$\veps\in \mu_2(R)$. 
In this section, we  
put the machinery of the previous sections 
to exhibit
new cases where $\aGW{A,\sigma,\veps}$ is exact, and as
a consequence,   verify some open cases of the Grothendieck--Serre conjecture.
We achieve this by appealing to a theorem of Balmer, Preeti and Walter:

\begin{thm}[Balmer, Preeti, Walter]
	\label{TH:Balmer-Preeti-Walter}
	$\aGW{R,\id_R,\veps}$ is exact when $R$ is   regular semilocal  of dimension $\leq 4$.
\end{thm}

\begin{proof}
	We first note that by 
	Proposition~\ref{PR:comparison-with-BW},
	the   complex  $\aGW{R,\id_R,1}$   is isomoprhic to the Gersten--Witt
	complex of $R$ defined in \cite{Balmer_2002_Gersten_Witt_complex}.

	The case $\veps=1$
	was verified by Balmer and Walter
	when $R$ is local, \cite[Corollary~10.4]{Balmer_2002_Gersten_Witt_complex},
	and  Balmer and Preeti 
	\cite[p.~3]{Balmer_2005_shifted_Witt_groups_semilocal}
	showed that the   assumption on $R$ can be relaxed to $R$ being semilocal.

	The case $\veps=-1$ is vacuous because $\aGW{R,\id_R,-1}$ is the zero complex. 
\end{proof}

The fact that Theorem~\ref{TH:Balmer-Preeti-Walter} applies only in dimension $\leq 4$ is the reason why  our
results require a similar assumption on $\dim R$ --- extending it to higher
dimensional rings will result in similar improvements to some of our main results.
The  precise formulation of this principle is the content of  Theorems~\ref{TH:cond-exactness-odd-ind}
and~\ref{TH:cond-exactness-even-cohomologies}.

\begin{lem}\label{LM:grid-lemma}
	Consider a double cochain complex  $A_{\bullet,\bullet}$ of abelian groups 
	(partially illustrated below).
	\[
	\xymatrix{
	A_{-2,-1} \ar[r] \ar[d] &
	A_{-1,-1} \ar[r] \ar[d] &
	A_{0,-1} \ar[r] \ar[d] &
	A_{1,-1} \ar[r] \ar[d] &
	\boxed{A_{2,-1}}  \ar[d] \\
	A_{-2,0} \ar[r] \ar[d] &
	\boxed{A_{-1,0}} \ar[r] \ar[d] &
	A_{0,0} \ar[r] \ar[d] &
	\boxed{A_{1,0}} \ar[r] \ar[d] &
	A_{2,0}  \ar[d]  \\
	\boxed{A_{-2,1}} \ar[r]   &
	A_{-1,1} \ar[r]  &
	A_{0,1} \ar[r]  &
	A_{1,1} \ar[r]  &
	A_{2,1}    
	}
	\]
	Suppose that there exist  $s,t\in\N$
	such that:
	\begin{enumerate} 
		\item $A_{ t,-t}=0$ and $A_{-s, s}=0$;
		\item the rows   are exact at
		$A_{i,-i}$ for $-s< i< t$;
		\item the columns  are exact
		at $A_{1,0},A_{2,-1},\dots,A_{t,1-t }$
		and $A_{-1,0}, A_{-2,1},\dots,A_{-s,s-1} $
		(these places are indicated by boxes in the illustration).
	\end{enumerate}
	Then the $0$-column is exact at $A_{0,0}$.
\end{lem}

\newcommand{\hd}{h}  
\newcommand{\vd}{v}  

\begin{proof}
	The proof is by diagram chasing.
	Throughout, subscripts of elements
	indicate the row in which they live.
	We write $\hd$ for the horizontal maps in
	the diagram and $\vd$ for the vertical maps.

	Let $a_0\in A_{0,0}$ be an element such that $\vd a_0=0$.	
	Define elements $a_{-n}\in A_{n,-n}$ for $n\in\{1,\dots,t\}$
	satisfying 
	\[\vd a_{-n}=\hd a_{1-n}
	.\]
	as follows: Assuming $a_{-n}$ has been defined,
	we have $\vd\hd a_{-n}=\hd\hd a_{1-n}=0$ if $n>0$
	and $\vd\hd a_{-n}=\hd\vd a_{-n}=0$ if $n=0$.
	Use the
	exactness of the columns at $A_{ n+1,-n}$
	to choose $a_{-n-1}\in A_{n+1,-n-1}$
	such that $\vd a_{-n-1}=\hd a_{-n}$.
	
	Since $A_{t,-t}=0$, we must have $a_{-t}=0$.
	Set $b_{-t}:=0\in A_{t-1,-t}$ and $b_{-t-1}:=0\in A_{t,-t-1}$.
	For $n\in \{t-1,\dots,0\}$,  
	define elements $b_{-n}\in A_{ n-1,-n}$ satisfying
	\[
	\hd b_{-n}=a_{-n}-\vd b_{-n-1}
	\]
	inductively as follows:
	Assuming  $b_{-n}$ has been defined,
	we have $\hd(a_{1-n}-\vd b_{-n})=\hd a_{1-n}-\vd \hd b_{-n}=
	\hd a_{1-n}-\vd (a_{-n}-\vd b_{-n-1})=\hd a_{1-n}-\vd a_{-n}=0$.
	By the exactness of the rows
	at $A_{1-n,1-n}$, there exists $b_{1-n}\in A_{n-2,1-n}$
	such that $\hd b_{1-n}=a_{1-n}-\vd b_{-n}$.

	Write $c_0:=b_0$ and observe that $\vd \hd c_0=\vd a_0-\vd\vd b_{-1}=0$.
	For $n\in\{1,\dots,s\}$, we define elements $c_n\in A_{-n-1,n}$
	satisfying 
	\[\hd c_n=\vd c_{n-1}\]
	by induction. Assuming $c_{n-1}$ has been
	defined, we have 
	$\hd\vd c_{n-1}=\vd\vd c_{n-2}=0$ if $n>1$
	and $\hd\vd c_{n-1} =\vd\hd  c_0=0$ if $n=1$.
	By the exactness of the rows at $A_{-n,n}$,
	there exists $c_n\in A_{-n-1,n}$
	such that $\hd c_n=\vd c_{n-1}$.
	
	Since $A_{ -s,s}=0$, we have $c_{s}=0$.
	Let $d_{s}:=0\in A_{-s-2,s}$
	and $d_{s-1}:=0\in A_{-s-1,s-1}$.
	For $n\in \{ s-2,\dots,-1\}$,
	define  $d_n\in A_{-n-2,n}$ satisfying
	\[
	\vd d_n= c_{n+1}-\hd d_{n+1} 
	\]
	as follows: Assuming $d_{n+1}$ has
	been defined,
	we have $\vd (c_{n+1}-\hd d_{n+1})=
	\vd c_{n+1}-\hd \vd d_{n+1}=
	\vd c_{n+1}-\hd (c_{n+2}-\hd d_{n+2})=\vd c_{n+1}-\hd c_{n+2}=0$.
	Thus, by the exactness of the columns at 
	$A_{-n-2,n+1}$, there exists $d_{n}\in A_{-n-2,n}$
	such that $\vd d_n=c_{n+1}-\hd d_{n+1}$.
	
	Finally, note that $\vd d_{-1}=c_0-\hd d_0$,
	and $\hd c_0=\hd b_0=a_0-\vd b_{-1}$.
	Thus, $\vd (b_{-1}+\hd d_{-1})=
	\vd b_{-1}+\hd \vd d_{-1}=
	\vd b_{-1}+\hd(c_0-\hd d_0)=\vd b_{-1}+ a_0-\vd b_{-1}=a_0$,
	which is what we want.
\end{proof}

\begin{thm}\label{TH:cond-exactness-odd-ind}
	Assume $R$ is   regular semilocal 
	and $\ind A$ is odd.
	If $\aGW{R_1,\id,1}$ is exact for every
	finite  \'etale $R$-algebra $R_1$, then $\aGW{A,\sigma,\veps}$
	is exact.
\end{thm}

\begin{proof}
	By writing $R$ as a product of connected rings and working over each
	factor separately, we may
	assume that $R$ is a   domain.
	Write $S=\Cent(A)$.
	By Theorem~\ref{TH:reduction-for-octagon},
	we may assume that $S=R\times R$ or $\deg A=1$, i.e.\ $A=S$.
	In the former case $\aGW{A,\sigma,\veps}=0$ (Remark~\ref{RM:split-primes-do-not-contribute}) 
	and there is nothing to prove,
	so
	assume   $A=S$. If $S=R$, then we are done by assumption.
	It remains to consider the case where	$S$ is a quadratic \'etale $R$-algebra and $\sigma$ is
	its standard $R$-involution.
	By Theorem~\ref{TH:five-term-is-compatible}, we have an
	exact sequence of cochain
	complexes
	\[
	0 \to 
	\aGW{S,\sigma,1}\to
	\aGW{R,\id ,1}\to
	\aGW{S,\id ,1} \to
	\aGW{R,\id ,1} \to
	\aGW{S,\sigma,-1} 
	\to 0 
	\] 
	which we view as a double cochain complex.
	By assumption, $\aGW{R,\id ,1}$ and
	$\aGW{S,\id ,1}$ are exact.
	Furthermore, by \cite[Lemma~1.26]{First_2022_octagon},
	there exists $u\in \Sym_{-1}(S,\sigma)\cap\units{S}$,
	which induces an isomorphism  $u_*:\aGW{S,\sigma,1}\cong \aGW{S,\sigma,-1}$
	by Example~\ref{EX:conjugation-is-compatible-with-GW}.
	
	We   now use induction on $i\in\Z$
	to show that $\HH^i(\aGW{S,\sigma,1})\cong \HH^i(\aGW{S,\sigma,-1})=0$.
	Assuming  $\HH^i(\aGW{S,\sigma,1})\cong \HH^i(\aGW{S,\sigma,-1})=0$
	has been established, we get $\HH^{i+1}(\aGW{S,\sigma,1})=0$ by applying 
	Lemma~\ref{LM:grid-lemma}.
\end{proof}

\begin{thm}\label{TH:exactness-odd-index}
	If $R$ is regular semilocal   of dimension $\leq 4$ and $\ind A$ is odd,
	then $\aGW{A,\sigma,\veps}$ is exact.
\end{thm}

\begin{proof}
	This follows from Theorems~\ref{TH:cond-exactness-odd-ind} and~\ref{TH:Balmer-Preeti-Walter}.
\end{proof}

\begin{thm}\label{TH:cond-exactness-even-cohomologies}
	Assume $R$ is a  regular 
	semilocal domain.
	If
	for every connected finite \'etale $R$-algebra $R_1$,
	every Azumaya $R_1$-algebra with involution
	$(B,\tau)$ 	
	with  $\deg B\mid  \deg A$,
	and every $i\in  \Z$, 
	we have
	$\HH^{2i}(\aGW{B,\tau,\pm 1})=0$ and $\HH^i(\aGW{R_1,\id,1})=0$,
	then $ \aGW{A,\sigma,\veps} $ is exact.
\end{thm}

\begin{proof}
	We prove the theorem by induction on $\deg A$.
	The case $\deg A=1$ holds by assumption,
	so assume that $\deg A>1$ and the theorem holds for all
	Azumaya algebras with involution of degree smaller than $\deg A$.

	By Theorem~\ref{TH:reduction-for-octagon} and the induction hypothesis,
	we may assume that   $Z(A)=R\times R$, or $\ind A=\deg A$ is a power of
	$2$ and there exist $\lambda$ and $\mu$ as in Section~\ref{sec:octagon}.
	In the first case, we have $\aGW{A,\sigma,\veps}=0$ (Remark~\ref{RM:split-primes-do-not-contribute}),
	so we only need to treat the second case.
	
	Define $B,\tau_1,\tau_2$ as in Section~\ref{sec:octagon}.
	By Theorem~\ref{TH:octagon-is-compatible}, we have an exact $8$-periodic
	sequence of cochain complexes
	\[
	\cdots
	\to
	\aGW{B,\tau_2,-\veps}\xrightarrow{\rho'_*} 
	\aGW{A,\sigma,\veps} \xrightarrow{\pi_*}
	\aGW{B,\tau_1,\veps} \xrightarrow{\rho_*}
	\aGW{A,\sigma,-\veps} \xrightarrow{\pi'_*}
	\aGW{B,\tau_2,\veps} \xrightarrow{\rho'_*}
	\cdots
	\]
	which we view as a double cochain complex. 
	The columns
	$\aGW{B,\tau_1,\pm\veps}$ and $\aGW{B,\tau_2,\pm \veps}$
	are exact by the induction hypothesis 
	($\deg B=\frac{1}{2}\deg A$, Remark~\ref{RM:base-ring-for-oct}), 
	and by assumption, $\HH^{2i}(\aGW{A,\sigma,\pm \veps})=0$ for all $i\in \Z$.
	Now, by Lemma~\ref{LM:grid-lemma}, we also have $\HH^{2i+1}(\aGW{A,\sigma,\veps})=0$
	for all $i\in \Z$, so $\aGW{A,\sigma,\veps}$ is exact.
\end{proof}

\begin{thm}\label{TH:GW-exact-dim-2}
	If $R$ is regular semilocal   of dimension $\leq 2$,
	then $\aGW{A,\sigma,\veps}$ is exact.
\end{thm}

\begin{proof}
	As in the proof of Theorem~\ref{TH:cond-exactness-odd-ind},
	we may assume $R$ is a domain.
	Let $R_1$ be a finite \'etale $R$-algebra
	and let $(B,\tau)$ be an Azumaya $R_1$-algebra with ivolution.
	Then $\aGW{R_1,\id,1}$ is exact by
	Theorem~\ref{TH:Balmer-Preeti-Walter}
	and $\HH^0(\aGW{B,\tau,\pm1})=\HH^2(\aGW{B,\tau,\pm1})=0$
	by Theorems~\ref{TH:surj-at-last-term} and~\ref{TH:purity-dim-two}.
	The corollary therefore follows from
	Theorem~\ref{TH:cond-exactness-even-cohomologies}.
\end{proof}

We use the previous theorems   
to establish new cases of the Gro\-then\-dieck--Serre conjecture 
(see the introduction)
and prove a purity result for Witt groups of hermitian forms.
Given an Azumaya $R$-algebra with involution $(A,\sigma)$,
recall that $\uU(A,\sigma)\to \Spec R$ denotes the group $R$-scheme
of $\sigma$-unitary elements in $A$ and   $\uU^0(A,\sigma)\to \Spec R$
is its neutral connected component (see~\cite[\S2.5]{First_2022_octagon}).

\begin{thm}\label{TH:GS-and-purity}
	Let $R$ be a regular semilocal domain with fraction field $F$,
	let   $(A,\sigma)$ be an Azumya $R$-algebra with involution,
	and assume  that   one of the following hold:
	\begin{enumerate}[label=(\arabic*)]
		\item $\dim R=2$;
		\item $\dim R\leq 4$ and $\ind A$ is odd. 
	\end{enumerate}
	Then:
	\begin{enumerate}[label=(\roman*)]
		\item  The restriction map $\HH^1_{\et}(R,\uU(A,\sigma))\to\HH^1_{\et}(F,\uU(A,\sigma))$ 
		has trivial kernel. Likewise for $\uU^0(A,\sigma)$.
		\item $\im \big(W_\veps(A,\sigma)\to W_\veps(A_F,\sigma_F)\big)=
		\bigcap_{\frakp\in R^{(1)}}
		\im  \big(W_\veps(A_\frakp,\sigma_\frakp)\to W_\veps(A_F,\sigma_F)\big)$.
	\end{enumerate}
\end{thm}

\begin{proof}
	Part (i)   follows from  $\HH^{-1}(\aGW{A,\sigma,1})=0$
	and \cite[Proposition~8.7]{First_2022_octagon}. Part
	(ii) follows from $\HH^0(\aGW{A,\sigma,\veps})=0$
	and Lemma~\ref{LM:purity-and-dd-zero}.
\end{proof}

\bibliographystyle{plain}
\bibliography{gs_bib}

\end{document}